\documentclass[11pt]{amsart}
\usepackage[top=1in, bottom=1in, left=1in, right=1in]{geometry}
\expandafter\let\csname ver@amsthm.sty\endcsname\relax

\usepackage{tikz}

\usepackage{amsmath}
\usepackage{amssymb}
\usepackage{mathdots}
\usepackage{mathtools}

\usepackage{dsfont}
\usepackage{chngpage}

\usepackage{hyperref}
\usepackage{amsthm}
\usepackage[capitalize,noabbrev]{cleveref}

\allowdisplaybreaks

\numberwithin{equation}{section}

\newtheorem{thm}{Theorem}[section]
\newtheorem{lemma}[thm]{Lemma}
\newtheorem{cor}[thm]{Corollary}
\newtheorem{prop}[thm]{Proposition}

\newtheorem{Example}[thm]{Example}
\newenvironment{example}
  {\begin{Example}\rm}{\hfill$\lozenge$\end{Example}}

\newtheorem{Remark}[thm]{Remark}
\newenvironment{remark}
  {\begin{Remark}\rm}{\hfill$\lozenge$\end{Remark}}
  
\crefname{thm}{Theorem}{Theorems}
\crefname{lemma}{Lemma}{Lemmas}
\crefname{cor}{Corollary}{Corollaries}
\crefname{prop}{Proposition}{Propositions}

\crefname{example}{Example}{Examples}
\crefname{remark}{Remark}{Remarks}

\newcommand{\emailhref}[1]{\email{\href{#1}{#1}}}

\newcommand{\dfn}[1]{\textcolor{blue}{\emph{#1}}}

\newcommand{\RR}{\mathbb{R}}
\newcommand{\RRgt}{\RR_{>0}}
\newcommand{\Zgt}{{\mathbb{Z}}_{> 0}}
\newcommand{\Zge}{{\mathbb{Z}}_{\ge 0}}
\newcommand{\quadrant}{\Zgt^2}

\newcommand{\J}{\mathcal{J}} 
\newcommand{\A}{\mathcal{A}}

\DeclareMathOperator{\vspan}{Span}
\newcommand{\Span}{\vspan_{\RR}}

\DeclareMathOperator{\rk}{rk}

\DeclareMathOperator{\const}{const}

\newcommand{\tog}[1]{\tau_{#1}}
\newcommand{\rktog}[1]{\boldsymbol{\tau}_{#1}}
\DeclareMathOperator{\rowm}{Row}
\DeclareMathOperator{\pan}{Row}

\newcommand{\togPL}[1]{\tau_{#1}^{\mathrm{PL}}}
\newcommand{\rktogPL}[1]{\boldsymbol{\tau}_{#1}^{\mathrm{PL}}}
\newcommand{\rowmPL}{\rowm^{\mathrm{PL}}}

\newcommand{\togB}[1]{\tau_{#1}^{\mathrm{B}}}
\newcommand{\rktogB}[1]{\boldsymbol{\tau}_{#1}^{\mathrm{B}}}
\newcommand{\rowmB}{\rowm^{\mathrm{B}}} 

\newcommand{\Tin}[1]{\mathcal{T}_{#1}^+}
\newcommand{\Tout}[1]{\mathcal{T}_{#1}^-}
\newcommand{\T}[1]{\mathcal{T}_{#1}}
\newcommand{\oii}[1]{\mathds{1}_{#1}}

\newcommand{\rook}[1]{R_{#1}}
\newcommand{\rrook}[1]{\widetilde{R}_{#1}}
\newcommand{\varrook}[1]{R'_{#1}}
\newcommand{\rvarrook}[1]{\widetilde{R}'_{#1}}

\newcommand{\hatz}{\widehat{0}}
\newcommand{\hato}{\widehat{1}}

\newcommand{\aPL}{\alpha^{\mathrm{PL}}}
\newcommand{\oPL}{\omega^{\mathrm{PL}}}
\newcommand{\TinPL}[1]{\mathcal{T}_{#1}^{+,\mathrm{PL}}}
\newcommand{\ToutPL}[1]{\mathcal{T}_{#1}^{-,\mathrm{PL}}}
\newcommand{\TPL}[1]{\mathcal{T}_{#1}^{\mathrm{PL}}}
\newcommand{\oiiPL}[1]{\mathds{1}_{#1}^{\mathrm{PL}}}

\newcommand{\aB}{\alpha^{\mathrm{B}}}
\newcommand{\oB}{\omega^{\mathrm{B}}}
\newcommand{\TinB}[1]{\mathcal{T}_{#1}^{+,\mathrm{B}}}
\newcommand{\ToutB}[1]{\mathcal{T}_{#1}^{-,\mathrm{B}}}
\newcommand{\TB}[1]{\mathcal{T}_{#1}^{\mathrm{B}}}
\newcommand{\oiiB}[1]{\mathds{1}_{#1}^{\mathrm{B}}}

\DeclareRobustCommand{\qbinom}{\genfrac{\lbrack}{\rbrack}{0pt}{}}

\newcommand{\qT}[1]{\mathcal{T}^{q}_{#1}}

\newcommand{\Fz}{\mathbf{F}_0}
\newcommand{\Fo}{\mathbf{F}_1}

\newcommand{\tequiv}{\equiv}
\newcommand{\tequivPL}{\tequiv^{\mathrm{PL}}}
\newcommand{\tequivB}{\tequiv^{\mathrm{B}}}
\newcommand{\qtequiv}{\tequiv^{q}}

\newcommand{\stat}[1]{\mathsf{#1}}
\newcommand{\statPL}[1]{\mathsf{#1}^{\mathrm{PL}}}
\newcommand{\statB}[1]{\mathsf{#1}^{\mathrm{B}}}

\newcommand{\rect}[2]{[#1] \! \times \! [#2]}
\newcommand{\twogrp}{\mathfrak{S}_2}
\newcommand{\sstair}[1]{(\rect{#1}{#1})/\twogrp}

\newcommand{\arootp}[1]{\Phi^+(A_{#1})}
\newcommand{\brootp}[1]{\Phi^+(B_{#1})}

\newcommand{\dtd}[1]{D(#1)}
\newcommand{\esixmin}{\Lambda_{E_6}}
\newcommand{\esevmin}{\Lambda_{E_7}}

\title[Homomesy via toggleability statistics]{Homomesy via toggleability statistics}

\author[C. Defant]{Colin Defant}\emailhref{colindefant@gmail.com}
\address{Department of Mathematics, Harvard University, Cambridge, MA 02138}
\thanks{C.D. was supported by a Fannie and John Hertz Foundation Fellowship and an NSF Graduate Research Fellowship \hphantom{aal}(grant number DGE-1656466).}

\author[S. Hopkins]{Sam Hopkins}\emailhref{samuelfhopkins@gmail.com}
\address{Department of Mathematics, Howard University, Washington, DC 20059}
\thanks{S.H. was supported by an NSF Mathematical Sciences Postdoctoral Research Fellowship (grant number 1802920).}

\author[S. Poznanovi\'c]{Svetlana Poznanovi\'c}\emailhref{spoznan@clemson.edu}
\address{School of Mathematical and Statistical Sciences, Clemson University, Clemson, SC 29634}
\thanks{S.P. was supported by NSF DMS 1815832.}

\author[J. Propp]{James Propp}\emailhref{jamespropp@gmail.com}
\address{Department of Mathematical Sciences, University of Massachusetts Lowell, Lowell, MA 01854}
\thanks{J.P. was supported by a Simons Collaboration Grant.}


\begin{document}

\begin{abstract}
The rowmotion operator acting on the set of order ideals of a finite poset has been the focus of a significant amount of recent research. One of the major goals has been to exhibit homomesies: statistics that have the same average along every orbit of the action. We systematize a technique for proving that various statistics of interest are homomesic by writing these statistics as linear combinations of ``toggleability statistics'' (originally introduced by Striker) plus a constant. We show that this technique recaptures most of the known homomesies for the posets on which rowmotion has been most studied. We also show that the technique continues to work in modified contexts. For instance, this technique also yields homomesies for the piecewise-linear and birational extensions of rowmotion; furthermore, we introduce a $q$-analogue of rowmotion and show that the technique yields homomesies for ``$q$-rowmotion'' as well.
\end{abstract}

\maketitle

\section{Introduction} \label{sec:intro}

\dfn{Dynamical algebraic combinatorics}~\cite{roby2016dynamical, striker2017dynamical} is the study of natural dynamical operators acting on objects familiar to algebraic combinatorics. A recurring theme in this field is \dfn{homomesy}, which occurs when a \dfn{statistic} (that is, a numerical function) on the set of objects has the same average along every orbit of the action. Homomesic statistics are, in a precise sense~\cite{proppspectrum}, complementary to \emph{invariant} statistics. But while it is usually hard to find nontrivial invariant statistics for these kinds of actions, there are often many interesting homomesies -- a phenomenon that is \emph{a priori} surprising when there are many orbits.

One of the most intensely studied operators from the point of view of dynamical algebraic combinatorics is the \dfn{rowmotion} operator acting on the distributive lattice~$\J(P)$ of order ideals of a (finite) poset $P$. For an order ideal $I\in\J(P)$, rowmotion applied to~$I$, denoted $\rowm(I)$, is the order ideal generated by $\min(P\setminus I)$, where $\min(P\setminus I)$ is the set of minimal elements of $P\setminus I$. Equivalently, $\rowm(I)$ is the unique $I' \in \J(P)$ with $\max(I')=\min(P \setminus I)$, where $\max(I')$ is the set of maximal elements of $I'$. See~\cite{striker2012promotion} or~\cite[\S 7]{thomas2019rowmotion} for the history of rowmotion. 

\begin{example} \label{ex:intro}
We use the notation $[n] \coloneqq \{1,2,\ldots,n\}$ for the set of natural numbers from~$1$ to $n$, which we also view as an $n$-element chain poset. Let $P=\rect{2}{2}$ be the (Cartesian) product of two $2$-element chains. We represent $\rect{2}{2}$ by a Young diagram (see \cref{subsec:posets} for more details about our conventions for drawing posets). The two rowmotion orbits of~$\J(P)$ are as shown in Figure~\ref{fig:two-orbits}.

We can see that one homomesic statistic for this action is the \dfn{order ideal cardinality} statistic $I \mapsto \#I$, which has average $2$ along both orbits. Another example of a homomesic statistic is the statistic $I\mapsto \#\max(I)$, which has average $1$ along both orbits; we call this the \dfn{antichain cardinality} statistic because $I \mapsto \max(I)$ is a canonical bijection from the set of order ideals to the set of antichains of $P$.
\end{example}

\begin{figure}
\begin{center}
\includegraphics[height=3cm]{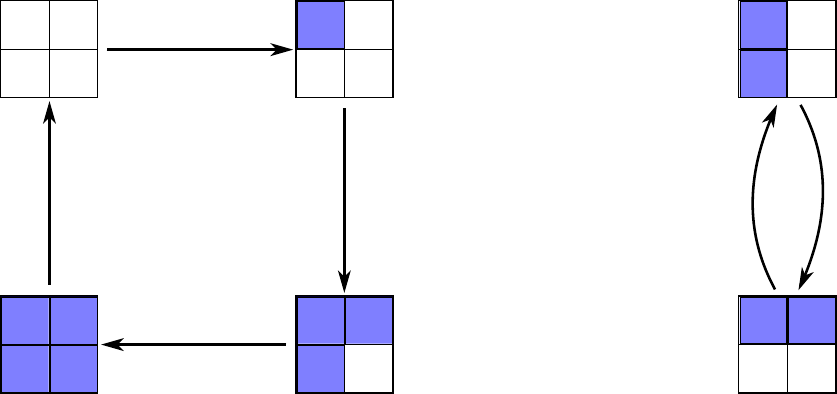}
\end{center}
\caption{The action of rowmotion on $\J(\rect{2}{2})$.}\label{fig:two-orbits}
\end{figure}

\Cref{ex:intro} is a special case of a more general phenomenon: it is known~\cite{propp2015homomesy} that both the order ideal and antichain cardinality statistics are homomesic for the action of rowmotion on $\J(\rect{a}{b})$ for any $a,b\geq 1$.

In this paper, we formalize a technique for proving instances of homomesy for rowmotion that is both powerful and robust. By ``powerful,'' we mean that it allows us to recapture, in a uniform way, most of the known instances of rowmotion homomesy; see \cref{sec:main}. By ``robust,'' we mean that it continues to apply in modified contexts, and hence allows us to obtain new results as well; see \cref{sec:pl_birational,sec:q}.

The basic idea behind the technique is quite simple. For any element $p \in P$, define the \dfn{toggleability statistics} $\Tin{p}, \Tout{p}, \T{p}\colon \J(P)\to \RR$ by
\begin{align*}
\Tin{p}(I) &\coloneqq \begin{cases} 1 &\textrm{if $p\in \min(P\setminus I)$}, \\ 
0 & \textrm{otherwise}; \end{cases} \\
\Tout{p}(I) &\coloneqq \begin{cases} 1 &\textrm{if $p\in \max(I)$}, \\ 
0 & \textrm{otherwise}; \end{cases} \\
\T{p}(I) &\coloneqq \Tin{p}(I) - \Tout{p}(I).
\end{align*}

Here, \dfn{toggling} refers to the operation of changing the status of whether or not~$p$ is in $I$, if possible: we add $p$ to $I$ if $I\cup\{p\}$ is an order ideal; we remove $p$ from $I$ if~$I \setminus\{p\}$ is an order ideal; and we leave $I$ alone otherwise. Since the work of Cameron and Fon-der-Flaass~\cite{cameron1995orbits} (see also~\cite{striker2012promotion}), it has been known that rowmotion can be written as a composition of these toggle involutions. The statistics $\Tin{p}$ and $\Tout{p}$ record when we can ``toggle $p$ into'' and ``toggle $p$ out of'' an order ideal, respectively. 

It is a straightforward fact, first explicitly observed by Striker~\cite{striker2015toggle}, that for any poset $P$ and any~$p\in P$, the \dfn{signed toggleability statistic}~$\T{p}$ is $0$-mesic (i.e., homomesic with average $0$) for rowmotion acting on $\J(P)$. The simple reason for this fact is that $\Tin{p}(I) = \Tout{p}(\rowm(I))$ for all~$p\in P$ and $I \in \J(P)$. Consequently, along any rowmotion orbit, we can toggle $p$ in exactly as often as we can toggle $p$ out.

While it is nice that the statistics $\T{p}$ are always homomesic under rowmotion for any poset, we are more interested in statistics that have more obvious combinatorial significance such as the order ideal and antichain cardinalities. Indeed, the statistics that have previously been studied from the point of view of rowmotion homomesy are all linear combinations of (the indicator functions of) ``is $p$ in the order ideal~$I$?''\ or ``is~$p$ in the antichain $\max(I)$?''\ for poset elements $p\in P$. The~$\T{p}$ themselves are not linear combinations of such indicator functions. However, if by some miracle we can write the statistic $\stat{f} \colon \J(P) \to \RR$ that we care about in the form
\[ \stat{f} = c + \sum_{p\in P} c_p \T{p},\]
for constants $c,c_p \in \RR$, then Striker's observation tells us that~$\stat{f}$ is $c$-mesic (i.e., homomesic with average $c$) for $\rowm$ acting on $\J(P)$. The remarkable fact is that, for many of the posets $P$ and statistics $\stat{f}$ about which we care, it \emph{is} possible to write $\stat{f}$ as a constant plus a linear combination of signed toggleability statistics. This was first established, for the antichain cardinality statistic and several posets~$P$ coming from Young diagrams (including $\rect{a}{b}$), by Chan, Haddadan, Hopkins, and Moci~\cite{chan2017expected}. We will show that this is also true for the order ideal cardinality statistic when $P=\rect{a}{b}$, and indeed is true for most statistics known to be homomesic under rowmotion when $P$ is a poset that exhibits good rowmotion behavior (see, e.g.,~\cite{armstrong2013uniform, propp2015homomesy, haddadan2021homomesy, rush2015homomesy}).

\begin{example} \label{ex:intro_continued}
Consider $P=\rect{2}{2}$, as in \cref{ex:intro}. Let $\stat{f}'\colon \J(P)\to \RR$ denote the order ideal cardinality statistic, i.e., $\stat{f}'\colon I\mapsto\#I$. Then we have
\[ \stat{f}' = 2 + \left(-2\cdot \T{(1,1)} -\frac{3}{2}\cdot\T{(1,2)}-\frac{3}{2}\cdot\T{(2,1)}-2\cdot\T{(2,2)}\right).\]
Let $\stat{f}''\colon \J(P)\to \RR$ denote the antichain cardinality statistic, i.e., $\stat{f}''\colon I\mapsto\#\max(I)$. Then we have
\[ \stat{f}'' = 1 + \left(-1\cdot \T{(1,1)} -\frac{1}{2}\cdot\T{(1,2)}-\frac{1}{2}\T{(2,1)}+0\cdot\T{(2,2)}\right).\]
These equations (which the reader can verify by inspecting~\cref{tab:tog_combs_ex}) ``explain'' the rowmotion homomesies we saw in \cref{ex:intro}.
\end{example}

\begin{table}
\[\renewcommand{\arraystretch}{1.5} \begin{array}{c|cccc|cc}
 I & \T{(1,1)}\!\!\!\! & \T{(1,2)}\!\!\!\! & \T{(2,1)}\!\!\!\! & \T{(2,2)}\! & I \mapsto \#I & I\mapsto \#\max(I) \\ \hline
\parbox{0.3in}{\begin{center}\includegraphics[height=0.7cm]{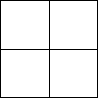}\end{center}} & +1 & 0 & 0 & 0\ \ & 0 & 0 \\
\parbox{0.3in}{\begin{center}\includegraphics[height=0.7cm]{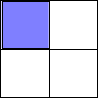}\end{center}} & -1 & +1 & +1 & 0\ \ & 1 & 1 \\
\parbox{0.3in}{\begin{center}\includegraphics[height=0.7cm]{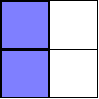}\end{center}} & 0 & +1 & -1 & 0\ \ & 2 & 1 \\
\parbox{0.3in}{\begin{center}\includegraphics[height=0.7cm]{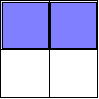}\end{center}} & 0 & -1 & +1 & 0\ \ & 2 & 1 \\
\parbox{0.3in}{\begin{center}\includegraphics[height=0.7cm]{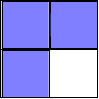}\end{center}} & 0 & -1 & -1 & +1\ \ & 3 & 2 \\
\parbox{0.3in}{\begin{center}\includegraphics[height=0.7cm]{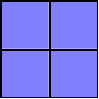}\end{center}} & 0 & 0 & 0 & -1\ \ & 4 & 1 
\end{array}\]
\caption{For \cref{ex:intro_continued}: various statistics on $\J(\rect{2}{2})$.}\label{tab:tog_combs_ex}
\end{table}

At a technical level, the key tools in our proofs are so-called \emph{rook} statistics, which are clever combinations of toggleability statistics that sum to a constant function. These rooks, first introduced by Chan, Haddadan, Hopkins, and Moci in~\cite{chan2017expected}, allow us to move from ``local'' to ``global'' statistics. Rooks have been subsequently extended for further applications in~\cite{hopkins2017cde, hopkins2019cde}. But until now, the focus has exclusively been on antichain cardinality; our aim here is to advertise how the technique is powerful and applies more broadly to many other statistics.

We also wish to advertise how this technique is, as mentioned, robust in the sense that it continues to work when we modify the setting in various ways. Let us briefly record some examples of this robustness: 

\begin{enumerate}
\item Rowmotion can be written as a composition of toggles. For many other operators written as a combination of toggles (including \emph{gyration}), Striker's observation persists, so this technique allows us to establish homomesies for these other operators as well (see~\cite{hopkins2017cde,striker2015toggle}).

\item There has been a great deal of interest lately in \emph{piecewise-linear} and \emph{birational} extensions of rowmotion~\cite{einstein2018combinatorial, grinberg2015birational2, grinberg2016birational1, musiker2018paths, okada2020birational}. As was essentially explained in~\cite{hopkins2021minuscule}, a representation of a statistic as a linear combination of toggleability statistics automatically ``lifts'' to the piecewise-linear and birational levels. Striker's observation is also easily seen to hold at the piecewise-linear and birational levels. Hence, when we use this technique to establish a homomesy for a certain statistic at the combinatorial level, we get \emph{for free} that the homomesy extends to the piecewise-linear and birational levels as well. In particular, using this automatic lifting to the higher realms, we are able to affirmatively resolve a recent homomesy conjecture of Bernstein, Striker, and Vorland~\cite{bernstein2021pstrict}; see~\cref{rem:bsv}.

\item In this paper, we define a $q$-analogue of rowmotion, which acts on a certain set of labelings of a finite poset $P$. The number of such labelings can be seen as a $q$-analogue of the number of order ideals of $P$. For example, $\J([a]\times[b])$ has $\binom{a+b}{b}$ order ideals (they are in bijection with lattice paths in an $a\times b$ grid), and the set of labelings on which \emph{$q$-rowmotion} acts has cardinality $\qbinom{a+b}{b}_q$ (or a related quantity), where $\qbinom{a+b}{b}_q$ is the $q$-binomial coefficient. We show that the toggleability statistics technique still works to establish homomesies for $q$-rowmotion. 
\end{enumerate}

We will discuss these points in much more detail later in the paper.

\begin{remark}
Let us briefly mention another example of the robustness of this technique that we will not discuss further after this remark:
\begin{enumerate}
\setcounter{enumi}{3}
\item Rowmotion, as we have defined it, acts on a distributive lattice. There has been recent interest~\cite{barnard2019canonical, semidistrim, thomas2019independence, thomas2019rowmotion} in extending rowmotion to act on other classes of lattices (or even other classes of posets). In~\cite{hopkins2019cde}, the toggleability statistics technique is used to establish some homomesies for rowmotion acting on certain \emph{semi}distributive lattices.
\end{enumerate}
\end{remark}

Four final points regarding how this technique is applied in practice merit emphasis. First, when one is looking at a particular poset $P$ and a particular statistic~$\stat{f}$ on $\J(P)$, the question of whether~$\stat{f}$ can be expressed in the form $c+\sum_{p\in P} c_p  \T{p}$ is just a matter of linear algebra and can easily be checked by computer. Second, since $\#\J(P)$ is typically much larger than~$\#P+1$, one would \emph{a priori} expect there to be {\em no} such expressions; the fact that in many cases of interest there are such expressions is part of the magic of the method that we do not understand. Third, Theorem~\ref{thm:linearindependence} guarantees that when a statistic~$\stat{f}$ can be expressed in the form $\stat{f} = c + \sum_{p\in P} c_p \T{p}$, the expression is unique. This means that if one is studying not just an individual poset but an infinite family of posets, as is usually the case, the experimenter does not need to make clever choices in a systematic way that will permit a unified analysis. Indeed, no choices (clever or otherwise) present themselves. When the technique is applicable, the challenge of applying it is not guessing coefficients that work but rather finding patterns governing the uniquely-determined coefficients and then finding a proof. Fourth, with this technique, it is often possible to present proofs of homomesy in an especially streamlined way, omitting explicit attention to the coefficients $c_p$. 

The rest of the paper is structured as follows. In \cref{sec:background}, we review background on posets, toggling, rowmotion, and homomesy. We explain Striker's observation that the signed toggleability statistics are $0$-mesic under rowmotion. We also prove that the signed toggleability statistics are linearly independent (in fact, we obtain a $q$-analogue of this result). In \cref{sec:main}, we explain in earnest the technique of writing a function as a linear combination of signed toggleability statistics plus a constant. We go through four important families of posets known to have good behavior of rowmotion and show how the technique works for each of these posets. We do this using ``rooks,'' which constitute our main technical tool.  In \cref{sec:pl_birational}, we explain the lifting to the piecewise-linear and birational realms. (As mentioned above, this lifting allows us to resolve a conjecture of Bernstein, Striker, and Vorland from \cite{bernstein2021pstrict}.) In \cref{sec:q}, we introduce $q$-rowmotion and explain how the toggleability statistics technique continues to work to prove interesting homomesy results for $q$-rowmotion. The definition of $q$-rowmotion and our results concerning it are some of the major new contributions of this article, and we believe it deserves further attention. \Cref{sec:pl_birational,,sec:q} are independent of one another and can be read in either order. For those who are just interested in $q$-rowmotion, we note that the definition of $q$-rowmotion at the beginning of \cref{sec:q} is self-contained and does not require the extensive machinery developed in \cref{sec:main}. Lastly, in \cref{sec:final}, we conclude with some final remarks and possible directions for future research.

\subsection*{Acknowledgments} 
Our collaboration was initiated at the 2020 Banff International Research Station (BIRS) online workshop on Dynamical Algebraic Combinatorics. We would like to thank BIRS, and the other organizers and participants of the conference, for creating a stimulating research environment. We are also very grateful to Noga Alon and Noah Kravitz for their help in discovering the proof of \cref{thm:linearindependence}. In doing this research, we found SageMath~\cite{sage} to be a valuable tool. For the sake of others who may wish to do computer experiments using our approach, we have publicly posted our Sage code (along with examples) at \begin{center}
\url{https://cocalc.com/share/ade9e90d5d9ccb662c5887908ccee47550c20c95/toggleability_statistics_code.sagews}
\end{center}

\section{Background} \label{sec:background}

\subsection{Posets} \label{subsec:posets}
We assume the reader is familiar with the basic notions and notations for (finite) posets, as laid out for instance in~\cite[Chapter 3]{stanley1996ec1}. In particular, we recall that $[n] \coloneqq \{1,2,\ldots,n\}$ is our notation for the $n$-element \dfn{chain} poset.  All posets in this article are assumed to be finite, and we drop this adjective from now on.

Henceforth in this section, let $P$ be a poset. All the constructions we care about will decompose in an obvious way into connected components; hence, we also tacitly assume $P$ is \dfn{connected} (i.e., the graph of its comparability relations is connected). 

An \dfn{order ideal} of $P$ is a subset $I\subseteq P$ that is downward-closed, i.e., for which $y \in I$ and $x \leq y \in P$ imply $x \in I$. We write $\J(P)$ for the set of order ideals of $P$. An \dfn{antichain} of $P$ is a subset $A\subseteq P$ of pairwise incomparable elements; we write~$\A(P)$ for the set of antichains of $P$. For an arbitrary subset~$S\subseteq P$, we use $\min(S)$ to denote the set of minimal elements of~$S$ and $\max(S)$ to denote the set of maximal elements. There is a natural bijection $\J(P)\to\A(P)$ given by $I \mapsto \max(I)$; the inverse of this bijection sends an antichain $A$ to the order ideal \dfn{generated} by $A$, which is defined to be $\{x\in P:x\leq y\text{ for some }y\in A\}$.

A \dfn{linear extension} of $P$ is a listing $p_1,\ldots,p_n$ of all of the elements of~$P$ that respects the order of $P$, i.e., for which $p_i \leq p_j$ implies $i \leq j$.

For $x,y\in P$, we say $x$ is \dfn{covered} by $y$, written $x\lessdot y$, if $x < y$ and there does not exist $z \in P$ with $x < z < y$; we also write $y \gtrdot x$ in this case and say $y$ \dfn{covers} $x$. We say that $P$ is \dfn{ranked} if there exists a \dfn{rank function} $\rk\colon P\to\Zge$ for which $\rk(y)=\rk(x)+1$ whenever $x\lessdot y$. We always assume $P$ contains an element~$p$ with $\rk(p)=0$; this assumption guarantees that a rank function is unique if it exists. The \dfn{rank} $\rk(P)$ of a ranked poset~$P$ is defined to be the maximum value of $\rk(p)$ over all elements $p\in P$.

Although it is traditional to depict a poset via its \dfn{Hasse diagram} (i.e., its graph of cover relations), we will instead use a more rectilinear representation. Indeed, all of the particular posets considered in this article can be viewed as subsets of~$\quadrant$.  In these depictions, we represent each element of the poset by a box and make the convention that two boxes $\square_1$ and $\square_2$ satisfy $\square_1\leq\square_2$ if and only if $\square_2$ appears weakly southeast of $\square_1$. (These depictions are thus analogous to \dfn{Young diagrams} of partitions.) For example, \cref{fig:four-elts} represents a $4$-element poset with cover relations $a\lessdot c$, $b\lessdot c$, and $c\lessdot d$. An order ideal is thus a northwest-closed set of boxes. We use ``matrix coordinates'' for the boxes; hence $(i,j)$ represents the box in row~$i$ (counted from top to bottom) and column~$j$ (counted from left to right). Thus, in \cref{fig:four-elts}, we have $a=(2,1)$, $b=(1,2)$, $c=(2,2)$, and $d=(2,3)$. The posets we will consider can in fact always be represented as \emph{contiguous} subsets of~$\quadrant$ and for that reason will all be ranked. The rank of $(i,j)$ in such a poset $P$ is~$\rk(i,j)=i+j-\min \{i'+j'\colon (i',j') \in P\}$.

\begin{figure}
\begin{center}
\includegraphics[height=1.2cm]{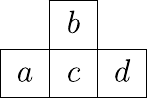}
\end{center}
\caption{A poset $\{a,b,c,d\}$ with covering relations $a\lessdot c$, $b\lessdot c$, and $c\lessdot d$.}\label{fig:four-elts}
\end{figure}

\subsection{Toggling and rowmotion} \label{subsec:toggling}

Recall from \cref{sec:intro} that \dfn{rowmotion} is the invertible map $\rowm\colon \J(P)\to \J(P)$ defined by
\[ \rowm(I)\coloneqq \{x\in P\colon x\leq y \textrm{ for some $y\in \min(P\setminus I)$}\}.\]
However, there is another way to describe rowmotion, in terms of toggles, that is also very useful and that we now review.

For each $p\in P$, define the \dfn{toggle} $\tog{p}\colon \J(P)\to\J(P)$ by 
\[\tog{p}(I)\coloneqq\begin{cases} I\cup\{p\} &\textrm{if $p\in \min(P\setminus I)$}, \\
I\setminus\{p\} &\textrm{if $p\in \max(I)$},\\ 
I &\textrm{otherwise}. \end{cases}\] 
Cameron and Fon-der-Flaass~\cite[Lemma 1]{cameron1995orbits} showed that
\[ \rowm=\tog{p_1}\circ\cdots\circ\tog{p_n} \]
for any linear extension $p_1,\ldots,p_n$ of $P$. (Here, as usual, functions in a composition are applied from right to left.) Note that different linear extensions give rise to the same operator because the toggles $\tog{p_i}$ and $\tog{p_j}$ commute whenever $p_i$ and $p_j$ are incomparable in $P$.

If $P$ is ranked and $0\leq i\leq\rk(P)$, we let $\rktog{i} \coloneqq \prod_{\rk(p)=i}\tog{p}$ be the product of the toggles at the elements of rank $i$; this product is well-defined because all of these toggles commute with each other. Observe that 
\[\rowm=\rktog{0} \circ \rktog{1} \circ \cdots \circ \rktog{\rk(P)}.\] 
One can also define \dfn{rank-permuted} variants of rowmotion as follows. Given a permutation $\sigma$ of the numbers $0,1,\ldots,\rk(P)$, we define
\[\rowm_{\sigma}\coloneqq \rktog{\sigma(0)} \circ \rktog{\sigma(1)} \circ \cdots \circ \rktog{\sigma(\rk(P))}.\]
For example, \dfn{gyration}~\cite[\S8.2]{striker2012promotion} is $\rowm_{1,3,5,\ldots,0,2,4,\ldots}$, the map that first toggles all the even ranks and then toggles all the odd ranks. We note that Cameron and Fon-der-Flaass~\cite[Lemma 2]{cameron1995orbits} showed that all the $\rowm_{\sigma}$ are conjugate, via toggles, to $\rowm$.

\begin{remark} \label{rem:promo_1}
Rowmotion toggles all the elements of $P$ ``from top to bottom.'' For various $P$ realized as subsets of $\quadrant$, a different ``left to right'' toggling of all the elements of $P$ called \emph{promotion} has also received significant attention (see, for instance, \cite{striker2012promotion}). We note that while promotion is also conjugate to rowmotion via toggles, it is \emph{not} of the form $\rowm_{\sigma}$, and, in general, our results do not apply to promotion. But for more on what we \emph{can} say about promotion, see \cref{rem:promo_2}.
\end{remark}

\begin{remark}
Via the bijection $I\mapsto \max(I)$ between $\J(P)$ and $\A(P)$, we can also view rowmotion as acting on the set of antichains of $P$, and this perspective is also often useful (for instance, in \cref{subsec:independence} below).
\end{remark}

\subsection{The statistics of interest} \label{subsec:statistics}

We now introduce notation for the statistics on~$\J(P)$ that will serve as building blocks for all the statistics we will prove to be homomesic under rowmotion. (Recall that by ``statistic,'' we just mean any function $\stat{f}\colon \J(P)\to\RR$.\footnote{Throughout, we work over the real numbers just for convenience and concreteness. Since all the statistics we care about are actually integer-valued, we could just as easily work over the rationals, or over the complex numbers, or over any field of characteristic zero. In \cref{sec:q}, we will consider statistics with values in $\RR(q)$ as well.}) 
Let us restate, in slightly altered form, the definitions of the \dfn{toggleability statistics} $\Tin{p}, \Tout{p}\colon \J(P)\to \RR$ and the \dfn{signed toggleability statistic} $\T{p}\colon \J(P)\to \RR$, for $p\in P$, from \cref{sec:intro}:
\begin{align*}
\Tin{p}(I) &\coloneqq \begin{cases} 1 &\textrm{if $p\notin I$ and $\tog{p}(I)=I\cup\{p\}$}, \\ 
0 & \textrm{otherwise}; \end{cases} \\
\Tout{p}(I) &\coloneqq \begin{cases} 1 &\textrm{if $p\in I$ and $\tog{p}(I)=I\setminus\{p\}$}, \\ 
0 & \textrm{otherwise}; \end{cases} \\
\T{p}(I) &\coloneqq \Tin{p}(I) - \Tout{p}(I).
\end{align*}
We have emphasized the direct connection between the statistics $\Tin{p}$, $\Tout{p}$ and the toggles~$\tog{p}$.  Observe that $\Tin{p}(I)=1$ if and only if $p \in \min (P \setminus I)$ and that $\Tout{p}(I)=1$ if and only if $p \in \max(I)$. For each $p\in P$, we also define the \dfn{order ideal indicator function} $\oii{p}\colon \J(P)\to\RR$ by 
\[\oii{p}(I) \coloneqq \begin{cases} 1 &\textrm{if $p\in I$}, \\ 
0 & \textrm{otherwise}. \end{cases}\]
Finally, by a slight abuse of notation, for $c\in \RR$, we write $c$ for the function that is identically equal to $c$. For instance, we will routinely consider the function $1\colon \J(P)\to\RR$ that is equal to $1$ on all order ideals.

Note that $\oii{p}$ is the indicator function of ``is $p$ in the order ideal $I$?,'' and similarly, $\Tout{p}$ is the indicator function of ``is $p$ in the antichain $\max(I)$?'' For that reason, we also refer to the $\Tout{p}$ as \dfn{antichain indicator functions}. Depending on whether we are viewing the action as order ideal or antichain rowmotion, we think of the~$\oii{p}$ or the~$\Tout{p}$ as the basic ``observable'' functions. Hence, the statistics we will prove to be homomesic under rowmotion are all either linear combinations of the $\oii{p}$ for $p\in P$ or linear combinations of the $\Tout{p}$ for~$p \in P$; i.e., they will belong to either $\Span\{\oii{p}\colon p\in P\}$ or $\Span\{\Tout{p}\colon p \in P\}$. The two prototypical such examples, to which we will devote the most attention, are the \dfn{order ideal cardinality statistic} $\sum_{p\in P}\oii{p}$ and the \dfn{antichain cardinality statistic} $\sum_{p\in P}\Tout{p}$. More generally, we might be interested in refinements of these cardinality statistics. We say that the collection of statistics $\stat{f}_1,\ldots,\stat{f}_s\colon \J(P)\to\RR$ \dfn{refines} the statistic $\stat{f}\colon \J(P)\to\RR$ when $\stat{f}=\stat{f}_1+\cdots+\stat{f}_s$. For example, one can obtain refinements of the cardinality statistics by choosing a partition of $P$ into subsets $B_1,\ldots,B_s$ and considering the collection of statistics $\sum_{p\in B_i}\oii{p}$ or $\sum_{p\in B_i}\Tout{p}$.

\subsection{Homomesy} \label{subsec:homomesy}

Recall from~\cite{propp2015homomesy} that a function $\stat{f}$ from a finite set $X$ to $\RR$ is said to be \dfn{homomesic} under the action of an invertible operator $T\colon X \to X$  if there exists a constant $c$ such that the average of $\stat{f}$ on each $T$-orbit is $c$;  in this case, we say $\stat{f}$ is \dfn{$c$-mesic}. A trivial example of a $c$-mesic function is the constant function $\stat{f}(x)=c$; also, if $T$ acts transitively on $X$ so that there is only one orbit, every statistic is homomesic. It is easy to see that every linear combination of homomesic functions is homomesic and that every linear combination of $0$-mesic functions is $0$-mesic.

We will be interested in homomesies for rowmotion acting on $\J(P)$. The following important observation of Striker~\cite{striker2015toggle} will be the building block for our subsequent investigation of such homomesies.

\begin{lemma}[{\cite[Lemma 6.2]{striker2015toggle}}] \label{lem:striker}
For any $p \in P$, the statistic $\T{p}$ is $0$-mesic under rowmotion.
\end{lemma}

\begin{proof}[Proof of~\cref{lem:striker}]
It follows immediately from the definition of rowmotion that for all $I \in \J(P)$, $\min(P \setminus I) = \max(\rowm(I))$, 
implying $\Tin{p}(I) = \Tout{p}(\rowm(I))$. Hence, if $O \subseteq \J(P)$ is any $\rowm$-orbit, we have that
\begin{eqnarray*}
\sum_{I \in O } \T{p}(I) 
& = & \sum_{I \in O } \left( \Tin{p}(I) - \Tout{p}(I) \right) \\ 
& = & \sum_{I \in O } \Tin{p}(I) - \sum_{I \in O} \Tout{p}(I) \\ 
& = & \sum_{I \in O } \Tin{p}(I) - \sum_{I \in O} \Tout{p}(\rowm(I)) \\ 
& = & \sum_{I \in O } \left( \Tin{p}(I) - \Tout{p}(\rowm(I)) \right) \\
& = & \sum_{I \in O } \left( \Tin{p}(I) - \Tin{p}(I) \right) = 0,
\end{eqnarray*}
where for the third equality, we use the fact that we are summing over a $\rowm$-closed set to replace~$I$ by $\rowm(I)$ in the second sum.
\end{proof}

In fact, Striker's lemma extends to the rank-permuted versions of rowmotion.

\begin{lemma}[{\cite[Theorem 6.7]{striker2015toggle}\cite[Lemma 7.7]{hopkins2017cde}}] \label{lem:striker_permuted}
If $P$ is ranked, then for any permutation $\sigma$ of~$0,1,\ldots,\rk(P)$ and any $p \in P$, the statistic $\T{p}$ is $0$-mesic under~$\rowm_{\sigma}$.
\end{lemma}

The proof of \cref{lem:striker_permuted} is slightly more involved than that of~\cref{lem:striker}; rather than include a proof of \cref{lem:striker_permuted}  here, we will prove in~\cref{sec:pl_birational} the birational and piecewise-linear extensions of this lemma (Lemmas~\ref{lem:striker_pl} and~\ref{lem:striker_b}) from which \cref{lem:striker_permuted} will follow as a corollary.

\bigskip

As we have already emphasized in \cref{sec:intro}, it follows trivially from \cref{lem:striker} that if a statistic $\stat{f} \colon \J(P)\to\RR$ has a representation $\stat{f}=c+\sum_{p\in P}c_p\T{p}$ for constants $c, c_p\in\RR$, then $\stat{f}$ is $c$-mesic under rowmotion (and \cref{lem:striker_permuted} says similarly for the rank-permuted versions of rowmotion).  In fact, this is how we will show various statistics are homomesic. 

To that end, we introduce the following notation that we will use repeatedly throughout the rest of the paper: if $\stat{f}, \stat{g}\colon \J(P)\to\RR$ are two statistics on $\J(P)$, we write $\stat{f} \tequiv \stat{g}$ to mean that there are constants $c_p \in \RR$ for which $\stat{f}-\stat{g} = \sum_{p\in P}c_p\T{p}$. Evidently, $\tequiv$ is an equivalence relation; moreover, it is a congruence in the sense that $\stat{f}_1 \tequiv \stat{g}_1$ and $\stat{f}_2 \tequiv \stat{g}_2$ imply $a_1\cdot \stat{f}_1+a_2\cdot\stat{f}_2\tequiv a_1\cdot\stat{g}_1+a_2\cdot\stat{g}_2$ for all $a_1,a_2\in\RR$. We will use this property of $\tequiv$ in what follows without further comment.

With this notation, we can state the following consequence of \cref{lem:striker,lem:striker_permuted}:

\begin{prop} \label{prop:homo}
If $\stat{f} \colon \J(P)\to\RR$ satisfies $\stat{f} \tequiv c$ for some $c \in \RR$, then $\stat{f}$ is $c$-mesic under rowmotion; if $P$ is ranked, it is $c$-mesic under $\rowm_{\sigma}$ for any $\sigma$ as well.
\end{prop}

For a statistic $\stat{f}\colon \J(P)\to \RR$, let us introduce the shorthand $\stat{f}\tequiv \const$ to mean that there is some constant $c\in \RR$ for which $\stat{f}\tequiv c$. In light of \cref{prop:homo}, our focus in \cref{sec:main} will be showing that~$\stat{f} \tequiv \const$ for various statistics $\stat{f}$ on $\J(P)$. From now on, each time we show $\stat{f}\tequiv \const$, we will not separately state as a corollary that this implies $\stat{f}$ is homomesic under rowmotion, but that is of course the impetus for showing $\stat{f} \tequiv \const$.

\begin{remark} \label{rem:promo_2}
As mentioned in \cref{rem:promo_1}, the ``left to right'' composition of toggles known as promotion is another operator on $\J(P)$, related to rowmotion, that has also received significant attention. We note importantly that the signed toggleability statistics $\T{p}$ are \emph{not} in general homomesic under promotion; this can already be seen for $P=\rect{2}{2}$. Therefore, in general, our results do not apply to promotion. However, for statistics $\stat{f}\colon\J(P)\to\RR$ that are linear combinations of the order ideal indicator functions, i.e., for $\stat{f}\in \Span\{\oii{p}\colon p\in P\}$, the \emph{recombination} method of Einstein and Propp~\cite[\S6]{einstein2018combinatorial} can be used to deduce that $\stat{f}$ is homomesic under promotion if it is homomesic under rowmotion (and vice-versa). Hence, in fact our results here \emph{do} yield some homomesy results for promotion as well as rowmotion, although we will not explain all the details for these implications.
\end{remark}

\subsection{Linear independence} \label{subsec:independence}

Our results below assert that for various statistics $\stat{f} \colon \J(P)\to \RR$, we have $\stat{f} \tequiv \const$, i.e., that there are $c, c_p\in\RR$ for which $\stat{f} =c+\sum_{p\in P}c_p\T{p}$. When one tries expressing~$\stat{f}$ as a linear combination of $1$ and the signed toggleability statistics $\T{p}$ in this way, it is natural to ask if the resulting coefficients are unique. In other words, we would like to know if these statistics are linearly independent. Here, we prove that this is indeed the case. In fact, we prove a more general result concerning $q$-analogues of the signed toggleability statistics. For $q\in\RR$, let $\qT{p}\coloneqq \Tin{p}- q\Tout{p}$. Of course, $\T{p}^1$ is the same as $\T{p}$. 

The only time that we will actually need the following theorem is in \cref{sec:q}, where we will use it to avoid a technical issue concerning singularities of rational functions. Therefore, the casual reader can safely bypass this.

\begin{thm}\label{thm:linearindependence}
Fix a real number $q\geq 0$. The statistics $\qT{p}\colon\J(P)\to\RR$ are linearly independent over $\RR$, and they are also linearly independent from $1$.
\end{thm}

\begin{proof}
Let $\stat{g}_p(I)\coloneqq q^{\#(P\setminus I)}\qT{p}(I)$. It is clear that $\stat{g}_p(I)=-\stat{g}_p(\tog{p}(I))$ for all $p\in P$ and $I\in\J(P)$. Since the toggles $\tog{p}$ are involutions, this implies $\sum_{I\in\J(P)}q^{\#(P\setminus I)}\qT{p}(I)=\sum_{I\in\J(P)}\stat{g}_p(I)=0$ for all $p\in P$. If we could express the constant function $1$ as a linear combination $1=\sum_{p\in P}c_p\qT{p}$ for some coefficients $c_p\in\RR$, then we would have \[\sum_{I\in\J(P)}q^{\#(P\setminus I)}=\sum_{I\in\J(P)}q^{\#(P\setminus I)}\sum_{p\in P}c_p\qT{p}(I)=\sum_{p\in P}c_p\sum_{I\in\J(P)}q^{\#(P\setminus I)}\qT{p}(I)=0.\] However, this cannot happen because $\sum_{I\in\J(P)}q^{\#(P\setminus I)} = 1 + \sum_{I\in\J(P)\setminus\{P\}}q^{\#(P\setminus I)} > 0$ (since $q \geq 0$). This proves that the statistics $\qT{p}$ are linearly independent from $1$. So from now on, we work to show that the $\qT{p}$ are linearly independent from each other.

For this proof, we will find it convenient to work with \dfn{antichain rowmotion}, which is the operator $\pan\colon\A(P)\to\A(P)$ defined by $\pan(A) \coloneqq \min(P\setminus \{x \in P\colon x \leq y \textrm{ for some $y\in A$}\})$. Here, $\{x \in P\colon x \leq y \textrm{ for some $y\in A$}\}$  is the order ideal generated by $A$. In other words, we have $\pan(\max(I))=\max(\rowm(I))$ for all $I\in\J(P)$,
where the first instance of $\pan$ is antichain rowmotion and the second is order ideal rowmotion. 

Now suppose $c_p \in \RR$ for $p\in P$ are real coefficients such that $\sum_{p\in P}c_p\T{p}^q=0$. Let~$A\in\A(P)$, and let $I$ be the order ideal generated by $A$. Then we have
\[\qT{p}(I) = \left\{ \begin{array}{rl}
1 & \mbox{if $p \in \pan(A)$}, \\
-q & \mbox{if $p \in A$}, \\
0 & \mbox{otherwise};
\end{array}\right.\]
so the identity $\sum_{p\in P}c_p\T{p}^q(I)=0$ is equivalent to $q\sum_{p\in A}c_p=\sum_{p\in\pan(A)}c_p$. Therefore, the linear independence of the $\T{p}^q$ is equivalent to the following claim:

\medskip
\noindent {\bf Claim.} If a function $\pi\colon P\to\RR$ satisfies $q\sum_{p\in A}\pi(p)=\sum_{p\in\pan(A)}\pi(p)$ for all antichains $A\in\A(P)$, then $\pi(p)=0$ for all $p\in P$. 
\medskip

We proceed to prove this claim. Let $\pi\colon P\to\RR$ be such that $q\sum_{p\in A}\pi(p)=\sum_{p\in\pan(A)}\pi(p)$ for all $A\in\A(P)$. We will show that $\pi(p)=0$ for all~$p\in P$.

First suppose $q\neq 1$. Let $m$ be the order of antichain rowmotion on~$\A(P)$ so that $\pan^m(A) = A$ for all $A \in \A(P)$. Thus, we have 
\[q^m\sum_{p\in A}\pi(p)=\sum_{p\in\pan^m(A)}\pi(p)=\sum_{p\in A}\pi(p),\] 
for all $A \in \A(P)$. In other words, $(1-q^m)\sum_{p\in A}\pi(p)=0$ for all $A\in \A(P)$. Since~$q\neq 1$, this means $\sum_{p\in A}\pi(p)=0$ for all $A\in \A(P)$. By considering the antichain $A=\{p\}$, we see $\pi(p)=0$ for each $p\in P$.

So now suppose $q=1$. The argument in this case is more involved. Let $n$ be the number of elements of $P$. The claim is vacuous in the case $n=0$, so we may assume $n\geq 1$ and proceed by induction on $n$. We will consider different posets in the proof, so for any poset $Q$, we write $\pan_Q$ to denote antichain rowmotion on~$\A(Q)$. 

Let us first consider the case where there is some maximal element $x\in P$ such that $\pi(x)=0$. Consider the subposet $P'\coloneqq P\setminus\{x\}$. Let $A\in\A(P')$. We can view $A$ as an antichain in~$\A(P)$ as well. Note that $\pan_P(A)$ is equal to either $\pan_{P'}(A)$ or $\pan_{P'}(A)\cup\{x\}$ (we are using the maximality of $x$). In either case, the condition $\sum_{p\in A}\pi(p)=\sum_{p\in\pan_P(A)}\pi(p)$ is equivalent (because $\pi(x)=0$) to the condition $\sum_{p\in A}\pi(p)=\sum_{p\in\pan_{P'}(A)}\pi(p)$. As this is true for all $A\in\A(P')$, we can use our induction hypothesis to see that $\pi(p)=0$ for all $p\in P'$. Hence, $\pi(p)=0$ for all $p\in P$. 

The previous paragraph tells us that it suffices to prove that $\pi(x)=0$ for some maximal element~$x$ of $P$. Choose some maximal element~$x$, and let $M$ be the (possibly empty) set of maximal elements of $P$ that are not $x$. Let $\kappa \coloneqq \sum_{p\in M}\pi(p)$. Let $I$ be the order ideal generated by $M$, and let $P''\coloneqq P\setminus I$. Let $m$ be the order of antichain rowmotion on $\A(P'')$ so that $\pan_{P''}^m(A)=A$ for all $A\in\A(P'')$. For each $A\in\A(P'')$, the set $A\cup M$ is an antichain in $P$ such that $\pan_P(A\cup M)=\pan_{P''}(A)$. Therefore, 
\[\kappa+\sum_{p\in A}\pi(p)=\sum_{p\in A\cup M}\pi(p)=\sum_{p\in\pan_P(A\cup M)}\pi(p)=\sum_{p\in\pan_{P''}(A)}\pi(p).\] 
This shows that $\sum_{p\in\pan_{P''}(A)}\pi(p)-\sum_{p\in A}\pi(p)=\kappa$ for all $A\in\A(P'')$. If we replace~$A$ by $\pan_{P''}^k (A)$ and sum over all $0 \leq k \leq m-1$, we find (after telescoping) that $\sum_{p\in\pan^m_{P''}(A)}\pi(p)-\sum_{p\in A}\pi(p)=m\kappa$. But the left-hand side of this last equation equals 0 since $\pan^m_{P''}(A) = A$. So $\kappa=0$. Finally, $\pan_P(M\cup\{x\})=\emptyset$, so 
\[\pi(x)=\pi(x)+\kappa=\sum_{p\in M\cup\{x\}}\pi(p)=\sum_{p\in\pan_P(M\cup\{x\})}\pi(p)=0.\] 
This proves that $\pi(x)=0$, which completes the proof.
\end{proof}

\begin{remark}
\Cref{thm:linearindependence} is not true with $q=-1$. For instance, when $P$ is an antichain and $q=-1$, all the statistics $\T{p}^q$ for $p \in P$ coincide with the constant statistic $1$.
\end{remark}

\begin{remark}
When $P$ is a chain, we have $\#\J(P)=\#P+1$ and hence, by \cref{thm:linearindependence}, $1$ and $\T{p}$ for $p\in P$ span the entire space of functions $\J(P)\to\RR$. In any other case,  we have $\#\J(P)>\#P+1$, so $1$ and the $\T{p}$ span a strict subspace of the space of functions $\J(P)\to\RR$.
\end{remark}

\begin{remark}
As mentioned above, we are interested in statistics in $\Span(\{1\}\cup\{\T{p}\colon p\in P\})$ that also belong either to $\Span\{\oii{p}\colon p \in P\}$ or to $\Span\{\Tout{p}\colon p \in P\}$. Clearly, these latter two sets are also linearly independent, so we will also have unique representations of our statistics as sums of order ideal/antichain indicator functions.
\end{remark}

\section{The toggleability statistics technique} \label{sec:main}

Our focus now turns to showing $\stat{f}\tequiv \const$ for ``interesting'' statistics $\stat{f}\colon \J(P)\to\RR$ and certain families of posets $P$. (Recall from \cref{subsec:homomesy} that we write $\stat{f} \tequiv \stat{g}$ if $\stat{f}-\stat{g} = \sum_{p \in P} c_p \T{p}$ for some $c_p \in \RR$, and write $\stat{f}\tequiv \const$ if $\stat{f}\tequiv c$ for some $c \in \RR$.) The posets we consider are those for which rowmotion has been shown to exhibit good behavior. The most prominent among these is the \dfn{rectangle} poset $\rect{a}{b}$. The others are all \emph{minuscule posets} or \emph{root posets} (rectangles are examples of minuscule posets). While such posets have algebraic significance, our treatment is ultimately case-by-case and does not use any algebraic machinery.

\begin{remark}
Homomesy is surprising when we consider an infinite family of posets and the number of orbits for each poset is relatively large. Every poset $P$ that we will discuss in this section is ranked and has rowmotion order dividing $2(\rk(P)+2)$ (see~\cite{armstrong2013uniform,rush2013orbits}). This order is (usually) small compared to $\#\J(P)$, so there must in fact be many orbits of rowmotion for such $P$. Hence, the homomesy corollaries we obtain from showing $\stat{f}\tequiv \const$ are indeed \emph{a priori} surprising.
\end{remark}

\begin{remark}
Much (though certainly not all) of this section is a streamlining of results from~\cite{chan2017expected,hopkins2017cde}. We repeat some arguments from those papers for completeness and to aid with the narrative flow. We review the history of results as we go.
\end{remark}

As we have explained in \cref{subsec:posets}, all our main examples of posets will be realized as contiguous subsets of $\quadrant$, and we will draw these posets rectilinearly with $(1,1)$ at the upper left. In particular, the elements of our posets will be boxes~$(i,j)$. We will often write $(i,j)$ as $i,j$ for brevity. Another useful convention we will adopt throughout is that if a given $(i,j)$ does not belong to our poset, we consider all its associated statistics $\oii{i,j}$, $\Tin{i,j}$, $\Tout{i,j}$, $\T{i,j}$ to be zero.
 
\subsection{The rectangle} \label{subsec:rectangle}

In this subsection, we fix $P = \rect{a}{b}$.

\subsubsection{Rooks}

As mentioned in \cref{sec:intro}, certain clever combinations of toggleability statistics that we refer to as ``rooks'' are at the technical heart of our proofs that $\stat{f} \tequiv \const$ for various statistics~$\stat{f}$. We now introduce these rooks.

For $(i,j)\in P$, define the \dfn{rook} $\rook{i,j}\colon \J(P)\to \RR$ by
\begin{equation} \label{eqn:rect_rook}
\rook{i,j} \ \coloneqq \sum_{i' \leq i, \ j' \leq j} \Tin{i',j'} \ \ - \sum_{i' < i,\  j' < j} \Tout{i',j'} \ \ + \sum_{i' \geq i, \ j' \geq j} \Tout{i',j'} \ \ - \sum_{i' > i, \ j' > j} \Tin{i',j'}
\end{equation}
and the \dfn{reduced rook} $\rrook{i,j}\colon \J(P)\to \RR$ by
\begin{equation}  \label{eqn:rect_red_rook}
\rrook{i,j} \ \coloneqq \ \sum_{j'} \Tout{i,j'} \ + \ \sum_{i'} \Tout{i',j}\ .
\end{equation}
The sum in~\eqref{eqn:rect_rook} defining $\rook{i,j}$ has a nice pictorial representation, as shown on the left in \cref{fig:rectrook} for $P = \rect{5}{6}$ and $(i,j)=(3,3)$.

\begin{figure}
\begin{center}
\includegraphics[height=4.717cm]{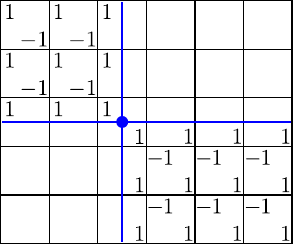}\qquad\qquad\qquad\qquad\includegraphics[height=4.717cm]{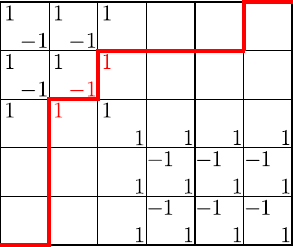}
\end{center} 
\caption{The left image illustrates the rook $R_{3,3}$ on the rectangle $[5]\times[6]$, while that on the right illustrates the proof of \cref{thm:rect_sumone}. A box with $d$ in its upper left and $e$ in its lower right at location $(i',j')$ denotes a contribution of $d \Tin{i',j'} + e \Tout{i',j'}$. There are $0$'s tacitly present at all locations where no number appears. In the image on the right, the boxes $(3,2)$, $(2,2)$, $(2,3)$ contribute $1$, $-1$, $1$, respectively, with the total sum along the red path being $1$.} \label{fig:rectrook}
\end{figure}

Observe that $\Tin{p} \tequiv \Tout{p}$ (since $\Tin{p} - \Tout{p} =\T{p}$). Thus, we have the following lemma relating $\rook{i,j}$ and $\rrook{i,j}$.

\begin{lemma} \label{lem:rect_rook_equiv}
For all $(i,j)\in \rect{a}{b}$, $\rook{i,j} \tequiv \rrook{i,j}$.
\end{lemma}

The name ``rook'' comes from the fact that $\rrook{i,j}$ consists of a sum of the functions $\Tout{p}$ over all boxes $p$ in the same row as $(i,j)$, plus a sum of the functions $\Tout{p}$ over all boxes $p$ in the same column as $(i,j)$. In this way, $\rrook{i,j}$ ``attacks'' each box in the same row/column as $(i,j)$, like a rook chess piece (although note that it attacks its own box $(i,j)$ twice). Since we often work modulo the span of the $\T{p}$, we also think of $\rook{i,j}$ attacking these boxes in this way. In the image on the left in \cref{fig:rectrook}, we have drawn blue lines in the row and column of~$(i,j)$ to visually represent which boxes this rook $\rook{i,j}$ attacks.

The following theorem, which is far from obvious, is what makes rooks important and useful.

\begin{thm} \label{thm:rect_sumone}
For $(i,j)\in \rect{a}{b}$, we have $\rook{i,j}(I) = 1$ for all $I\in\J(P)$. In other words, we have $\rook{i,j}=1$ as an equality of functions.
\end{thm}

\begin{proof}
Given an order ideal $I$ viewed as a set of boxes inside an $a\times b$ rectangle, define the \dfn{boundary} of $I$ as the lattice path that joins the lower left corner of the rectangle to the upper right corner of the rectangle and separates elements of $I$ from elements of $P \setminus I$. Evaluating $\rook{i,j}(I)$ amounts to summing the numbers from the pictorial representation of the rook $\rook{i,j}$ that border the boundary of $I$, as shown (in red) on the right in \cref{fig:rectrook}. For if box $p$ does not border the boundary path, then it cannot be toggled into or toggled out of $I$ and~$\Tin{p}(I) = \Tout{p}(I) = 0$; on the other hand, if box $p$ does border the boundary path, then it can be toggled in (resp.\ out) if and only if the path contains the left and top (resp.\ bottom and right) sides of $p$.

There are two cases to consider according to whether $(i,j) \not\in I$ (in which case the boundary of~$I$ passes above and to the left of box $(i,j)$) or $(i,j) \in I$ (in which case the boundary of $I$ passes below and to the right of box $(i,j)$). But in either case, the contributions to the path-sum alternate between $+1$ and $-1$ as we traverse the path, beginning and ending with $+1$, so the total is $1$.
\end{proof}

Rooks were introduced by Chan, Haddadan, Hopkins, and Moci~\cite{chan2017expected}, and \cref{lem:rect_rook_equiv} and~\cref{thm:rect_sumone} were first proved in that paper (see~\cite[Lemma~3.5, Theorem~3.4]{chan2017expected}). (In fact, rooks were defined in~\cite{chan2017expected} not just for the rectangle, but for any Young diagram shape, including skew shapes.)

\subsubsection{Antichains} \label{subsubsec:rectant}

We now focus on proving that various $\stat{f}$ that are linear combinations of the antichain indicator functions, i.e., $\stat{f}\in \Span\{\Tout{p}\colon p \in P\}$, satisfy $\stat{f}\tequiv \const$. In fact, thanks to \cref{lem:rect_rook_equiv} and~\cref{thm:rect_sumone}, we already know this is true of the reduced rooks, which satisfy $\rrook{i,j}\tequiv 1$. But now we want to show how the~$\rrook{i,j}$ can be combined to yield other, natural functions $\stat{f}$ in the span of the functions~$\Tout{p}$. We begin with the antichain cardinality statistic. 

\begin{thm} \label{thm:aboapb}
For $P=\rect{a}{b}$, we have $\sum_{p\in P}\Tout{p} \tequiv ab/(a+b)$.
\end{thm}

\begin{proof}
Note that 
\[\sum_{i,j}\rrook{i,j} = \sum_{i,j} \left( \sum_{j'} \Tout{i,j'} + \sum_{i'} \Tout{i',j} \right) = \sum_{i,j,j'} \Tout{i,j'} + \sum_{i,i',j} \Tout{i',j} = \sum_{i,j,j'} \Tout{i,j} + \sum_{i,i',j} \Tout{i,j} = \sum_{i,j} (a+b) \Tout{i,j}.\] 
The third equality holds because we can swap $j$ with $j'$ in the first sum and $i$ with $i'$ in the second sum, and the fourth equality holds because every box belongs to the same row as $a$ boxes and the same column as $b$ boxes. (An abbreviated way to say $\sum_{i,j}\rrook{i,j} = \sum_{i,j} (a+b) \Tout{i,j}$ is ``every box is attacked $a+b$ times.'') It follows from \cref{lem:rect_rook_equiv} and \cref{thm:rect_sumone} that 
\[ab=\sum_{i,j}\rook{i,j}\tequiv\sum_{i,j}\rrook{i,j}=\sum_{i,j}(a+b)\Tout{i,j}=(a+b)\sum_{p\in P}\Tout{p},\] 
so $\sum_{p\in P}\Tout{p} \tequiv ab/(a+b)$. 
\end{proof}

\begin{remark}
Although the proof of \cref{thm:aboapb} does not require an explicit representation of~$\sum_{p\in P}\Tout{p}$ in the form $ab/(a+b) + \sum_{p \in P} c_p \T{p}$, we mention that unwinding the argument above yields 
\[c_{i,j} = (ab - a(b+1-j) - b(a+1-i))/(a+b).\] 
These coefficients could be useful to future researchers for other purposes. However, from now on, we will not bother to work out these coefficients $c_p$ when we show $\stat{f}\tequiv \const$ for other statistics $\stat{f}$.
\end{remark}

A \dfn{positive fiber} of the rectangle $\rect{a}{b}$ is a subset $B \subseteq \rect{a}{b}$ of the form $\{i\} \times [b]$ for some~$i\in[a]$. A \dfn{negative fiber} of $\rect{a}{b}$ is a subset $B \subseteq \rect{a}{b}$ of the form $[a] \times \{j\}$ for some~$j\in[b]$. 

\begin{thm} \label{thm:rectfiber}
If $B=\{i\}\times[b] \subseteq \rect{a}{b}$ is a positive fiber, then $\sum_{p \in B} \Tout{p} \tequiv b/(a+b)$. If $B=[a] \times \{j\}$ is a negative fiber, then $\sum_{p \in B} \Tout{p} \tequiv a/(a+b)$.
\end{thm}

\begin{proof}
By symmetry, it suffices to consider the case when $B=\{i\}\times[b] \subseteq \rect{a}{b}$ is a positive fiber. For each $1\leq i\leq a$, set $\stat{f}_i \coloneqq \sum_j \Tout{i,j}$. To prove that $\stat{f}_i \tequiv b/(a+b)$, we begin by noting that, thanks to \cref{lem:rect_rook_equiv} and \cref{thm:rect_sumone},
\[ 0 = \rook{i,1} - \rook{i+1,1} \tequiv \rrook{i,1} - \rrook{i+1,1} = \stat{f}_i - \stat{f}_{i+1}\]
for all $1\leq i < a$. (In fact, we can also see that $\stat{f}_i - \stat{f}_{i+1}\tequiv 0$ more directly from the observation that for any $I\in \J(P)$, we can toggle some element of the $i$-th positive fiber out of $I$ if and only if we can toggle some element of the $(i+1)$-st positive fiber \emph{into} $I$.)

Thus, we have $\stat{f}_1\tequiv \stat{f}_2 \tequiv \cdots \tequiv \stat{f}_a$. But also, we know from \cref{thm:aboapb} that
\[\stat{f}_1 + \stat{f}_2 + \cdots + \stat{f}_a = \sum_{p\in P}\Tout{p} \tequiv ab/(a+b).\]
So for any $1\leq i \leq a$, we have $a\cdot \stat{f}_i\tequiv ab/(a+b)$. We conclude that $\stat{f}_i\tequiv  b/(a+b)$, as claimed.
\end{proof}

That the antichain cardinality statistic, and its fiber refinements, are homomesic under rowmotion for the rectangle was first proved by Propp and Roby~\cite[Theorem~27]{propp2015homomesy}. The stronger results that these statistics are $\tequiv \const$ (i.e., \cref{thm:aboapb,thm:rectfiber}) were first proved by Chan, Haddadan, Hopkins, and Moci~\cite[pg.~23]{chan2017expected}, with the same proofs we have given above.

\subsubsection{Order ideals} \label{subsubsec:rectord}

We now show that various $\stat{f}$ that are linear combinations of the order ideal indicator functions, i.e., $\stat{f}\in \Span\{\oii{p}\colon p \in P\}$, satisfy $\stat{f}\tequiv \const$.

We first need a lemma relating the order ideal indicator functions to the toggleability statistics.

\begin{lemma} \label{lem:rect_ds1s}
For any $p=(i,j) \in \rect{a}{b}$, we have
\begin{equation}\label{eqn:rectfilea}
 \oii{p} \ = \sum_{i' \geq i, \ j' \geq j} \Tout{i',j'} \ \ - \sum_{i' > i, \ j' > j} \Tin{i',j'}.
\end{equation}
\end{lemma}

\begin{proof}
Consider an order ideal $I \in \J(P)$. The elements $(i',j')$ that contribute $+1$ to the right-hand side of \eqref{eqn:rectfilea} are those for which $i' \geq i$, $j' \geq j$, and $\Tout{i',j'}(I) = 1$; the elements $(i',j')$ that contribute~$-1$ to the right-hand side of \eqref{eqn:rectfilea} are those for which $i' > i$, $j' > j$, and $\Tin{i',j'}(I) = 1$. For example, \cref{fig:rect_ds1s} depicts $p=(2,3)\in P = \rect{4}{7}$ (the box marked with an ``X''), with the boundary of an order ideal $I\in \J(P)$ drawn in red.  In the lower right corners of those $(i',j')$ with $i' \geq i$, $j' \geq j$, we have written a $1$, and in the upper left corners of those $i' > i$, $j' > j$, we have written a $-1$. Observe how along the red path we encounter a $+1$ corner, then a $-1$ corner, then a $+1$ corner, and so on, ending with a~$+1$ corner. 

More generally, it is easy to see that if $p \in I$, then the contributions to the right-hand side of \eqref{eqn:rectfilea} alternate in sign, beginning and ending with a~$+1$, so that the net contribution is~$+1$.

On the other hand, if $p\notin I$ then there are no $i',j'$ satisfying either $i' \geq i$, $j' \geq j'$ and $\Tout{i',j'}(I)=1$, or $i' > i$, $j' > j'$ and $\Tin{i',j'}(I)=1$, so the contribution to the right-hand side of \eqref{eqn:rectfilea} is $0$. 
\end{proof}

\begin{figure}
\begin{center}
\includegraphics[height=3.789cm]{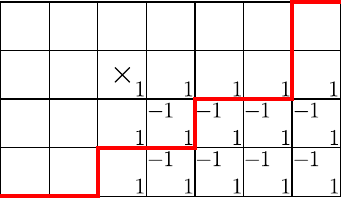}
\end{center}
\caption{An illustration of the proof of \cref{lem:rect_ds1s} with $a=4$, $b=7$, and $(i,j)=(2,3)$. }\label{fig:rect_ds1s} 
\end{figure}

\begin{remark}
The reader may notice considerable similarity between the proof of \cref{lem:rect_ds1s} we just gave and the earlier proof of \cref{thm:rect_sumone}: for instance, both involved noticing a pattern of~$+1,-1,\ldots,+1$ coefficients along a boundary path to obtain an overall sum of $1$. In fact, the ``$180^\circ$ rotation'' of \cref{lem:rect_ds1s} asserts that
\begin{equation} \label{eqn:rectfilea_alt}
1 - \oii{p} = \sum_{i' \leq i, \ j' \leq j} \Tin{i',j'} - \sum_{i' < i, \ j' < j} \Tout{i',j'}.
\end{equation}
Summing \eqref{eqn:rectfilea} and \eqref{eqn:rectfilea_alt} yields precisely the equation $1=\rook{i,j}$ (where we recall the definition~\eqref{eqn:rect_rook} of $\rook{i,j}$ as a sum of toggleability statistics). In this way, \cref{lem:rect_ds1s} actually implies \cref{thm:rect_sumone}. What is more, insofar as \eqref{eqn:rectfilea} and \eqref{eqn:rectfilea_alt} each contribute two of the four sums defining the rook $\rook{i,j}$, we see that the sum of toggleability statistics in \eqref{eqn:rectfilea} can be thought of as ``half a rook.''
\end{remark}

The $k$-th \dfn{file} in $[a]\times [b]$ is the set $\{(i,j)\in[a]\times[b]:j-i=k\}$. 

\begin{thm} \label{thm:rectfile}
Fix $k$ satisfying $1-a\leq k\leq b-1$. Let $B =\{(i,j) \colon \ j-i=k\}\subseteq \rect{a}{b}$ be the $k$-th file in $[a]\times[b]$.  Then $\sum_{p \in B} \oii{p} \tequiv c$  with $c = a(b-k)/(a+b)$ if $k \geq 0$ and $c = b(a+k)/(a+b)$ if $k < 0$.
\end{thm}

\begin{proof}
We first consider the case $k \geq 0$. 

Let $f(i,j,k)$ denote the number of elements in the sub-rectangle $\rect{i}{j}$ belonging to $B$. (If~$i \leq 0$ or $j\leq 0$, then $f(i,j,k)=0$.) For example, \cref{fig:fijk} depicts these numbers $f(i,j,k)$ when $P=\rect{4}{8}$ and $k=3$.

\begin{figure}
\begin{center}
\includegraphics[height=2.5cm]{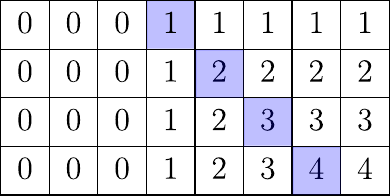}
\end{center} 
\caption{The numbers $f(i,j,k)$ from the proof of \cref{thm:rectfile} with $a=4$, $b=8$, and $k=3$. We have written $f(i,j,k)$ inside of box $(i,j)$ and highlighted the boxes of the $k$-th file in blue.}  \label{fig:fijk}
\end{figure}

With this notation, we have
\begin{eqnarray}
\notag \qquad \sum_{p \in B} \oii{p} \! & = & \sum_{(i,j) \in P} (f(i,j,k) \, \Tout{i,j} - f(i-1,j-1,k) \, \Tin{i,j}) \\
\notag & = & \sum_{(i,j) \in P} (f(i,j,k) - f(i-1,j-1,k)) \, \Tout{i,j}  - \sum_{(i,j) \in P} f(i-1,j-1,k) \, \T{i,j} \\
\notag & = & \!\! \sum_{\substack{ (i,j) \in P, \\ 1-i \leq k \leq j-1}} \!\!\!\! \Tout{i,j}\ - \ \sum_{(i,j) \in P} f(i-1,j-1,k) \, \T{i,j} \\
\label{eqn:rectfileb} & = & \!\! \sum_{\substack{1 \leq i \leq a, \\ \ k+1 \leq j \leq b}} \!\!\!\! \Tout{i,j}  \ - \ \sum_{(i,j) \in P} f(i-1,j-1,k) \, \T{i,j}.
\end{eqnarray}
The first equality above holds by applying \cref{lem:rect_ds1s} to all $p\in B$  (with $i',j'$ replaced by $i,j$). The second equality holds because $\Tin{i,j} = \T{i,j} + \Tout{i,j}$. The third equality holds because the summand $f(i,j,k) - f(i-1,j-1,k)$ is $1$ or $0$ according to whether there is or is not an element of $B$ contained in the \reflectbox{$\mathsf L$}-shape consisting of the elements $(1,j),(2,j),\dots,(i-1,j),(i,j),(i,j-1),\dots,(i,2),(i,1)$, and such an element exists when $1-i \leq k \leq j-1$ (consult \cref{fig:fijk}). The fourth equality follows from rewriting the inequalities governing the indices, using the assumption~$k\geq 0$.

An alternative way to write the right-hand side of \eqref{eqn:rectfileb} is as
\[ \sum_{\substack{(i,j) \in P_R}} \!\! \Tout{i,j} \ - \ \sum_{(i,j) \in P} f(i-1,j-1,k) \, \T{i,j},\]
where we have divided the rectangle $P$ into left and right halves $P_L\coloneqq \{(i,j)\in\rect{a}{b}\colon j \leq k\}$ and $P_R\coloneqq\{(i,j)\in\rect{a}{b}\colon j > k\}$.  We have so far shown that
\begin{equation} \label{eqn:rectfilec}
\sum_{p\in B} \oii{p} \tequiv \sum_{\substack{(i,j) \in P_R}} \!\! \Tout{i,j}.
\end{equation}

Applying \cref{lem:rect_rook_equiv}, \cref{thm:rect_sumone}, and \eqref{eqn:rect_red_rook} yields
\begin{equation}\label{eqn:rectfiled}
ak = \sum_{p \in P_L} \rook{p} \tequiv (a+k) \sum_{(i,j) \in P_L} \Tout{i,j} + k \sum_{(i,j) \in P_R} \Tout{i,j} \ .
\end{equation}
The coefficient of the sum over $P_L$ is $a+k$ because a rook in $P_L$ is attacked by $a+k$ rooks in $P_L$ ($a$ in its column and $k$ in its row); the coefficient of the sum over $P_R$ is only $k$ because a rook in~$P_R$ is attacked by only $k$ rooks in $P_L$.  Likewise, we have
\begin{equation}\label{eqn:rectfilee}
a(b-k) = \sum_{p \in P_R} \rook{p} \tequiv (b-k) \sum_{(i,j) \in P_L} \Tout{i,j} + (a+b-k) \sum_{(i,j) \in P_R} \Tout{i,j}.
\end{equation}
Multiplying \eqref{eqn:rectfilee} by $a+k$ and \eqref{eqn:rectfiled} by $b-k$ and subtracting, we get
\[(a+k)a(b-k) - ak(b-k) \tequiv ((a+k)(a+b-k)-k(b-k)) \sum_{(i,j) \in P_R} \Tout{i,j},\]
which simplifies to
\[a^2 (b-k) \tequiv a(a+b) \sum_{(i,j) \in P_R} \Tout{i,j}.\]
Returning to~\eqref{eqn:rectfilec}, we see from the above that
\[\sum_{p \in B} \oii{p} \tequiv \frac{a(b-k)}{a+b},\]
which proves the claimed result in the case $k \geq 0$.

To handle the case $k < 0$, exchange $a$ and $b$, and replace $k$ by $-k$.
\end{proof}

\begin{cor} \label{cor:rect_oi}
For $P=\rect{a}{b}$, we have $\sum_{p\in P}\oii{p} \tequiv \frac{ab}{2}$.
\end{cor}
\begin{proof}
The statistics in \cref{thm:rectfile} are refinements of order ideal cardinality, i.e., they sum to~$\sum_{p \in P} \oii{p}$. Hence, by invoking \cref{thm:rectfile} and summing over all files, we see that $\sum_{p\in P}\oii{p} \tequiv c$ for some constant $c \in \RR$. One could find $c$ by explicitly evaluating and combining two finite geometric progressions, but a simpler way to find $c$ uses symmetry. Specifically, consider the complementary statistic $\sum_{p \in P} (1 - \oii{p})$. On the one hand, the sum of the statistics $\sum_{p\in P}\oii{p}$ and~$\sum_{p \in P} (1 - \oii{p})$ is the constant statistic $\sum_{p\in P} 1 = ab$, so $\sum_{p\in P}\oii{p} \tequiv ab-c$. On the other hand, the involution on order ideals $I$ that rotates the complement of $I$ by $180^\circ$ swaps the two statistics, implying $ab-c=c$. Hence, $c=ab/2$. (In the proof of \cref{thm:min_oi} below, we will explain that this same symmetry argument works for any self-dual poset.)
\end{proof}

That the order ideal cardinality statistic, as well as its file refinements, are homomesic under rowmotion for the rectangle was again first proved by Propp and Roby~\cite[Theorem~23]{propp2015homomesy}. The stronger results that these statistics are $\tequiv \const$ (i.e., \cref{thm:rectfile} and \cref{cor:rect_oi}) are new. One of our main tools in establishing these results was \cref{lem:rect_ds1s}, which is also new. 

\subsection{The shifted staircase}

The $n$-th \dfn{shifted staircase} poset, which we denote by~$\sstair{n}$, is the poset realized as the subset $\{(i,j): 1\leq i\leq j\leq n\}$ of~$\quadrant$. Our notation for this poset comes from the fact that it is also the quotient of $\rect{n}{n}$ by the involutive poset automorphism $(i,j)\mapsto (j,i)$.

We use $\iota\colon \J(\sstair{n})\to \J(\rect{n}{n})$ to denote the map sending each order ideal to its preimage under the quotient map; that is, $\iota(I) \coloneqq \{(i,j), (j,i)\colon (i,j)\in I\}$. Observe that $\iota(I)$ for $I\in \J(\sstair{n})$ is indeed always an order ideal in $\J(\rect{n}{n})$. Note importantly that all the statistics on $\J(P)$ we care about are invariant under $\iota$. In other words, 
\begin{align*}
\Tin{i,j}(I) &= \Tin{i,j}(\iota(I)) =  \Tin{j,i}(\iota(I)); \\
\Tout{i,j}(I) &= \Tout{i,j}(\iota(I)) =  \Tout{j,i}(\iota(I)); \\
\oii{i,j}(I) &= \oii{i,j}(\iota(I)) =  \oii{j,i}(\iota(I))
\end{align*}
for all $(i,j)\in\sstair{n}$ and $I\in \J(\sstair{n})$. In other words, if we take any identity satisfied by the $\Tin{i,j}$, $\Tout{i,j}$, and $\oii{i,j}$ for $\rect{n}{n}$, it will also hold as an identity for $\sstair{n}$, provided we equate $(i,j)$ and $(j,i)$. As we shall see, $\iota$ therefore allows us to easily transfer many results from $\rect{n}{n}$ to $\sstair{n}$.

From now on throughout this subsection, we fix $P=\sstair{n}$. 

\subsubsection{Rooks}

For $(i,j)\in P$, define the \dfn{rook} $\rook{i,j}\colon \J(P)\to\RR$ by
\begin{equation} \label{eqn:sstair_rook}
\rook{i,j} \ \coloneqq \sum_{i' \leq i, \ j' \leq j} \Tin{i',j'} \ \ - \sum_{\substack{i' < i,\  j' < j,\\ i' < j'}} \Tout{i',j'} \ \ + \sum_{i' \geq i, \ j' \geq j} \Tout{i',j'} \ \ - \sum_{\substack{i' > i, \ j' > j,\\ i'<j' }} \Tin{i',j'}
\end{equation}
and the \dfn{reduced rook} $\rrook{i,j}\colon \J(P)\to\RR$ by
\begin{equation} \label{eqn:sstair_red_rook}
\rrook{i,j} \ \coloneqq \ \sum_{j'} \Tout{i,j'} \ + \ \sum_{i'} \Tout{i',j} \ + \ \sum_{i' < i}\Tout{i',i'} \ + \ \sum_{j' > j}\Tout{j',j'}.
\end{equation}
(Observe how we use the same notation for rooks for different posets, allowing context to differentiate the different $\rook{i,j}$.) The sum in~\eqref{eqn:sstair_rook} defining $\rook{i,j}$ can be represented pictorially as in the left image in \cref{fig:sstair_rook}. 

\begin{figure}
\begin{center}
\includegraphics[height=5.595cm]{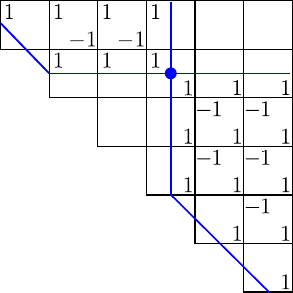}\qquad\qquad\qquad\qquad\includegraphics[height=5.595cm]{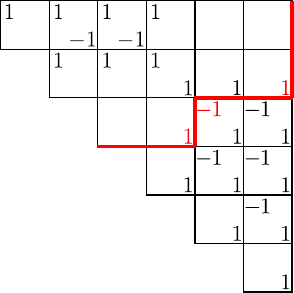}
\end{center}
\caption{On the left is an illustration of the rook $R_{2,4}$ on the shifted staircase $([6]\times[6])/\mathfrak{S}_2$. On the right is an illustration of the proof of \cref{thm:sstair_sumone}. The sum along the red boundary path is $+1-1+1=1$.} \label{fig:sstair_rook}
\end{figure}

It is again immediate from the definitions \eqref{eqn:sstair_rook} and \eqref{eqn:sstair_red_rook} of $\rook{i,j}$ and $\rrook{i,j}$ that we have the following relation between these statistics.

\begin{lemma} \label{lem:sstair_rook_equiv}
For all $(i,j)\in \sstair{n}$, $\rook{i,j} \tequiv \rrook{i,j}$.
\end{lemma}

The blue lines in the left image in \cref{fig:sstair_rook} show the boxes ``attacked'' by $\rrook{i,j}$ and $\rook{i,j}$ (with the box~$(i,j)$ itself again being attacked twice). Because they are analogous to the rooks on the rectangle, we still call these statistics rooks even though they no longer attack the way a rook chess piece attacks.

Furthermore, we again have the same fundamental fact as with the rectangle that the $\rook{i,j}$ are identically equal to one.

\begin{thm} \label{thm:sstair_sumone}
For all $(i,j)\in \sstair{n}$, $\rook{i,j}=1$.
\end{thm}

The proof of \cref{thm:sstair_sumone} is analogous to the proof of \cref{thm:rect_sumone}: we consider the southeast boundary of the boxes in any order ideal $I\in \J(P)$ and observe that the nonzero coefficients adjacent to this path are $1, -1, 1,\ldots, 1$, so always sum to $1$. The right image in \cref{fig:sstair_rook} shows an example of this.

Rooks for the shifted staircase were introduced by Hopkins~\cite{hopkins2017cde}, and \cref{lem:sstair_rook_equiv} and \cref{thm:sstair_sumone} were first proved there (see~\cite[Lemma~4.3 and Lemma~4.4]{hopkins2017cde}). (In fact, rooks were defined  in~\cite{hopkins2017cde} not just for the shifted staircase, but for any shifted Young diagram shape.)

\subsubsection{Antichains}

Now we consider statistics that are in $\Span\{\Tout{p}:p\in P\}$ and satisfy $\stat{f}\tequiv \const$. The reduced rooks $\rrook{i,j}$ are of this form; but we want to see if other, natural such $\stat{f}$ can be realized as linear combinations of the $\rrook{i,j}$.

\begin{lemma} \label{lem:sstair_diag}
We have $\sum_{i}\Tout{i,i} \tequiv \frac{1}{2}$ for $P=\sstair{n}$.
\end{lemma}
\begin{proof}
It is routine to verify that
\[2\cdot \!\!\sum_{1\leq i \leq n}\Tout{i,i} = \sum_{1\leq i \leq n}\rrook{i,i} \ -  \sum_{1\leq i \leq n-1}\rrook{i,i+1}\]
in a way similar to what we have already seen in other arguments, keeping track of how many times each box is attacked by various combinations of rooks. Since \cref{lem:sstair_rook_equiv} and \cref{thm:sstair_sumone} tell us that $\rrook{i,j}\tequiv 1$ for every $(i,j)\in P$, we conclude that~$\sum_{i}\Tout{i,i} \tequiv \frac{1}{2}$, as claimed.
\end{proof}

\begin{lemma} \label{lem:sstair_els}
Fix $1\leq i \leq n$. Then we have $\sum_{j \leq i} \Tout{j,i} + \sum_{i < j} \Tout{i,j} \tequiv \frac{1}{2}$ for $P=\sstair{n}$.
\end{lemma}
\begin{proof}
It is routine to verify that
\[ \sum_{1 \leq j \leq i} \Tout{j,i} \ + \sum_{i < j \leq n} \Tout{i,j} = \rrook{i,i} \ - \sum_{1\leq j \leq n}\Tout{j,j}.\]
\Cref{lem:sstair_rook_equiv} and \cref{thm:sstair_sumone} tell us that~$\rrook{i,j}\tequiv 1$ for every $(i,j)\in P$, and \cref{lem:sstair_diag} tells us that~$\sum_{i}\Tout{i,i}\tequiv \frac{1}{2}$; hence, $\sum_{j \leq i} \Tout{j,i} + \sum_{i < j} \Tout{i,j} \tequiv \frac{1}{2}$.
\end{proof}

Observe that the statistics from \cref{lem:sstair_diag,lem:sstair_els} refine (twice) the antichain cardinality statistic. That is, we have
\[\sum_{1\leq i \leq n}\Tout{i,i} \ + \ \sum_{i=1}^{n} \left( \sum_{1\leq j \leq i} \Tout{j,i} \ + \sum_{i < j \leq n} \Tout{i,j} \right)= 2\cdot \sum_{p\in P} \Tout{p}.\]
Hence we deduce the following corollary.

\begin{thm} \label{thm:sstair_a}
For $P=\sstair{n}$, we have $\sum_{p\in P} \Tout{p} \tequiv \frac{n+1}{4}$.
\end{thm}

That the antichain cardinality statistic is homomesic under rowmotion for the shifted staircase was first proved by Rush and Wang~\cite[Theorem~1.4]{rush2015homomesy} (who, as we will mention below, in fact addressed all minuscule posets). The stronger result that it is $\tequiv \const$ (i.e., \cref{thm:sstair_a}) was first proved by Hopkins~\cite[Theorem~4.2]{hopkins2017cde}, with the same proof we have presented above.

\subsubsection{Order ideals}

Finally, we consider those statistics $\stat{f}$ that belong to $\Span\{\oii{p}:p\in P\}$ and satisfy $\stat{f}\tequiv \const$. Here, $\iota$ will take care of everything for us.

\begin{thm} \label{thm:sstair_file}
Fix $0\leq k \leq n-1$, and let $B\coloneqq \{(i,j)\colon j-i = k\}\subseteq \sstair{n}$ be the $k$-th file of the shifted staircase. Then $\sum_{p\in B}\oii{p}\tequiv (n-k)/2$.
\end{thm}

\begin{proof}
We have $\sum_{p\in B}\oii{p}(I) = \sum_{p\in B}\oii{p}(\iota(I))$ for all $I\in \J(\sstair{n})$. So the claim follows directly from \cref{thm:rectfile}.
\end{proof}

Summing up these file refinements of order ideal cardinality, we obtain:

\begin{cor} \label{cor:sstair_oi}
For $P=\sstair{n}$, we have $\sum_{p\in P} \oii{p} \tequiv \frac{n(n+1)}{4}$.
\end{cor}

That the order ideal cardinality statistic, and its file refinements, are homomesic under rowmotion for the shifted staircase was again first proved by Rush and Wang~\cite[Corollary~1.3]{rush2015homomesy} (and again they in fact dealt with all minuscule posets). The stronger results that they are $\tequiv \const$ (i.e., \cref{cor:sstair_oi} and \cref{thm:sstair_file}) are new.

\begin{remark}
Just as with the rectangle, there are ``half a rook'' identities for the shifted staircase, such as
\[\oii{p} = \sum_{i' \geq i, \ j' \geq j} \Tout{i',j'} \ \ - \sum_{\substack{i' > i, \ j' > j,\\ i'<j' }} \Tin{i',j'}\]
for any $p=(i,j)\in\sstair{n}$. But we will not need these.
\end{remark}

\subsection{Other minuscule posets}

The minuscule posets are a class of posets arising from the representation theory of simple Lie algebras, which exhibit remarkable combinatorial properties. Briefly, each minuscule poset arises from the choice of a Dynkin diagram $\Gamma$ and a \emph{minuscule node} $i$ of~$\Gamma$. See~\cite{proctor1984bruhat, stembridge1994minuscule} for details on minuscule posets and the algebraic context in which they arise. The minuscule posets have been classified, and for our purposes, it is enough to review this classification without explaining the formal definition of a minuscule poset.

We have already encountered the two most prominent examples of minuscule posets, which are the rectangles $\rect{a}{b}$ and the shifted staircases $\sstair{n}$. Beyond these, the other minuscule posets include one very simple infinite family $\dtd{n}$ for $n\geq 2$, the so-called \dfn{double-tailed diamonds}, and two exceptional posets $\esixmin$ and $\esevmin$. These are depicted in~\cref{fig:minuscule_posets}. (Note that $\dtd{2}\simeq \rect{2}{2}$ and $\dtd{3}\simeq \sstair{3}$.)

\begin{figure}
\begin{center}
 \includegraphics[height=3.9cm]{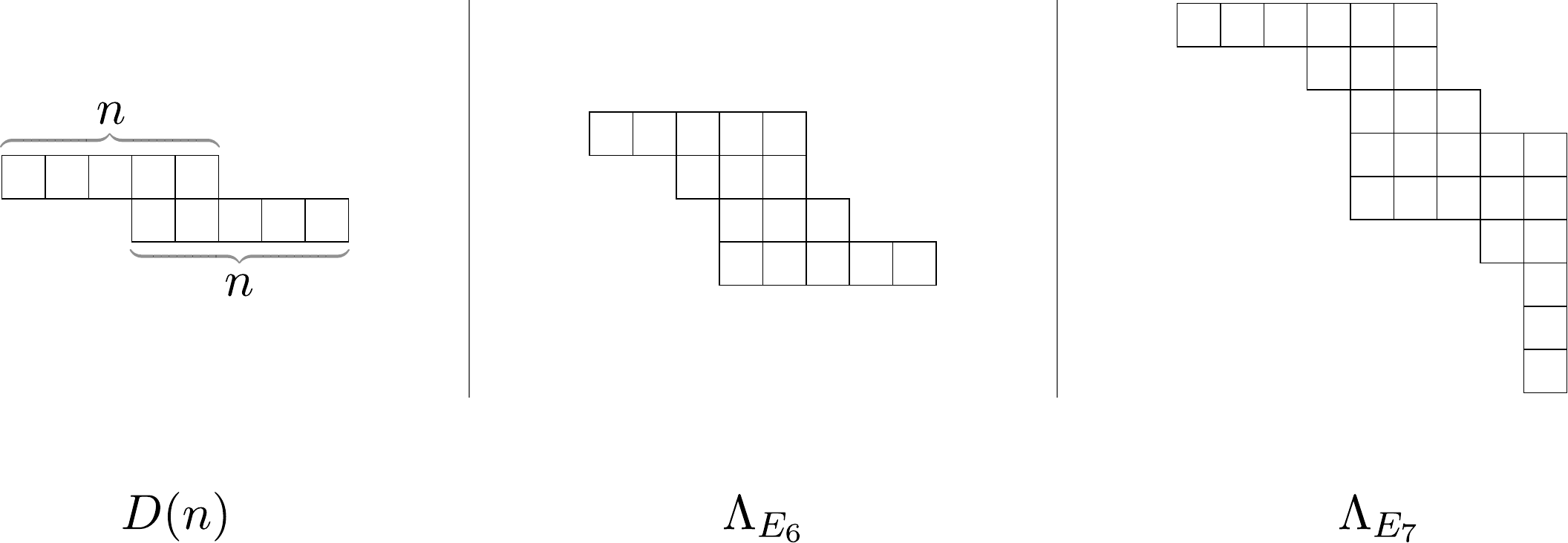}
\end{center} 
\caption{The other minuscule posets beyond the rectangle and shifted staircase.}
\label{fig:minuscule_posets}
\end{figure}

A poset $P$ is \dfn{graded} if all maximal chains in $P$ have the same length. If~$P$ is graded, then it is ranked, and in this case, $\rk(P)$ is the length of any maximal chain. Also, recall that the \dfn{dual poset} $P^*$ is the poset on the same elements as $P$ but with $x \leq_P y$ if and only if $y \leq_{P^*} x$. We say $P$ is \dfn{self-dual} if it is isomorphic to $P^*$. 

All minuscule posets are graded (hence, ranked), and all are also self-dual.

The minuscule posets are known to exhibit good behavior of rowmotion (see~\cite{rush2013orbits}). We will now show that the toggleability statistics technique also works well for the minuscule posets. In particular, we will show that both the antichain and order ideal cardinality statistics are $\tequiv \const$ for all minuscule posets. Since we have already shown this for the rectangle and shifted staircase, there is not much work left to do.

\begin{thm} \label{thm:min_a}
Let $P$ be a minuscule poset. Then $\sum_{p\in P}\Tout{p}\tequiv \frac{\#P}{\rk(P)+2}$.
\end{thm}

\begin{proof}
First, let us observe that if $P$ is any graded poset and $\sum_{p\in P}\Tout{p}\tequiv c$ for some $c\in\RR$, then we must have $c=\frac{\#P}{\rk(P)+2}$. Indeed, \cref{prop:homo} tells us that $c$ is the average antichain cardinality along any rowmotion orbit. And graded posets have one distinguished rowmotion orbit 
\[\{ \{p\in P\colon \rk(p)\leq i\}\colon i = -1,0,1,\ldots,\rk(P)\},\] 
which is easily seen to have average antichain cardinality equal to $\frac{\#P}{\rk(P)+2}$.

Now let $P$ be a minuscule poset. We will prove that there exists \emph{some} constant $c$ such that $\sum_{p\in P}\Tout{p}\tequiv c$; since $P$ is graded, it will then follow from the preceding paragraph that $c$ must actually be $\frac{\#P}{\rk(P)+2}$. We have already shown this for the rectangle and the shifted staircase (\cref{thm:aboapb} and \cref{thm:sstair_a}), so we need only address $\dtd{n}$, $\esixmin$, and $\esevmin$. 

First, consider $P=\dtd{n}$. This poset consists of an ``initial tail'' $x_1<\cdots<x_{n-1}$, two incomparable elements $y_1,y_2$ in the middle, and a ``final tail'' $z_{n-1}<\cdots<z_{1}$ such that $x_i<y_j<z_k$ for all $i,j,k$. We claim that
\[\sum_{p\in P} \Tout{p} = 1-\sum_{i=1}^{n-1}\T{x_i}-\frac{1}{2}\T{y_1}-\frac{1}{2}\T{y_2}.\]
Indeed, to verify this is just a matter of checking the few possible cases: the empty order ideal; an order ideal generated by an $x_i$; an order ideal generated by a $z_i$; an order ideal generated by a $y_i$; and the order ideal generated by $\{y_1,y_2\}$.

Now consider $P=\esixmin$ or $\esevmin$. In this case, we checked by computer that the relevant system of equations has a solution (and the Sage code worksheet referenced in the acknowledgments includes this verification).
\end{proof}

\begin{thm} \label{thm:min_oi}
Let $P$ be a minuscule poset. Then $\sum_{p\in P}\oii{p}\tequiv \frac{\#P}{2}$.
\end{thm}

\begin{proof}
First, let us observe that if $P$ is any self-dual poset and $\sum_{p\in P}\oii{p}\tequiv c$ for some $c\in\RR$, then we must have $c=\frac{\#P}{2}$. Indeed, $c$ must be the average order ideal cardinality over all $I\in \J(P)$, and we can see that this average has to be $\frac{\#P}{2}$ in the following way. For $I\in \J(P)$, set $I^* \coloneqq \omega(P\setminus I)\in\J(P)$, where $\omega\colon P\to P^*$ is the self-duality. Then $I\mapsto I^*$ is an involution, and $\#I + \#I^*=\#P$ for all $I$, which implies the average of $\#I$ is $\frac{\#P}{2}$.

Now let $P$ be a minuscule poset. Since $P$ is self-dual, we know by the preceding paragraph that $\frac{\#P}{2}$ is the right constant for $\sum_{p\in P}\oii{p}$ to be $\tequiv$ to. Thus, we only need to show that there is some $c \in \RR$ such that $\sum_{p\in P}\oii{p}\tequiv c$. We have already shown this for the rectangle and the shifted staircase (\cref{cor:rect_oi} and \cref{cor:sstair_oi}), so we need only address $\dtd{n}$, $\esixmin$, and $\esevmin$. 

First consider $P=\dtd{n}$. Let $x_1<\cdots<x_{n-1}$ be the elements in the initial tail, $y_1,y_2$ be the two incomparable elements in the middle, and $z_{n-1}<\cdots<z_{1}$ the elements in the final tail. Then we claim that
\begin{align*}
\sum_{p\in P}\oii{p} = \, &n-\sum_{i=1}^{n-1}(n+(n-1)+\cdots+(n-i+1))\T{x_i}-\frac{(n+(n-1)+\cdots+1)}{2}\T{y_1}\\
&-\frac{(n+(n-1)+\cdots+1)}{2}\T{y_2}-\sum_{i=1}^{n-1}(n+(n-1)+\cdots+(n-i+1))\T{z_i}.
\end{align*}
As we mentioned in the proof of \cref{thm:min_a}, the order ideals in $\J(\dtd{n})$ are divided into a very small number of cases, and we can verify this equation by checking these cases.

Now consider $P=\esixmin$ or $\esevmin$. In this case, we checked by computer that the relevant system of equations has a solution (and the Sage code worksheet referenced in the acknowledgments includes this verification).
\end{proof}

That the antichain cardinality statistic is homomesic under rowmotion for any minuscule poset~$P$ was first proved by Rush and Wang~\cite[Theorem~1.4]{rush2015homomesy}. The stronger result that it is $\tequiv \const$ (i.e., \cref{thm:min_a}) was first proved by Hopkins~\cite[Theorem~5.2]{hopkins2017cde}, in a case-by-case fashion exactly as we have done above. Rush~\cite[Theorem~1.5]{rush2016cde} subsequently gave a \emph{uniform} proof that $\sum_{p\in P}\Tout{p}\tequiv \const$ for all minuscule posets, building on his earlier work with coauthors~\cite{rush2013orbits, rush2015homomesy} studying toggling for minuscule posets. That the order ideal cardinality statistic is homomesic under rowmotion for any minuscule poset $P$ was again first proved by Rush and Wang~\cite[Corollary~1.3]{rush2015homomesy}. The stronger result that it is $\tequiv \const$ (i.e., \cref{thm:min_oi}) is new.

\begin{remark}
For the double-tailed diamond $P=\dtd{n}$, $\Span\{\T{p}\colon p \in P\}$ is the entire space of functions $\J(P)\to\RR$ that are $0$-mesic under rowmotion. This will follow from \Cref{thm:anti_span}. 
\end{remark}

\subsection{The Type A root poset}

The \dfn{root poset of Type $A_n$}, which we denote by~$\arootp{n}$, is typically defined as the set of positive roots in the Type $A_n$ root system under the natural partial order on these roots. We will review root posets in general below, but for the moment, we will simply realize $\arootp{n}$ as the subset $\{(i,j): 1\leq i,j\leq n\text{ and }i+j\geq n+1\}$ of $\quadrant$. Observe that $\arootp{n}\subseteq \rect{n}{n}$.

Throughout this subsection, we fix $P=\arootp{n}$.

\subsubsection{Rooks}

As with the rectangle and shifted staircase, we begin with rooks for the Type A root poset.  So for $1 \leq i \leq n$, define the \dfn{rook} $\rook{i}\colon \J(P)\to\RR$ by
\begin{equation} \label{eqn:a_rook}
\rook{i} \ \coloneqq \ \Tin{i,n+1-i} \ \ + \sum_{i' \geq i, \, j'\geq n+1-i} \Tout{i',j'} \ \ - \sum_{i' > i, \, j' > n+1-i} \Tin{i',j'}
\end{equation}
and the \dfn{reduced rook} $\rrook{i}\colon \J(P)\to\RR$ by
\begin{equation} \label{eqn:a_red_rook}
\rrook{i} \ \coloneqq \sum_{j'\geq n+1-i} \Tout{i,j'}\ +\ \sum_{i' \geq i} \Tout{i',n+1-i}.
\end{equation}
The sum~\eqref{eqn:a_rook} defining $\rook{i}$ can be represented pictorially as in~\cref{fig:a_rook}. Note that for $\arootp{n}$, unlike for the rectangle and the shifted staircase, we do not get a rook for every box, only for the boxes on the main anti-diagonal.

\begin{figure}
\begin{center}
\includegraphics[height=5.595cm]{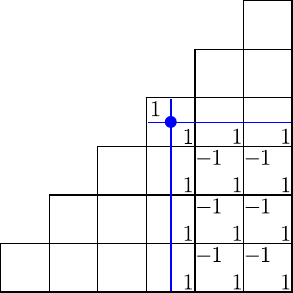}
\end{center}
\caption{A rook for $\arootp{n}$ with $n=6$ and $i=3$.} \label{fig:a_rook}
\end{figure}

As always, we have the following fundamental facts about the rooks, whose proofs are analogous to what we have seen above.

\begin{lemma} \label{lem:a_rook_equiv}
For $P=\arootp{n}$, $\rook{i} \tequiv \rrook{i}$ for all $i=1,\ldots,n$.
\end{lemma}

\begin{thm} \label{thm:a_sumone}
For $P=\arootp{n}$, $\rook{i} = 1$ for all $i=1,\ldots,n$.
\end{thm}

As mentioned earlier, rooks for all (possibly skew) Young diagram shapes, including $\arootp{n}$, were introduced by Chan, Haddadan, Hopkins, and Moci~\cite{chan2017expected}. In particular, \cref{lem:a_rook_equiv} and \cref{thm:a_sumone} were first proved in that paper (see~\cite[Lemma~3.5, Theorem~3.4]{chan2017expected}).

\subsubsection{Antichains} 

Our goal is now to show that the antichain cardinality statistic is~$\tequiv \const$ for~$\arootp{n}$. This is quite easy with what we already know because the reduced rooks $\rrook{i}$ refine (twice) the antichain cardinality statistic. That is,
\[ \sum_{i=1}^{n} \rrook{i} = 2 \cdot \sum_{p \in P} \Tout{p}.\]
We know that $\rrook{i}\tequiv 1$ for all $i$ thanks to \cref{lem:a_rook_equiv} and \cref{thm:a_sumone}. Hence we conclude:

\begin{thm} \label{thm:a_ac}
For $P=\arootp{n}$, $\sum_{p \in P} \Tout{p}\tequiv \frac{n}{2}$.
\end{thm}

That the antichain cardinality statistic is homomesic for rowmotion for the Type~A root poset was first proved by Armstrong, Stump, and Thomas~\cite[Theorem~1.2(iii)]{armstrong2013uniform} (and in fact, they showed this for \emph{all} root posets, as we will discuss in a moment). The homomesy of the $\rrook{i}$ statistics under rowmotion was observed recently by Hopkins and Joseph~\cite[Corollary~4.15]{hopkins2020}. The stronger result that these statistics are $\tequiv \const$ follows from the work of Chan, Haddadan, Hopkins, and Moci~\cite{chan2017expected}, in exactly the way we have presented here.

\subsubsection{Order ideals} 

Unlike for the prior posets, the order ideal cardinality statistic is \emph{not} $\tequiv \const$ for the Type A root poset. However, a variation of it is. For a ranked poset $P$, let us call $\sum_{p\in P}(-1)^{\rk(p)}\oii{p}$ the \dfn{rank-alternating order ideal cardinality}. We will show that the rank-alternating order ideal cardinality is $\tequiv \const$ for the Type~A root poset.

\begin{remark}
For both the rectangle and the shifted staircase, the rank-alternating order ideal cardinality statistic is a linear combination of the file refinements of order ideal cardinality. Therefore, in these cases, it is also $\tequiv \const$ by \cref{thm:rectfile} and \cref{thm:sstair_file}.
\end{remark}

To relate the $\oii{p}$ to the toggleability statistics, we need a version of \cref{lem:rect_ds1s} for $\arootp{n}$. In fact, \cref{lem:rect_ds1s} holds verbatim for $\arootp{n}$, and with the same proof.\footnote{The identity~\eqref{eqn:rectfilea} will hold for any poset $P$ and element $p\in P$ for which the principal \dfn{order filter} (i.e., upward-closed subset) generated by $p$ is isomorphic to a rectangle poset. Similarly, the dual identity~\eqref{eqn:rectfilea_alt} will hold whenever the principal order ideal generated by $p$ is isomorphic to a rectangle poset. In this way, one can obtain many identities relating the statistics $\oii{p}$, $\Tout{p}$, $\Tin{p}$, and $1$ for many posets, especially for Young diagram shape posets.} 

\begin{lemma} \label{lem:a_ds1s}
For any $p=(i,j) \in \arootp{n}$, we have
\[ \oii{p} \ = \sum_{i' \geq i, \ j' \geq j} \Tout{i',j'} \ - \sum_{i' > i, \ j' > j} \Tin{i',j'}.\]
\end{lemma}

With \cref{lem:a_ds1s} in hand, we can prove that $\sum_{p\in P}(-1)^{\rk(p)}\oii{p}\tequiv \const$. As the reader may expect by now, we will first do this for some refinements of the rank-alternating order ideal cardinality statistic.

\begin{thm} \label{thm:a_oi_refined}
Fix $k\in \{n-1, n-3, \ldots,-(n-3),-(n-1)\}$. Then for $P=\arootp{n}$, we have
\[ 2\cdot\sum_{j-i=k} \oii{i,j} \ - \sum_{j-i=k-1} \oii{i,j} \ - \sum_{j-i=k+1} \oii{i,j} \ \tequiv \ 1.\]
\end{thm}
\begin{proof}
\Cref{lem:a_ds1s} says that for any box $(i,j)\in \arootp{n}$, we have
\begin{equation} \label{eqn:a_1ds_helper}
2\cdot \oii{i,j} - \oii{i+1,j} -\oii{i,j+1} = 2\cdot ( \Tout{i,j} - \Tin{i+1,j+1}) + \sum_{i' > i} (\Tout{i',j}-\Tin{i'+1,j+1}) + \sum_{j' > j} (\Tout{i,j'}-\Tin{i+1,j'+1}).
\end{equation}

Now fix $k$ as in the statement of the theorem, and set 
\[\stat{f}_k \coloneqq 2\cdot\sum_{j-i=k} \oii{i,j} \ - \sum_{j-i=k-1} \oii{i,j} \ - \sum_{j-i=k+1}\oii{i,j}.\] 
By summing \eqref{eqn:a_1ds_helper} over all $(i,j)\in\arootp{n}$ with $j-i=k$, we obtain
\[\stat{f}_k \ = \sum_{j-i=k}(2\cdot \oii{i,j} - \oii{i+1,j} -\oii{i,j+1}) = \rook{\frac{n+1-k}{2}} - \sum_{j-i=k} \T{i,j},\]
where we recall the definition~\eqref{eqn:a_rook} of $\rook{i}$ as a sum of toggleability statistics. By \cref{thm:a_sumone}, we know $\rook{i} = 1$ for all $i$, so indeed $\stat{f}_k \tequiv 1$.
\end{proof}

The statistics appearing in \cref{thm:a_oi_refined} refine (twice) the rank-alternating order ideal cardinality. In other words, if we write $\stat{f}_k$ for these statistics, then
\[\stat{f}_{n-1} +\stat{f}_{n-3} +\cdots + \stat{f}_{-(n-3)} + \stat{f}_{-(n-1)} = 2 \cdot \sum_{p\in P} (-1)^{\rk(p)} \oii{p}.\]
Hence, we obtain the following corollary.

\begin{cor} \label{cor:a_oi}
For $P=\arootp{n}$, we have $\sum_{p\in P} (-1)^{\rk(p)} \oii{p} \tequiv \frac{n}{2}$.
\end{cor}

That the rank-alternating order ideal cardinality is homomesic under rowmotion for $\arootp{n}$ was first proved by Haddadan~\cite[Corollary~2.3]{haddadan2021homomesy}. The stronger result that it is $\tequiv \const$ (i.e., \cref{cor:a_oi}) is new. The refinements in \cref{thm:a_oi_refined} are also apparently new.

\begin{remark} \label{rem:bsv}
Recently, Bernstein, Striker, and Vorland~\cite[Conjecture~4.35]{bernstein2021pstrict} conjectured that the rank-alternating order ideal cardinality homomesy for $\arootp{n}$ should extend to the piecewise-linear realm. In \cref{sec:pl_birational}, we will explain how, whenever it works, the toggleability statistics technique automatically yields piecewise-linear homomesies as well. However, there is a catch: Bernstein, Striker, and Vorland's homomesy conjecture actually concerns piecewise-linear \emph{promotion}, not rowmotion. Nevertheless, as we mentioned in \cref{rem:promo_2}, the recombination method of Einstein--Propp~\cite[\S6]{einstein2018combinatorial} allows one to transfer homomesies from rowmotion to promotion, provided that the statistic in question is a linear combination of the $\oii{p}$'s (which the rank-alternating order ideal cardinality certainly is). Moreover, the recombination method works equally well at the PL level. Hence, we \emph{are} able to affirmatively resolve the conjecture of Bernstein--Striker--Vorland using toggleability statistics in the way we have just sketched. Giving all the details for this argument would require us to formally introduce promotion, recombination, etc., which unfortunately is beyond the scope of the present article. 
\end{remark}

\subsection{The Type B root poset}

The \dfn{root poset of Type $B_n$}, denoted~$\brootp{n}$, will for our purposes be realized as $\{(i,j): 1\leq i\leq j\leq 2n-1\textrm{ and }i+j\geq 2n\} \subseteq \quadrant$. Observe the inclusions $\brootp{n}\subseteq \arootp{2n-1} \subseteq \rect{2n-1}{2n-1}$. 

What is more, we have $\brootp{n}=\arootp{2n-1}/\twogrp$, i.e., $\brootp{n}$ is the quotient of $\arootp{2n-1}$ by the involutive poset automorphism $(i,j)\mapsto (j,i)$. As with the rectangle and shifted staircase, we use $\iota\colon \J(\brootp{n})\to\J(\arootp{2n-1})$ to denote the map sending each order ideal to its preimage under the quotient map; that is, $\iota(I)\coloneqq \{(i,j),(j,i)\colon (i,j) \in I\}$. Once again, because it respects the relevant statistics, the map $\iota$ will let us transfer many results from $\arootp{2n-1}$ to $\brootp{n}$.

Throughout this subsection, we fix $P=\brootp{n}$.

\subsubsection{Rooks} 

As always, we start by defining rooks for $\brootp{n}$.

For $1\leq i \leq n$, define the \dfn{rook} $\rook{i}\colon \J(P)\to\RR$ by
\begin{equation} \label{eqn:b_rook}
\rook{i} \ \coloneqq \ \Tin{i,2n-i} \ \ + \sum_{\substack{i' \geq i, \, j'\geq 2n-i,\\ j\geq i}} \Tout{i',j'} \ \ - \sum_{\substack{i' > i, \, j' > 2n-i,\\ j>i}} \Tin{i',j'}
\end{equation}
and the \dfn{reduced rook} $\rrook{i}\colon \J(P)\to\RR$ by
\begin{equation} \label{eqn:b_red_rook}
\rrook{i} \ \coloneqq \sum_{j'\geq 2n-i} \Tout{i,j'} \ + \ \sum_{i' \geq i} \Tout{i',2n-i}\ + \sum_{j'>2n-i}\Tout{j',j'}.
\end{equation}
The sum~\eqref{eqn:b_rook} defining $\rook{i}$ can be represented pictorially as in the left image in \cref{fig:b_rook}.

\begin{figure}
\begin{center}
\includegraphics[height=6.523cm]{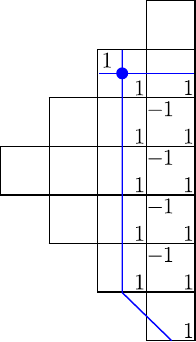} \qquad\qquad\qquad\qquad \includegraphics[height=6.523cm]{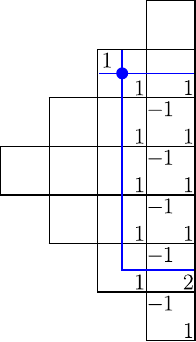}
\end{center}
\caption{A rook (left) and a variant rook (right) for $\brootp{n}$. Here, $n=4$ and $i=2$.} \label{fig:b_rook}
\end{figure}

As before, we have the following fundamental facts about the rooks.

\begin{lemma} \label{lem:b_rook_equiv}
For $P=\brootp{n}$, $\rook{i} \tequiv \rrook{i}$ for all $i=1,\ldots,n$.
\end{lemma}

\begin{thm} \label{thm:b_sumone}
For $P=\brootp{n}$, $\rook{i} = 1$ for all $i=1,\ldots,n$.
\end{thm}

As mentioned earlier, rooks for all shifted Young diagram shapes, including $\brootp{n}$, were introduced by Hopkins~\cite{hopkins2017cde}. In particular, \cref{lem:b_rook_equiv} and \cref{thm:b_sumone} were first proved in there (see~\cite[Lemma~4.3 and Lemma~4.4]{hopkins2017cde}).

We will also make use of variant rooks for $\brootp{n}$, defined via $\iota$. That is, for $1\leq i \leq n$, we define $\varrook{i}, \rvarrook{i} \colon \J(P)\to\RR$ by setting
\begin{align*}
\varrook{i}(I) &\coloneqq \rook{i}(\iota(I)); \\
\rvarrook{i}(I) &\coloneqq \rrook{i}(\iota(I)),
\end{align*}
for all $I \in \J(P)$, where the $\rook{i}$ and $\rrook{i}$ on the right-hand sides above are defined on $\arootp{2n-1}$ via~\eqref{eqn:a_rook} and~\eqref{eqn:a_red_rook}. 

The right image in~\cref{fig:b_rook} depicts one example of a variant rook. It is certainly possible to write a definition of $\varrook{i}$ as an explicit sum of toggleability statistics, but it is a bit cumbersome to do so. However, we do find it convenient to record that
\begin{equation} \label{eqn:b_red_var_rook}
\rvarrook{i} \ = \sum_{j'\geq 2n-i} \Tout{i,j'}\ + \ \sum_{i' \geq i} \Tout{i',2n-i} \ +\sum_{j'>2n-i}\Tout{2n-i,j'}.
\end{equation}

For the variant rooks, we have the same fundamental facts, which now follow from~\cref{lem:a_rook_equiv} and~\cref{thm:a_sumone}.

\begin{lemma} \label{lem:b_var_rook_equiv}
For $P=\brootp{n}$, $\varrook{i} \tequiv \rvarrook{i}$ for all $i=1,\ldots,n$.
\end{lemma}

\begin{thm} \label{thm:b_var_sumone}
For $P=\brootp{n}$, $\varrook{i} = 1$ for all $i=1,\ldots,n$.
\end{thm}

Unlike with the shifted staircase, it is no longer the case that the $\rvarrook{i}$ are linear combinations of the~$\rrook{i}$, which is why it is useful to also consider them. (But note that $\rvarrook{1}=\rrook{1}$.)

\subsubsection{Antichains} 

Our goal is now to show that the antichain cardinality statistic is $\tequiv \const$ for~$\brootp{n}$ by taking linear combinations of the $\rrook{i}$ and $\rvarrook{i}$. 

\begin{lemma} \label{lem:b_ac}
We have $\sum_{i}\Tout{i,i} \tequiv \frac{1}{2}$ for $P=\brootp{n}$.
\end{lemma}
\begin{proof}
It is routine to verify
\[ 2\cdot \rrook{n} - \rvarrook{n} = 2\cdot \sum_{i}\Tout{i,i}.\]
We know $\rrook{n}\tequiv 1$ thanks to \cref{lem:b_rook_equiv} and~\cref{thm:b_sumone}, and we know $\rvarrook{n}\tequiv 1$ thanks to \cref{lem:b_var_rook_equiv} and~\cref{thm:b_var_sumone}, so we conclude $\sum_{i}\Tout{i,i} \tequiv \frac{1}{2}$.
\end{proof}

\begin{thm} \label{thm:b_ac}
We have $\sum_{p\in P}\Tout{p} \tequiv \frac{n}{2}$ for $P=\brootp{n}$.
\end{thm}
\begin{proof}
It is routine to verify
\[ \sum_{i}\Tout{i,i} + \frac{1}{2}\rvarrook{n} + \sum_{i=1}^{n-1} \rvarrook{i} = 2\cdot\sum_{p\in P}\Tout{p}.\]
Thanks to \cref{lem:b_ac}, we know $\sum_{i}\Tout{i,i} \tequiv \frac{1}{2}$, and thanks to \cref{lem:b_var_rook_equiv} and~\cref{thm:b_var_sumone}, we know $\rvarrook{i}\tequiv 1$ for all $i$. We conclude that $\sum_{p\in P}\Tout{p} \tequiv \frac{n}{2}$.
\end{proof}

That the antichain cardinality statistic is homomesic for rowmotion for the Type~B root poset was first proved by Armstrong, Stump, and Thomas~\cite[Theorem~1.2(iii)]{armstrong2013uniform}. The stronger result that it is $\tequiv \const$ (i.e., \cref{thm:b_ac}) was first proved by Hopkins~\cite[Theorem~4.2]{hopkins2017cde}, in the way we have presented above.

\subsubsection{Order ideals} 

Unfortunately, for $\brootp{n}$, neither the order ideal cardinality statistic, nor even the rank-alternating order ideal cardinality statistic, are $\tequiv \const$. (Indeed, they are not homomesic under rowmotion for $\brootp{n}$.) So the best we can do is transfer the statistics in \cref{thm:a_oi_refined} from $\arootp{2n-1}$ to $\brootp{n}$ via $\iota$.

\begin{thm} \label{thm:b_oi}
Suppose $P=\brootp{n}$. Then for $k=2,4,\ldots,2n$ we have
\[ 2\cdot\sum_{j-i=k} \oii{i,j} \ - \sum_{j-i=k-1} \oii{i,j} \ - \sum_{j-i=k+1} \oii{i,j} \ \tequiv \ 1.\]
Furthermore, we also have
\[ 2\cdot\sum_{i} \oii{i,i} - 2\cdot \sum_{i} \oii{i,i+1} \tequiv 1 \]
\end{thm}
\begin{proof}
These are exactly the statistics we get by translating the statistics in \cref{thm:a_oi_refined} to $\brootp{n}$ from $\arootp{2n-1}$ via $\iota$.
\end{proof}

\begin{remark}
It is also possible to observe that $\sum_{i} \oii{i,i} - \sum_{i} \oii{i,i+1} = \sum_{i}\Tout{i,i}$, and hence deduce that it is $\tequiv \frac{1}{2}$ from \cref{lem:b_ac}.
\end{remark}

\subsection{Non-examples} \label{subsec:nonexamples}

Having demonstrated the power of the toggleability statistics technique in re-proving known rowmotion homomesy results and proving new results for the four important families of posets treated above, we now briefly discuss some non-examples. By ``non-examples,'' we mean instances where a rowmotion homomesy result is known or conjectured, but the toggleability statistics technique fails to prove it.

\subsubsection{Reciprocity and symmetry homomesies}

Grinberg and Roby~\cite[Theorem~32]{grinberg2015birational2} proved a remarkable \emph{reciprocity} theorem for rowmotion of the rectangle, which asserts that for any box $(i,j)\in \rect{a}{b}$, we have
\[ (i,j) \in \rowm^{i+j-1}(I) \Longleftrightarrow (a+1-i,b+1-j) \notin I \]
for all $I\in \J(\rect{a}{b})$. (In fact, they proved this reciprocity result at the birational level; we discuss piecewise-linear and birational extensions in \cref{sec:pl_birational} below.) As was observed by Einstein and Propp~\cite[Theorem~8.1]{einstein2018combinatorial}, this reciprocity result immediately implies that $\oii{i,j} + \oii{a+1-i,b+1-j}$ is homomesic under rowmotion of $\rect{a}{b}$ for any $(i,j) \in \rect{a}{b}$. Unfortunately, these statistics are not $\tequiv \const$, except for some very special cases like $\oii{a,1} + \oii{1,b}$, which happen to be linear combinations of the file refinements of order ideal cardinality.

Similarly, Armstrong, Stump, and Thomas~\cite[Theorem~1.2(ii)]{armstrong2013uniform} proved that $\rowm^{n+1}(I)=I'$ for any $I\in \J(\arootp{n})$, where $I'\coloneqq \{(j,i)\colon (i,j)\in I\}$.  In other words, applying rowmotion $n+1$ times transposes the order ideal across the main diagonal axis of symmetry of the poset $\arootp{n}$. It is immediate from their result that the statistics $\oii{i,j} - \oii{j,i}$ and $\Tout{i,j}-\Tout{j,i}$ are 0-mesic under rowmotion of~$\arootp{n}$ for any $(i,j) \in \arootp{n}$. But again, these statistics are not $\tequiv \const$, except in some very special cases.

\subsubsection{Other root posets}

Let $\Phi$ be a crystallographic root system. Suppose we have chosen a set~$\Phi^+$ of positive roots of $\Phi$, and hence also a set of simple roots. There is a natural partial order on $\Phi^+$ whereby $\alpha\leq \beta$ if and only if $\beta-\alpha$ is a nonnegative linear combination of simple roots. We denote this poset by $\Phi^+$ or by $\Phi^+(X)$, where $X$ is the Cartan--Killing type of $\Phi$ (so this is consistent with our earlier notation $\arootp{n}$ and $\brootp{n}$). We call $\Phi^+$ the \dfn{root poset} of $\Phi$.

Panyushev~\cite{panyushev2009orbits} initiated the study of rowmotion acting on $\J(\Phi^+)$. In particular, he conjectured that the antichain cardinality statistic should be homomesic under antichain rowmotion.\footnote{Panyushev did not use the terminology of homomesy. Actually, Panyushev's conjecture was a major motivating example for the introduction of the notion of homomesy in~\cite{propp2015homomesy}.} This conjecture was subsequently proved, in all crystallographic types, by Armstrong, Stump, and Thomas~\cite[Theorem~1.2(iii)]{armstrong2013uniform}.

We have already seen that for $\arootp{n}$ and $\brootp{n}$, the antichain cardinality statistic is $\tequiv \const$. However, for the other root posets, it is not the case that the antichain cardinality statistic is~$\tequiv \const$. Indeed, since $\brootp{n} \simeq \Phi^+(C_n)$, the smallest example of a root poset that is not of Type~A or~B is~$\Phi^+(D_4)$. And for $\Phi^+(D_4)$, one can check directly that $\sum_{p \in P} \Tout{p} \not \tequiv \const$.

One difference between $\Phi^+(D_4)$ and the posets we studied above is that for $\Phi^+(D_4)$ there is an element that covers $3$ other elements. This means that $\Phi^+(D_4)$, unlike the posets we have been discussing, does not embed as a contiguous subset of $\quadrant$. Another difference is that piecewise-linear and birational rowmotion do not behave well for $\Phi^+(D_4)$ (again, we discuss these in~\cref{sec:pl_birational} below).

\begin{remark}
Armstrong~\cite[\S5.4.1]{armstrong2009generalized} (see also~\cite{cuntz2015root}) came up with \emph{ad hoc} constructions of root posets for the non-crystallographic types $\Phi=H_3$ and $\Phi=I_2(m)$ that enjoy several desirable properties. The root posets $\arootp{n}$, $\brootp{n}\simeq \Phi^+(C_n)$, $\Phi^+(H_3)$ and $\Phi^+(I_2(m))$ are called the \dfn{root posets of coincidental type} and in many ways are better behaved than arbitrary root posets; see, e.g.,~\cite[\S8]{hamaker2020doppelgangers}. In particular, Hopkins~\cite[Theorem~3.8]{hopkins2021minuscule} showed that for all the coincidental type root posets, the antichain cardinality statistic is $\tequiv \const$. So in addition to $A_n$ and $B_n$, which we have seen above, he treated the \emph{ad hoc} root posets for $H_3$ and $I_2(m)$.
\end{remark}

\subsubsection{The trapezoid and the ``chain of \texorpdfstring{$V$}{V}'s''}

For $a \leq b$, the \dfn{trapezoid} poset $T(a,b)$ is the poset realized as the subset $\{(i,j) \colon 1\leq i \leq a+b-1, b\leq j \leq a+b-1, i+j \geq a+b, \textrm{ and } i \leq j \}$ of $\quadrant$. Note that $T(n,n)=\brootp{n}$.

The trapezoid $T(a,b)$ is a \emph{doppelg\"{a}nger} of the rectangle $\rect{a}{b}$. What this means is that the two posets have many similarities such as their number of elements, number of order ideals, number of linear extensions, et cetera. See~\cite{hamaker2020doppelgangers} for more on doppelg\"{a}ngers. In particular, one property that~$T(a,b)$ should have in common with $\rect{a}{b}$ is good behavior of rowmotion. Indeed, it has recently been shown that $T(a,b)$ and $\rect{a}{b}$ have the same orbit structure of rowmotion~\cite[Main Theorem]{dao2019trapezoid}, and it is conjectured that this should extend to the piecewise-linear and birational realms~\cite[Conjectures~4.35 and~4.46]{hopkins2021minuscule}.

It has also been conjectured (see~\cite[Conjecture~4.9]{hopkins2021minuscule}) that $T(a,b)$ should exhibit homomesy of antichain cardinality under rowmotion. However, the toggleability statistic technique fails here: $\sum_{p \in P} \Tout{p} \not \tequiv \const$ for $T(a,b)$, except for the special cases $a=1$ and $a=b$. (Though note that $T(n,n+1)=\arootp{2n}/\twogrp$, so we can at least extract \emph{some} rowmotion homomesies in this case.)

\medskip

The \dfn{chain of $V$'s} is the poset $V(n) \coloneqq \raisebox{-0.1cm}{\begin{tikzpicture}[scale=0.3] \node[shape=circle,fill=black,inner sep=1.5] (B) at (-1,0) {}; \node[shape=circle,fill=black,inner sep=1.5] (C) at (1,0) {}; \node[shape=circle,fill=black,inner sep=1.5] (A) at (0,-1) {}; \draw (B)--(A); \draw (C)--(A); \end{tikzpicture}} \times [n]$, which is the Cartesian product of the $3$-element ``V-shaped'' poset and the $n$-element chain. It has been conjectured that $V(n)$ should exhibit good behavior of rowmotion~\cite{hopkins2020order}. In particular, a brute force analysis of all the rowmotion orbits of $V(n)$ reveals that it exhibits the antichain cardinality homomesy. However, $\sum_{p \in P} \Tout{p} \not \tequiv \const$ for $V(n)$ (as long as~$n \geq 2$). Note that $V(n)$ \emph{cannot} be realized as a contiguous subsets of $\quadrant$, because it has some elements that are covered by $3$ elements. Nevertheless, unlike the non-coincidental type root posets, $V(n)$ conjecturally has finite order of piecewise-linear rowmotion.

\medskip

There is an extension of our technique (using ``antichain toggleability statistics'') that could possibly prove some of the homomesy results mentioned in this subsection, which we are otherwise unable to address. We discuss this extension at the end of the paper, in \cref{subsec:antichain_extension}.

\section{Piecewise-linear and birational lifts} \label{sec:pl_birational}

Over the past 25 or so years, various classical combinatorial constructions have been extended to piecewise-linear and birational maps (see~\cite{kirillov2001introduction,kirillov1995groups} for some prominent examples). In 2013, Einstein and Propp~\cite{einstein2018combinatorial} introduced piecewise-linear and birational extensions of rowmotion, which have subsequently been the focus of a significant amount of research. Many of the families of posets that exhibit good behavior of rowmotion at the combinatorial level (e.g.: small, predictable order; regular orbit structure; interesting homomesies) continue to exhibit good behavior of piecewise-linear and birational rowmotion. It is well known that results that hold at the birational level must hold at the piecewise-linear level as well, and similarly that results that hold at the piecewise-linear level hold at the combinatorial level.  But there is no general process for going in the other direction, and indeed, it is not true necessarily that results on the combinatorial level must hold at the piecewise-linear or birational levels. (For example, all posets trivially have finite order of combinatorial rowmotion, but most do not have finite order of piecewise-linear or birational rowmotion.) So the fact that these special posets continue to exhibit good rowmotion behavior at these higher levels is remarkable, and as of yet, we do not have a completely satisfactory explanation for why this happens.

In this section, we start to offer some kind of general explanation for the persistence of good rowmotion behavior by showing how we sometimes \emph{can} automatically ``lift'' results (specifically, homomesy results) from the combinatorial to the higher levels of rowmotion using the toggleability statistics technique. Much of the material we present in this section was essentially already presented in~\cite{hopkins2021minuscule}, but because these results allow us to obtain such significant corollaries from the work we did in \cref{sec:main}, we reiterate all the arguments here.

We offer only a very brief review of the definitions of piecewise-linear and birational rowmotion; for more background, consult~\cite{einstein2018combinatorial, grinberg2016birational1, grinberg2015birational2, musiker2018paths, okada2020birational}. As above, $P$ will be a fixed poset. 

First, we describe piecewise-linear rowmotion. We let $\RR^P$ denote the set of functions $\pi\colon P\to \RR$. We use $\widehat{P}$ to denote the poset obtained from $P$ by adding a minimal element $\hatz$ and a maximal element $\hato$. We fix parameters $\aPL, \oPL\in\RR$. We view any $\pi\in \RR^P$ as also a function $\pi\colon \widehat{P} \to \RR$ via the convention that $\pi(\hatz)=\aPL$ and $\pi(\hato)=\oPL$. For $p\in P$, we define the \dfn{piecewise-linear toggle} $\togPL{p}\colon \RR^P\to\RR^P$ by
\[ \togPL{p}(\pi) (p') \coloneqq \begin{cases} \pi(p') &\textrm{if $p \neq p'$}; \\ \min\{\pi(r)\colon p \lessdot r\in \widehat{P}\}+\max\{\pi(r)\colon p \gtrdot r \in \widehat{P}\}-\pi(p) &\textrm{if $p=p'$}.\end{cases}\]
Following the Cameron--Fon-der-Flaass definition of (combinatorial) rowmotion as a composition of toggles, we define \dfn{piecewise-linear rowmotion} $\rowmPL\colon \RR^P\to\RR^P$ by setting
\[\rowmPL\coloneqq \togPL{p_1}\circ \cdots \circ \togPL{p_n},\]
where $p_1,\ldots,p_n$ is any linear extension of $P$.

To each order ideal $I \in \J(P)$, we associate the function $\pi_I \in \RR^P$, which is defined to be the indicator function of the complement of $I$. With the parameters $\aPL=0$ and $\oPL=1$, (piecewise-linear) toggling of the $\pi_I$ exactly corresponds to (combinatorial) toggling of the $I$. Hence, results that are true for PL rowmotion are true for combinatorial rowmotion as well; this process of going from piecewise-linear to combinatorial is called \dfn{specialization}.

Now we describe birational rowmotion. The key is to ``detropicalize'' the PL expressions (i.e., replace $+$ by $\times$ and $\max$ by $+$ ). We use $\RRgt^P$ to denote the set of functions $\pi\colon P\to \RRgt$.\footnote{Here, we restrict to $\pi$ taking values in the positive real orthant to avoid problems that arise when the denominator in the definition of $\togB{p}$ vanishes. In the context of algebraic dynamics, it is more natural to allow the values of $\pi$ to be zero, negative, or even complex; to make this work, one must deal honestly with \emph{rational} maps, i.e., maps defined only on an open dense subset of $\pi$. We have opted for the technically least demanding avenue. Homomesy results that hold in the positive orthant can be extended to the aforementioned broader settings through standard arguments; see~\cite[\S3]{grinberg2016birational1}.}  We fix parameters $\aB, \oB\in\RRgt$. We view any $\pi\in \RRgt^P$ as also a function $\pi\colon \widehat{P} \to \RRgt$ via the convention that $\pi(\hatz)=\aB$ and $\pi(\hato)=\oB$. For~$p\in P$, we define the \dfn{birational toggle} $\togB{p}\colon \RRgt^P\to\RRgt^P$ by
\[ \togB{p}(\pi) (p') \coloneqq \begin{cases} \pi(p') &\textrm{if $p \neq p'$}; \\ \displaystyle \frac{\displaystyle \sum_{p \gtrdot r \in\widehat{P}}\pi(r)}{\pi(p) \cdot \displaystyle \sum_{p\lessdot r \in \widehat{P}}\pi(r)^{-1}} &\textrm{if $p=p'$}.\end{cases}\]
We define \dfn{birational rowmotion} $\rowmB\colon \RRgt^P\to\RRgt^P$ by setting
\[\rowmB\coloneqq \togB{p_1}\circ \cdots \circ \togB{p_n},\]
where $p_1,\ldots,p_n$ is any linear extension of $P$.

It is always possible to translate equalities of expressions at the birational level to the piecewise-linear level by replacing $+$ by $\max$ and $\times$ by $+$ everywhere in expressions; this process of going from birational to piecewise-linear is called \dfn{tropicalization}. For a precise statement about what tropicalization entails in the context of rowmotion, see, e.g.,~\cite[\S7]{einstein2018combinatorial}.

\begin{example}
\Cref{fig:b_row} depicts an orbit of $\rowmB$ for a generic point $\pi \in \RRgt^P$, where $P=[2]\times[2]$. Observe that $(\rowmB)^4$ is the identity.
\end{example}

\begin{figure}
\begin{center}
\includegraphics[height=3.9cm]{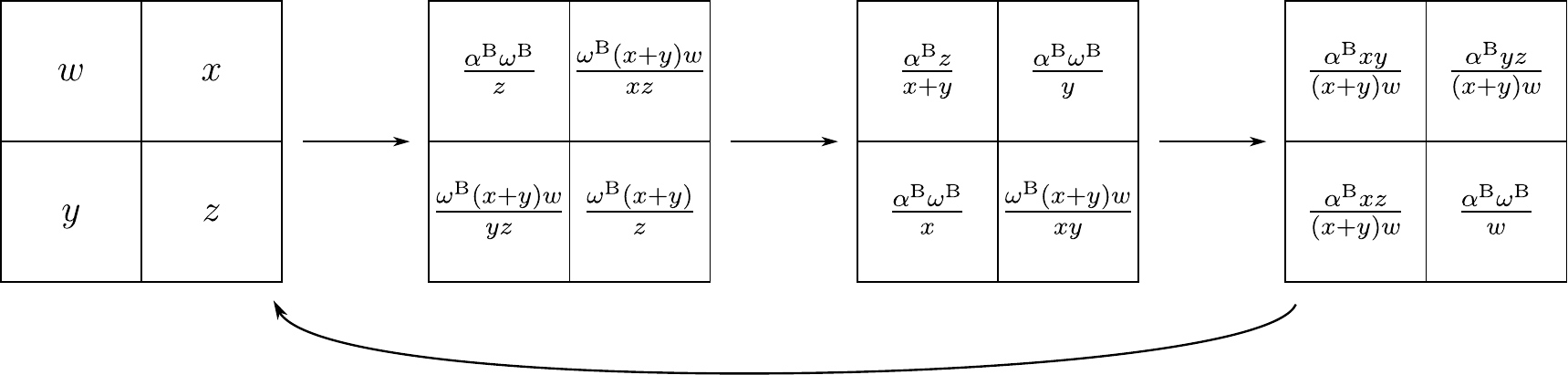}
\end{center}
\caption{An orbit of $\rowmB$ for a generic point $\pi \in \RRgt^P$.} \label{fig:b_row}
\end{figure}

We want to show that Striker's observation persists at the piecewise-linear and birational levels. Hence, we need to introduce PL and birational extensions of the toggleability statistics. So, for $p\in P$, define $\TinPL{p}, \ToutPL{p}, \TPL{p}\colon \RR^P\to\RR$ by
\begin{align*}
\TinPL{p}(\pi) &\coloneqq \pi(p) - \max\{\pi(r)\colon p \gtrdot r\in \widehat{P}\}; \\
\ToutPL{p}(\pi) &\coloneqq \min\{\pi(r)\colon p \lessdot r\in \widehat{P}\}-\pi(p); \\
\TPL{p}(\pi) &\coloneqq \TinPL{p}(\pi)-\ToutPL{p}(\pi),
\end{align*}
and define $\TinB{p}, \ToutB{p}, \TB{p}\colon \RRgt^P\to\RRgt$ by
\begin{align*}
\TinB{p}(\pi) &\coloneqq \frac{\pi(p) }{\sum_{ p \gtrdot r\in \widehat{P}} \pi(r)}; \\
\ToutB{p}(\pi) &\coloneqq \frac{1}{ \pi(p)\cdot \sum_{ p \lessdot r\in \widehat{P}} \pi(r)^{-1}}; \\
\TB{p}(\pi) &\coloneqq \TinB{p}(\pi)/\ToutB{p}(\pi)
\end{align*}
so that $\TinB{p},\ToutB{p},\TB{p}$ tropicalize to $\TinPL{p},\ToutPL{p},\TPL{p}$ which in turn specialize to $\Tin{p},\Tout{p},\T{p}$. We then have the following extensions of \cref{lem:striker} (which appeared previously as \cite[Lemma~4.31, Lemma~4.42]{hopkins2021minuscule}):

\begin{lemma} \label{lem:striker_pl}
Let $O  \subseteq \RR^P$ be a finite $\rowmPL$-orbit. Then for all $p\in P$, we have $\sum_{\pi \in O }\TPL{p}(\pi)=0$.
\end{lemma}

\begin{lemma} \label{lem:striker_b}
Let $O  \subseteq \RRgt^P$ be a finite $\rowmB$-orbit. Then for all $p\in P$, we have $\prod_{\pi \in O }\TB{p}(\pi)=1$.
\end{lemma}

\begin{remark}
Notice the restriction to finite orbits in \cref{lem:striker_pl,lem:striker_b}. We could deal with infinite orbits by taking limits in the appropriate way. However, in practice, this will not matter because for all the posets of interest for us (i.e., those discussed in \cref{sec:main}), both PL and birational rowmotion have finite order (see~\cite{grinberg2015birational2, okada2020birational}), hence all orbits are finite. Also, modulo this issue of finite orbits, we could state these two lemmas in terms of orbit averages (arithmetic means and geometric means, respectively).
\end{remark}

\begin{proof}[Proof of \cref{lem:striker_pl,lem:striker_b}]
It suffices to prove \cref{lem:striker_b}, thanks to tropicalization. The proof is exactly the same as that of \cref{lem:striker}. It is not hard to see, by induction on the position of~$p$ in the linear extensions $p_1,\ldots,p_n$ defining rowmotion, that we have $\TinB{p}(\pi) = \ToutB{p}(\rowmB(\pi))$ for all $\pi \in \RRgt^{P}$. Thus, the numerator and denominator of $\prod_{\pi \in O }\TB{p}(\pi)$ cancel, so this product equals~$1$.
\end{proof}

In fact, the version of Striker's lemma for the rank-permuted variants of rowmotion also holds at the PL and birational levels, as we now explain. 

Suppose now that~$P$ is ranked. For $i=0,1,\ldots,\rk(P)$, define $\rktogPL{i}\coloneqq \prod_{p\in P, \:\rk(p)=i}\togPL{p}$ and $\rktogB{i}\coloneqq \prod_{p\in P, \:\rk(p)=i}\togB{p}$. For any permutation $\sigma$ of $0,1,\ldots,\rk(P)$, define
\begin{align*}
\rowmPL_{\sigma} &\coloneqq \rktogPL{\sigma(0)}\circ \rktogPL{\sigma(1)} \circ \cdots \circ \rktogPL{\sigma(\rk(P))} \\
\rowmB_{\sigma} &\coloneqq \rktogB{\sigma(0)}\circ \rktogB{\sigma(1)} \circ \cdots \circ \rktogB{\sigma(\rk(P))}
\end{align*} 
so that ordinary rowmotion corresponds to the identity permutation. We have the following extensions of \cref{lem:striker_permuted}.

\begin{lemma} \label{lem:striker_permuted_pl}
Suppose $P$ is ranked. Let $O  \subseteq \RR^P$ be a finite $\rowmPL_{\sigma}$-orbit, for any choice of $\sigma$. Then for all $p\in P$, we have $\sum_{\pi \in O }\TPL{p}(\pi)=0$.
\end{lemma}

\begin{lemma} \label{lem:striker_permuted_b}
Suppose $P$ is ranked. Let $O  \subseteq \RRgt^P$ be a finite $\rowmB_{\sigma}$-orbit, for any choice of $\sigma$. Then for all $p\in P$, we have $\prod_{\pi \in O }\TB{p}(\pi)=1$.
\end{lemma}

Note that, via specialization, \cref{lem:striker_permuted_pl} implies \cref{lem:striker_permuted}, whose proof we omitted before.

\begin{proof}[Proofs of~\cref{lem:striker_permuted_pl,lem:striker_permuted_b}]
We prove~\cref{lem:striker_permuted_b}. Via tropicalization, this will also prove \cref{lem:striker_permuted_pl}. Let $i\coloneqq \rk(p)$. For simplicity, assume that $i\neq 0,\rk(P)$. The proof depends on the relative order that the numbers $i-1$, $i$, and $i+1$ appear in the sequence $\sigma(0),\sigma(1),\ldots,\sigma(\rk(P))$. 

If $\sigma^{-1}(i-1) < \sigma^{-1}(i) <\sigma^{-1}(i+1)$, then, just as with ordinary rowmotion, we will have $\TinB{p}(\pi) = \ToutB{p}(\rowmB_{\sigma}(\pi))$ for all $\pi \in \RRgt^{P}$, and so, as before, all the terms in the numerator and denominator of $\prod_{\pi \in O }\TB{p}(\pi)$ cancel.

If $\sigma^{-1}(i+1) < \sigma^{-1}(i) <\sigma^{-1}(i-1)$, then, just as with the inverse of ordinary rowmotion,  we will have $\ToutB{p}(\pi) = \TinB{p}(\rowmB_{\sigma}(\pi))$ for all $\pi \in \RRgt^{P}$, and so again, all the terms in $\prod_{\pi \in O }\TB{p}(\pi)$ cancel.

If $\sigma^{-1}(i-1) < \sigma^{-1}(i)$ and $\sigma^{-1}(i+1)<\sigma^{-1}(i)$, then the situation is slightly different. In this case, when applying the toggles that constitute $\rowmB$, $\togB{p}$ will be applied before $\togB{p'}$ for any $p'$ that either covers or is covered by $p$. This means that $\rowmB_{\sigma}(\pi)(p) = \togB{p}(\pi)(p)$. So, looking at the definitions of the toggle $\togB{p}$ and of the toggleability statistics $\ToutB{p}$ and $\TinB{p}$, we can therefore see that $\ToutB{p}(\pi)/\TinB{p}(\pi)=\rowmB_{\sigma}(\pi)(p)/\pi(p)$ will hold for all~$\pi \in \RRgt^{P}$. But this identity still suffices to show that $\prod_{\pi \in O }\TB{p}(\pi)=1$.

Finally, if $\sigma^{-1}(i) < \sigma^{-1}(i-1)$ and $\sigma^{-1}(i)<\sigma^{-1}(i+1)$, then this is the inverse of the previous case. So (replacing $\rowmB_{\sigma}$ by its inverse and then replacing $\pi$ by $\rowmB_{\sigma}(\pi)$ in the identity we obtained above) we have $\ToutB{p}(\rowmB_{\sigma}(\pi))/\TinB{p}(\rowmB_{\sigma}(\pi))=\pi(p)/\rowmB_{\sigma}(\pi)(p)$ for all $\pi \in \RRgt^{P}$, and this again implies that $\prod_{\pi \in O }\TB{p}(\pi)=1$.

If $i=0$ or $i=\rk(P)$, the details are slightly different, but the analysis is essentially similar. 
\end{proof}

In light of these extensions of \cref{lem:striker}, it makes sense to introduce PL and birational extensions of the $\tequiv$ notation. Thus, for two statistics $\stat{f},\stat{g}\colon\RR^{P}\to \RR$, we write $\stat{f}\tequivPL \stat{g}$ if there exist $c_p \in \RR$ such that $\stat{f}-\stat{g}=\sum_{p\in P}c_p\TPL{p}$. Similarly, for two statistics $\stat{f},\stat{g}\colon\RRgt^{P}\to \RRgt$, we write $\stat{f}\tequivB \stat{g}$ if there exist $c_p \in \RR$ such that $\stat{f}/\stat{g}=\prod_{p\in P}(\TB{p})^{c_p}$. As before, we also write $c$ for the function that is constantly equal to $c$ (on any domain). With this notation, we have the following corollaries of \cref{lem:striker_pl,lem:striker_b,lem:striker_permuted_pl,lem:striker_permuted_b}.

\begin{prop} \label{prop:homo_pl}
Suppose $\stat{f}\colon\RR^{P}\to \RR$ satisfies $\stat{f} \tequivPL c(\oPL-\aPL)$ for some $c\in \RR$. Then for any finite $\rowmPL$-orbit $O\subseteq \RR^{P}$, we have $\sum_{\pi \in O} \stat{f}(\pi)= \#O\cdot c(\oPL-\aPL)$. If $P$ is ranked, then the equality holds for all finite $\rowmPL_{\sigma}$ orbits.
\end{prop}

\begin{prop} \label{prop:homo_b}
Suppose $\stat{f}\colon\RRgt^{P}\to \RRgt$ satisfies $\stat{f} \tequivB (\oB/\aB)^{c}$ for some $c\in \RRgt$. Then for any finite $\rowmB$-orbit $O\subseteq \RR^{P}$, we have $\prod_{\pi \in O} \stat{f}(\pi)= (\oB/\aB)^{\#O\cdot c}$. If $P$ is ranked, then we have similarly for all finite $\rowmB_{\sigma}$ orbits.
\end{prop}

So \cref{prop:homo_pl,prop:homo_b} say that we can obtain piecewise-linear and birational rowmotion homomesies if we can prove the piecewise-linear and birational equivalents of $\stat{f}\tequiv \const$. But so far, we have only discussed how to show $\stat{f}\tequiv \const$ at the combinatorial level. However, we will now explain how it is possible to \emph{automatically} lift the fact that $\stat{f}\tequiv \const$ to the PL and birational levels for statistics $\stat{f}$ of the forms we have been studying and for certain families of posets~$P$, including all of the main examples we discussed in \cref{sec:main}.
 
Because we want to work with the same kinds of statistics from \cref{sec:main}, we will need extensions of the $\oii{p}$. So for $p\in P$, define $\oiiPL{p}\colon \RR^{P}\to\RR$ by
\[\oiiPL{p}(\pi) \coloneqq \oPL-\pi(p),\]
and similarly define $\oiiB{p}\colon \RRgt^{P}\to\RRgt$ by
\[\oiiB{p}(\pi) \coloneqq \oB / \pi(p).\]
(The $\oPL-$ part may look a little strange, but, as always, this definition is chosen so that when $\aPL=0$ and $\oPL=1$, the function agrees with its combinatorial counterpart under specialization, i.e., $\oiiPL{p}(\pi_I)=\oii{p}(I)$ for all $I\in\J(P)$.)

\begin{lemma} \label{lem:lifting}
Suppose that $P$ is such that each element $p\in P$ covers at most two elements and is covered by at most two elements.
Let $\stat{f}\colon \J(P)\to\RR$ be of the form 
\[\sum_{p\in P} (a_p \Tin{p} + a'_p \Tout{p} + a''_p \oii{p})\] for $a_p,a_p',a_p''\in\RR$. Set
\[\statPL{f} \coloneqq \sum_{p\in P} (a_p \TinPL{p} + a'_p \ToutPL{p} + a''_p \oiiPL{p})\] 
and 
\[\statB{f} \coloneqq \prod_{p\in P} (\ToutB{p})^{a_p} \cdot (\ToutPL{p})^{a_p} \cdot (\oiiB{p})^{a''_p}.\] 
If $\stat{f} = c$ for some $c\in \RR$, then we have $\statPL{f} = c(\oPL-\aPL)$ and $\statB{f} = (\oB/\aB)^{c}$.
\end{lemma}

\begin{proof}
The proof closely follows~\cite[Proofs of Lemma 4.30 and Lemma 4.41]{hopkins2021minuscule}. Hence we will try to be brief and omit some details.

We first prove the PL statement. For the moment, suppose $\aPL=0$ and $\oPL=1$. Consider the \dfn{order polytope} $\mathcal{O}(P)\subseteq \RR^P$ of points $\pi \in \RR^P$ satisfying $0\leq \pi(p)\leq 1$ for all $p\in P$ and satisfying $\pi(p)\leq \pi(q)$ whenever $p\leq q$ in $P$. Stanley~\cite{stanley1986twoposet} proved that the vertices of $\mathcal{O}(P)$ are the $\pi_I$ for~$I\in\J(P)$. He also observed that $\mathcal{O}(P)$ has a canonical triangulation whose maximal simplices are indexed by the linear extensions of $P$, with the linear extension $p_1,\ldots,p_n$ corresponding to the simplex consisting of all $\pi$ satisfying $0\leq \pi(p_1) \leq \pi(p_2)\leq \cdots \leq \pi(p_n)\leq 1$. Observe that $\statPL{f}$ is an affine linear function on each maximal simplex in this triangulation; moreover, it is, by supposition, equal to~$c$ at each vertex of each such simplex. Hence, on each simplex, it is in fact constantly equal to $c$. Consequently, $\statPL{f}(\pi)=c$ for all $\pi\in\mathcal{O}(P)$. 

Now let $\aPL, \oPL\in \RR$ be arbitrary. So far, we have not used the hypothesis that each element of~$P$ covers, and is covered by, at most two elements. Now we will use that supposition. Specifically, we repeatedly use the identity $\max(x,y)+\min(x,y)=x+y$ to rewrite $\statPL{f}$ as
\[ \statPL{f}(\pi) = b_{\alpha} \aPL + b_{\omega} \oPL + \sum_{p} b_p \pi(p) + \sum_{p \parallel q}  b_{p,q} \max(\pi(p),\pi(q))\]
for certain constants $ b_{\alpha}, b_{\omega}, b_p, b_{p,q} \in \RR$. (We use $p \parallel q$ to denote that $p$ and $q$ are incomparable.)

First of all, we can easily check that $b_{\alpha} = -c$ by evaluating $\stat{f}$ at the order ideal~$\varnothing$, and similarly that $b_{\omega} = c$ by evaluating $\stat{f}$ at the order ideal $P$. 

Thus, it remains to show $b_p=0$ and $b_{p,q}=0$. We again consider the case $\aPL=0$ and $\oPL=1$.  Recall that in this case, we have shown $\statPL{f}(\pi)=c$ for all $\pi\in\mathcal{O}(P)$. First, one can show that all~$b_{p,q}$ are $0$ by focusing on the set of $\pi\in\mathcal{O}(P)$ with $\pi(x)=0$ if $x < p$ or $x < q$, and $\pi(y)=1$ if~$y\neq,p,q$ and $y$ is not less than either $p$ or $q$. Then one can show that all $b_p$ are $0$ by focusing on the subset of $\pi\in\mathcal{O}(P)$ with $\pi(x)=0$ if $x < p$ and $\pi(y)=1$ if $y\not\leq p$. So we have shown that indeed $\statPL{f} = c(\oPL-\aPL)$.

Now we consider the birational statement. Analogous to the PL setting, we repeatedly use the identity $(x+y) \cdot (x^{-1}+y^{-1})^{-1}=xy$ to rewrite $\statB{f}$ as
\[ \statB{f}(\pi) = (\aB)^{ b_{\alpha}} \cdot (\oB)^{b_{\omega}} \cdot \prod_{p} (\pi(p))^{b_p} \cdot \prod_{p \parallel q} (\pi(p)+\pi(q))^{b_{p,q}}\]
for certain constants $b_{\alpha}, b_{\omega}, b_p, b_{p,q} \in \RR$. But here is the key observation: these $b_{\alpha}, b_{\omega}, b_p, b_{p,q}$ are exactly the same as the coefficients in the PL case, because they are obtained in exactly the same way (i.e., via the de-tropicalization of the $\max(x,y)+\min(x,y)=x+y$ identity). So we indeed have $\statB{f} = (\oB/\aB)^{c}$, and we are done.
\end{proof}

\begin{remark} \label{rem:pl_proofs}
\Cref{lem:lifting} says that (for a poset satisfying its hypothesis) if we take any identity relating the statistics $\oii{p}$, $\Tout{p}$, $\Tin{p}$, and $1$ that holds at the combinatorial level, and we replace all the terms by their PL extensions, then the identity will still be true. What is more, we can prove it is true by repeatedly using the identity $\max(x,y)+\min(x,y)=x+y$ and cancelling all terms. In fact, this perspective actually leads to more algebraic proofs of the identities from \cref{sec:main}.  For instance, consider the PL extension of~\eqref{eqn:rectfilea}:
\begin{equation} \label{eqn:rectfilea_pl}
 \oiiPL{p} \ = \sum_{i' \geq i, \ j' \geq j} \ToutPL{i',j'} \ \ - \sum_{i' > i, \ j' > j} \TinPL{i',j'}
\end{equation}
for any $p=(i,j)\in \rect{a}{b}$. Let $s=(i',j')$ be a generic box with $i'\geq i$ and $j' \geq j$. Consider the contribution of terms depending on $s$ to the right-hand side of \eqref{eqn:rectfilea_pl}. Locally near $s$, the picture of this sum of toggleability statistics looks like~\cref{fig:pl_proofs}. If we write all the terms on the right-hand side of \eqref{eqn:rectfilea_pl} in terms of $\max$'s and $\min$'s of $\pi(p')$'s, we see that the only parts involving $\pi(s)$ will be
\[-2\cdot \pi(s) + \max(\pi(s),\pi(t))+\min(\pi(s),\pi(t)) + \max(\pi(s),\pi(u))+\min(\pi(s),\pi(u)),\]
and via the identity $\max(x,y)+\min(x,y)=x+y$, this becomes
\[ -2\cdot \pi(s) + \pi(s)+\pi(t) +\pi(s)+\pi(u)=\pi(t)+\pi(u), \]
which actually does not depend on $\pi(s)$ at all. In this way, we could see that the right-hand side of \eqref{eqn:rectfilea_pl} has no terms depending on any of the $\pi(p')$'s, except $\pi(p)$, and that it in fact simplifies to~$\oPL-\pi(p)= \oiiPL{p}$. To do this, we would need to also check some other boundary cases that look slightly different from \cref{fig:pl_proofs}. But this cancellation is the essential idea in \emph{every} identity treated in~\cref{sec:main}.
\end{remark}

\begin{figure}
\begin{center}
\includegraphics[height=2.786cm]{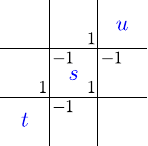}
\end{center}
\caption{For \cref{rem:pl_proofs}: how terms of an identity cancel locally. } \label{fig:pl_proofs}
\end{figure}

If $P$ satisfies the hypothesis of \cref{lem:lifting}, then this lemma says that for any combinatorial statistic $\stat{f}\in\Span\{\Tin{p},\Tout{p},\oii{p}\colon p\in P\}$, the PL counterpart $\statPL{f}\colon \RR^P\to \RR$ and birational counterpart $\statB{f}\colon \RRgt^P\to \RRgt$, as defined in that lemma, are uniquely defined; i.e., they do not depend on how we represent $\stat{f}$ as a linear combination of the $\Tin{p}$, $\Tout{p}$, $\oii{p}$. Moreover, we also clearly have the following consequence of \cref{lem:lifting}:

\begin{thm} \label{thm:lifting}
Suppose that $P$ satisfies the hypothesis of \cref{lem:lifting}, and consider any statistic $\stat{f}\in\Span\{\Tin{p},\Tout{p},\oii{p}\colon p\in P\}$. Suppose we have $\stat{f}\tequiv c$ for some $c\in \RR$. Then $\statPL{f} \tequivPL c(\oPL-\aPL)$ and $\statB{f} \tequivB (\oB/\aB)^{c}$.
\end{thm}

Observe that all the posets $P$ we studied in \cref{sec:main} satisfy the hypothesis of \cref{lem:lifting}. Thus, by combining \cref{thm:lifting} and \cref{prop:homo_pl,prop:homo_b}, we see how we can obtain PL and birational analogues of the rowmotion homomesies for all the $\stat{f}$ we showed satisfy $\stat{f}\tequiv \const$ in \cref{sec:main}. In this way, we recover many homomesy results from~\cite{einstein2018combinatorial, musiker2018paths, okada2020birational}, and we obtain some new ones as well. (As we suggested above, this way of lifting the toggleability statistics technique to the PL and birational setting was previously developed in~\cite{hopkins2021minuscule}, but there it was only used for the antichain cardinality statistic.) We will not go individually over every PL/birational rowmotion homomesy obtained in this way, but let us end with one example.

\begin{example}
Suppose $P=\rect{a}{b}$ is the rectangle. \Cref{cor:rect_oi} says that the order ideal cardinality statistic $\sum_{p\in P}\oii{p}$ satisfies $\sum_{p\in P}\oii{p} \tequiv ab/2$. Thus, \cref{thm:lifting} tells us that $\prod_{p\in P} \frac{\oB}{\pi(p)}\tequivB (\oB/\aB)^{ab/2}$. Together with~\cref{prop:homo_b}, this means that for any finite birational rowmotion orbit $O\subseteq \RRgt^{P}$, we have that $\prod_{\pi \in O} \prod_{p\in P} \frac{\oB}{\pi(p)} = (\oB/\aB)^{\#O\cdot(ab/2)}$. Grinberg and Roby~\cite[Theorem~30]{grinberg2015birational2} proved that for $P=[a]\times[b]$, the order of birational rowmotion is $a+b$. Hence, we can say that $\prod_{k=0}^{a+b-1} \prod_{p\in P} \frac{\oB}{(\rowmB)^{k}(\pi)(p)} = (\oB/\aB)^{(a+b)\cdot (ab/2)}$ for any $\pi \in \RRgt^{P}$. Rewriting this, we get
\[ \prod_{k=0}^{a+b-1} \prod_{p\in P} (\rowmB)^{k}(\pi)(p) = (\oB)^{(a+b)\cdot (ab/2)} \cdot (\aB)^{(a+b) \cdot (ab/2)}\]
for any $\pi \in \RRgt^{P}$, a result which was previously proved by Einstein and Propp~\cite[Theorem~5.3 and Theorem~6.6]{einstein2018combinatorial} (see also~\cite{musiker2018paths}). For example, the reader is encouraged to verify that if we multiply together all the values appearing in \cref{fig:b_row}, we obtain $(\oB)^8\cdot(\aB)^8$. We also obtain PL/birational lifts of the file refinements of order ideal cardinality in this way as well.
\end{example}

\section{\texorpdfstring{$q$}{q}-analogues}
\label{sec:q}

\subsection{\texorpdfstring{$q$}{q}-rowmotion}

We begin this section by defining a $q$-analogue of rowmotion. For this purpose, we assume $q$ is a positive rational number, which we have written as $q=r/s$ for some positive integers~$r$ and $s$. Note that we do not assume $r$ and $s$ are relatively prime. Note also that $q$-rowmotion will actually depend on $r$ and $s$, and not just~$q$, but we suppress this dependence in the name to emphasize that it is indeed a $q$-analogue of rowmotion.

Let $\Fz$ and $\Fo$ be disjoint sets of cardinality $s$ and $r$, respectively, and fix a cyclic permutation $\theta\colon \Fz\cup \Fo\to \Fz\cup \Fo$. We think of $\Fz$ as a set consisting of $s$ different ``flavors'' of the number~$0$, say $0_1,\ldots,0_s$. Similarly, we think of the elements of $\Fo$ as $r$ different flavors of the  number $1$, say~$1_1,\ldots,1_r$. The cyclic permutation $\theta$ can be chosen arbitrarily. 

As in the classical setting, we have a fixed finite poset $P$. Define $\J_{r,s}(P)$ to be the set of labelings $L\colon P\to\Fz\cup\Fo$ such that $L^{-1}(\Fz)$ is an order ideal of $P$. For $L\in\J_{r,s}(P)$, we say an element~$x\in P$ is \dfn{active} in $L$ if $x$ is either a maximal element of $L^{-1}(\Fz)$ or is a minimal element of $L^{-1}(\Fo) = P\setminus L^{-1}(\Fz)$. For each $p\in P$, define the \dfn{toggle} $\tog{p}\colon \J_{r,s}(P)\to\J_{r,s}(P)$ by
 \[\tog{p}(L)(x)\coloneqq\begin{cases} \theta(L(x)) &\textrm{if $x=p\in \max(L^{-1}(\Fz))\cup\min(L^{-1}(\Fo))$}, \\ 
L(x) &\textrm{otherwise}. \end{cases}\]
In other words, $\tog{p}$ applies the cyclic permutation $\theta$ to the label of $p$ (and fixes all other labels) if $p$ is active in $L$, and $\tog{p}$ does nothing if $p$ is not active in $L$. 
Although $\tog{p}$ is generally not an involution, it is invertible. Indeed, if $p$ is active in $L$, then it is also active in $\tog{p}(L)$, $\tog{p}^2(L)$, etc., so $\tog{p}^{r+s}(L) =L$. On the other hand, if $p$ is not active in $L$, then $\tog{p}(L)=L$.

Finally, define the \dfn{$q$-rowmotion} operator $\rowm\colon \J_{r,s}(P)\to\J_{r,s}(P)$ by $\rowm=\tog{p_1}\circ\cdots\circ\tog{p_n}$, where $p_1,\ldots,p_n$ is an arbitrary linear extension of $P$. This operator is well-defined (i.e., does not depend on the choice of the linear extension) because the toggles $\tog{p_i}$ and~$\tog{p_j}$ commute whenever $p_i$ and $p_j$ are incomparable in $P$. Note that the $q$-rowmotion operator depends on the choice of $\theta$.

\begin{example}\label{ex:Exam1}
Suppose $q=r=s=1$, and let $\Fz=\{0\}$ and $\Fo=\{1\}$. In this case, there is a natural correspondence between the labelings in $\J_{1,1}(P)$ and the order ideals in $\J(P)$; each labeling~$L$ simply corresponds to $L^{-1}(0)$. The cyclic permutation $\theta\colon \{0,1\}\to\{0,1\}$ must send $0$ to~$1$ and send $1$ to $0$. The $q$-rowmotion operator on $\J_{1,1}(P)$ coincides with the classical rowmotion operator on~$\J(P)$. 
\end{example}

\begin{figure}
\begin{center}
\includegraphics[height=7.444cm]{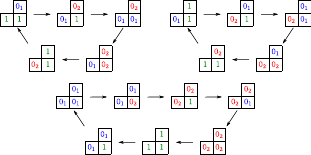}
\end{center}
\caption{The orbits of the $q$-rowmotion operator on $\J_{1,2}(\Phi^+(A_2))$.}\label{fig:q-orbits}
\end{figure}

\begin{example}\label{ex:Exam2}
Suppose $r=1$ and $s=2$ (so $q=1/2$), and let $\Fz=\{0_1,0_2\}$ and $\Fo=\{1\}$. Define the cyclic permutation $\theta$ by $\theta(0_1)=0_2$, $\theta(0_2)=1$, $\theta(1)=0_1$. Suppose $P=\arootp{2}$ is the root poset of Type $A_2$. Then $\J_{1,2}(P)$ contains $17$ labelings. The action of $q$-rowmotion on these labelings is depicted in \cref{fig:q-orbits}.
\end{example}

\begin{remark}
It is easy to see that 
\[\#\J_{r,s}(P) = \sum_{I\in \J(P)} r^{\#(P\setminus I)} s^{\#I} = s^{\# P} \cdot \sum_{I \in \J(P)} q^{\#(P\setminus I)}.\]
For example, consider $P=\rect{a}{b}$. In this case, it is well known that $\sum_{I \in \J(P)} q^{\#(P \setminus I)} = \qbinom{a+b}{b}_q$, so $\#\J_{r,s}(P) = s^{ab} \cdot \qbinom{a+b}{b}_q$. Here, we use the common notation for the \dfn{$q$-number} $[n]_q \coloneqq (1-q^n)/(1-q) = (1+q+\cdots+q^{n-1})$, \dfn{$q$-factorial} $[n]_q! \coloneqq [n]_q\cdot [n-1]_q \cdots [1]_q$, and \dfn{q-binomial} $\qbinom{n}{k}_q \coloneqq \frac{[n]_q!}{[k]_q![n-k]_q!}$. 
\end{remark}

Every statistic $\stat{f}\colon \J(P)\to\RR$ can naturally be viewed as a statistic $\stat{f}\colon \J_{r,s}(P)\to \RR$ by setting $\stat{f}(L) \coloneqq \stat{f}(L^{-1}(\Fz))$ for all $L\in \J_{r,s}(P)$. We slightly abuse notation by using the same symbol to refer to both the statistic on~$\J(P)$ and the associated statistic on~$\J_{r,s}(P)$; in what follows, we do not distinguish between these statistics. Our goal is to prove that various statistics $\stat{f}\colon\J_{r,s}(P)\to\RR$ are homomesic for $q$-rowmotion. All of the statistics that we will consider are actually obtained from statistics on $\J(P)$, so they are defined for all choices of positive integers $r$ and $s$. One might hope to find such statistics that are homomesic for $q$-rowmotion for all choices of $r$ and $s$ and all choices of the cyclic permutation $\theta$. We will find, somewhat surprisingly, that there are several natural statistics with these properties. 

In order to prove instances of homomesy that hold for all $q$, it is convenient to shift our perspective and view $q$ as a formal indeterminate instead of a fixed rational number. Hence, we will consider statistics $\stat{f}\colon\J(P)\to\RR(q)$, where $\RR(q)$ is the field of rational functions in $q$ with coefficients in $\RR$. 
Suppose we are given $\stat{f}\colon\J(P)\to\RR(q)$ and a real number $z \in \RR$. If none of the rational functions $\stat{f}(I)$ for $I\in \J(P)$ have singularities at $z$, then we obtain the real-valued statistic $\stat{f}\vert_{z}\colon\J(P)\to\RR$ by specializing $q$ to equal $z$. Thus, $\stat{f}\vert_{z}(I)$ is the real number obtained by evaluating $\stat{f}(I)$ at $q=z$. For example, in the context of $q$-rowmotion, we will want to evaluate $\stat{f}$ at $r/s$, where $r$ and $s$ are positive integers; as mentioned, we will also view this $\stat{f}\vert_{r/s}$ as a real-valued statistic on $\J_{r,s}(P)$.

Note that any statistic $\stat{g}\colon \J(P)\to\RR$ can be viewed as a statistic $\stat{g}\colon \J(P)\to\RR(q)$ via the natural inclusion $\RR\subseteq \RR(q)$; in this case, $\stat{g}\vert_{z}$ is the same as $\stat{g}$ for any $z\in\RR$.  

We will need $q$-analogues of the toggleability statistics introduced earlier. The statistics $\Tin{p}$ and~$\Tout{p}$ defined in \cref{sec:intro} are real-valued, so they can also be viewed as statistics on $\J(P)$ with values in $\RR(q)$. We then define $\qT{p}\colon\J(P)\to\RR(q)$ by 
\[\T{p}^q\coloneqq\Tin{p}-q\Tout{p}.\] 
For every $I\in \J(P)$, the rational function $\qT{p}(I)$ is a polynomial in $q$, so it has no real singularities. Hence, given positive integers $r$ and $s$, we obtain the statistic $\qT{p}\vert_{r/s}\colon\J_{r,s}(P)\to\RR$ by specializing to $q \coloneqq r/s$ and then viewing the resulting statistic on $\J(P)$ as one defined on $\J_{r,s}(P)$. The utility of this definition comes from the following $q$-version of Striker's lemma. 

\begin{lemma}\label{lem:qStriker}
Fix positive integers $r$ and $s$ and a cyclic permutation $\theta\colon\Fz\cup\Fo\to\Fz\cup\Fo$. For any poset $P$ and any $p\in P$, the statistic $\qT{p}\vert_{r/s}$ is $0$-mesic under the $q$-rowmotion operator $\rowm\colon \J_{r,s}(P)\to\J_{r,s}(P)$. 
\end{lemma}

\begin{proof}
Let $L\colon P\to\Fz\cup\Fo$ be a labeling in $\J_{r,s}(P)$. To ease notation, let us write $L_k=\rowm^k(L)$. Let $N$ be the smallest positive integer such that $L_N=L$ (i.e., the size of the $q$-rowmotion orbit containing $L$). Let $K$ be the set of integers $k\in\{0,\ldots,N-1\}$ such that $L_k(p)\neq L_{k+1}(p)$. For each $k\in K$, we have $L_{k+1}(p)=\theta(L_k(p))$. Thus,  $L(p)=L_N(p)=\theta^{\#K}(L(p))$. Since $\theta$ is a cyclic permutation of a set of size $r+s$, we must have $\#K=\alpha(r+s)$ for some nonnegative integer $\alpha$. The number of integers $i\in\{0,\ldots,N-1\}$ such that $L_i(p)\in\Fo$ and $i\in K$ is $\alpha r$, and the number of integers $j\in\{1,\ldots,N\}$ such that $L_{j}(p)\in\Fz$ and $j-1\in K$ is $\alpha s$. 

Recall that the statistics $\Tin{p}$ and $\Tout{p}$ only take values $0$ and $1$. For $0\leq i\leq N-1$, we have~$\Tout{p}(L_i)=1$ if and only if $p\in\min(L_i^{-1}(\Fo))$, and it follows from the definition of $q$-rowmotion that this occurs if and only if $i\in K$ and $L_i(p)\in\Fo$. Appealing to the previous paragraph, we find that $\sum_{i=0}^{N-1}\Tout{p}(L_i)=\alpha r$. Similarly, for $1\leq j\leq N$, we have $\Tin{p}(L_j)=1$ if and only if~$p\in\max(L_j^{-1}(\Fz))$, and it follows from the definition of $q$-rowmotion that this occurs if and only if $j-1\in K$ and $L_{j}(p)\in\Fz$. Appealing to the previous paragraph, we find that $\sum_{j=1}^{N}\Tin{p}(L_i)=\alpha s$. Consequently, the sum of $\qT{p}$ along the $q$-rowmotion orbit containing $L$ is 
\begin{align*}
\sum_{i=0}^{N-1}\qT{p}(L_i) &= \sum_{i=0}^{N-1}(\Tin{p}(L_i)-q\Tout{p}(L_i)) \\
&= \sum_{i=0}^{N-1}\Tin{p}(L_i)-q\sum_{j=1}^N\Tout{p}(L_j) \\
&= \alpha r-q\alpha s.
\end{align*}
Evaluating at $q=r/s$, we find that the sum of $\qT{p}\vert_{r/s}$ along the $q$-rowmotion orbit containing $L$ is~$\alpha r-(r/s)\alpha s=0$.
\end{proof}

\begin{remark}
Let ${\sf f}: \J_{r,s}(P) \rightarrow \RR$ be given by $\stat{f}(L) = \#L^{-1}(\Fz)$, the cardinality of the order ideal associated with $L$. We point out that the rescaled signed toggleability statistics $-\frac{1}{r+s}\T{p}^q$ (which could equally well serve as building blocks for 0-mesies) satisfy 
\[-\,\frac{1}{r+s}\T{p}^q =  {\sf f} - \frac{{\sf f} + {\sf f} \circ \tog{p} + \ \dots + \ {\sf f} \circ \tog{p}^{r+s-1}}{r+s},\] 
where the right-hand side is the difference between the value of ${\sf f}$ and the average value of ${\sf f}$ within the $\tog{p}$ orbit of $L$. Relatedly, the statistic $\qT{p}\vert_{r/s}$ is 0-mesic under the action of $\tog{p}$.
\end{remark}

Because of \cref{lem:qStriker}, we will want to be able to specialize $\RR(q)$-linear combinations of the $\qT{p}$ at positive rational values of $q$. The next lemma is a technical result saying that our rational functions never have singularities at nonnegative real numbers in the situations we care about.

\begin{lemma} \label{lem:singularities}
Consider a real-valued statistic $\stat{f}\colon \J(P)\to\RR$. Suppose that, viewing $\stat{f}$ also as a statistic $\stat{f}\colon \J(P)\to\RR(q)$, we have $\stat{f} = c(q) + \sum_{p\in P}c_p(q) \qT{p}$ for $c(q),c_p(q)\in \RR(q)$. Then none of the $c(q),c_p(q)$ have singularities at any nonnegative real number. Hence, for any $z \in \RR$ with $z\geq 0$, we have $\stat{f} = c(z) + \sum_{p\in P}c_p(z) \qT{p}\vert_{z}$.
\end{lemma}

\begin{proof}
Let $\stat{f}$ be as in the statement of the lemma, and let $z$ be a nonnegative real number. By clearing denominators, we can rewrite the expression for $\stat{f}$ as
\[h(q)\stat{f}=d(q)+\sum_{p\in P}d_p(q)\qT{p},\]
where $h(q)$, $d(q)$, and the $d_p(q)$ are polynomials in $q$. Explicitly, we have $d(q)=c(q)h(q)$ and $d_p(q)=c_p(q)h(q)$. We may assume the polynomials $h(q)$, $d(q)$, and~$d_p(q)$ for all $p\in P$ have no common factor; in particular, these polynomials do not all vanish at $q=z$. Let us specialize the above equation at $q=z$ to obtain 
\[h(z)\stat{f}=d(z)+\sum_{p\in P}d_p(z)\qT{p}\vert_{z}\] 
(note that $\stat{f}\vert_{z}$ is the same as $\stat{f}$ since $\stat{f}$ is a real-valued statistic). \Cref{thm:linearindependence} tells us that the statistics $\qT{p}\vert_{z}$ are linearly independent over $\RR$ and that they are linearly independent from $1$. (We used different notation in \cref{thm:linearindependence}; what we call $\qT{p}\vert_{z}$ here is what we called $\T{p}^{z}$ in that theorem.) This implies that $h(z)$ cannot be $0$. It follows that the rational functions $c(q)=d(q)/h(q)$ and $c_p(q)=d_p(q)/h(q)$ do not have singularities at $q=z$, which completes the proof.
\end{proof}

If $\stat{f}, \stat{g}\colon \J(P)\to\RR(q)$ are two $\RR(q)$-valued statistics on $\J(P)$, we write $\stat{f} \qtequiv \stat{g}$ to mean that there exist $c_p(q) \in \RR(q)$ for $p\in P$ such that $\stat{f}-\stat{g} = \sum_{p\in P}c_p(q)\T{p}^q$. The equivalence relation~$\qtequiv$ is a congruence on $\RR(q)$ in the sense that $\stat{f}_1 \qtequiv \stat{g}_1$ and $\stat{f}_2 \qtequiv \stat{g}_2$ imply $a_1(q)\cdot \stat{f}_1+a_2(q)\cdot\stat{f}_2\qtequiv a_1(q)\cdot\stat{g}_1+a_2(q)\cdot\stat{g}_2$ for all $a_1(q),a_2(q)\in\RR(q)$.

\begin{prop} \label{prop:q_homo}
Consider a real-valued statistic $\stat{f}\colon \J(P)\to\RR$. Suppose that $\stat{f}\qtequiv c(q)$ for some $c(q)\in\RR(q)$. Then for any positive integers $r$ and $s$ and any choice of cyclic permutation $\theta\colon\Fz\cup\Fo\to \Fz\cup\Fo$, the statistic $\stat{f}\colon \J_{r,s}(P)\to\RR$ is $c(r/s)$-mesic for the $q$-rowmotion operator $\rowm\colon \J_{r,s}(P)\to \J_{r,s}(P)$.   
\end{prop}

\begin{proof} 
By hypothesis, there are rational functions $c_p(q)\in\RR(q)$ such that 
\[\stat{f}=c(q)+\sum_{p\in P}c_p(q)\qT{p}.\]
\Cref{lem:singularities} tells us that we may specialize all of these rational functions to $q \coloneqq r/s$ and that
\[\stat{f}=c(r/s)+\sum_{p\in P}c_p(r/s)\qT{p}\vert_{r/s}.\]
Then \cref{lem:qStriker} yields the desired homomesy result.
\end{proof}

\begin{remark}
When defining the toggles $\tog{p}$ and, therefore, the $q$-rowmotion operator $\rowm$, we fix a single cyclic permutation $\theta$ of the set $\Fz\cup\Fo$. However, one could consider a more general setup in which each $\tog{p}$ is defined via a ``local'' cyclic permutation $\theta_p\colon\Fz\cup\Fo\to \Fz\cup\Fo$, with $\theta_p$ not necessarily equal to $\theta_{p'}$ when $p\neq p'$. The proofs of \cref{lem:qStriker} and \cref{prop:q_homo} apply in this more general setup. All of our results in this section rest on \cref{lem:qStriker} and \cref{prop:q_homo}, so they all hold in this more general setup as well. We have restricted to the setting in which all the~$\theta_p$ are equal for the sake of simplicity. 
\end{remark}

\begin{example} \label{ex:Exam2_cont}
As in \cref{ex:Exam2}, suppose $P=\arootp{2}$. Let $\stat{f}\coloneqq \Tout{1,2}-\Tout{2,1}$. Then one can check
\[ \stat{f}=0 + \left(-\frac{1}{1+q}\qT{1,2} + \frac{1}{1+q}\qT{2,1}\right).\]
(In fact, we will give a generalization $\stat{f}\colon \J(\arootp{n})\to \RR$ of this statistic for any Type~A root poset and prove that there is always some $c(q)\in \RR(q)$ with $\stat{f}\qtequiv c(q)$ in \cref{thm:a_q_stat} below.) We can look at \cref{fig:q-orbits} and see that the statistic $\stat{f}$ is indeed $0$-mesic under $(1/2)$-rowmotion, in agreement with \cref{prop:q_homo}.
\end{example}

Let us use the notation $\stat{f}\qtequiv \const(q)$ as shorthand for ``there exists $c(q)\in\RR(q)$ such that $\stat{f}\qtequiv c(q)$.'' In light of \cref{prop:q_homo}, our attention now turns to showing that various statistics $\stat{f}\colon \J(P)\to \RR$ satisfy $\stat{f}\qtequiv \const(q)$ for the posets we studied in detail in \cref{sec:main} above. We will not separately state as a corollary every time that this proves the statistic is homomesic under $q$-rowmotion, but that is of course our motivation.

\begin{remark}
Unlike the situation with the PL and birational extensions discussed in~\cref{sec:pl_birational}, we have no procedure for automatically lifting a proof that $\stat{f} \tequiv \const$ to a proof that $\stat{f} \qtequiv \const(q)$. (\Cref{lem:singularities} shows that we can always go in the opposite direction, that is, specialize $q \coloneqq 1$.) However, we will see that many of the $\stat{f}$ studied in \cref{sec:main} do satisfy $\stat{f} \qtequiv \const(q)$. Moreover, the tools employed in \cref{sec:main} (i.e., rooks) remain extremely useful in the $q$-ified world.
\end{remark}

\subsection{The rectangle}

In this subsection, we fix $P=\rect{a}{b}$. Recall the rooks $\rook{i,j}$ for the rectangle defined in~\eqref{eqn:rect_rook}. 

Our goal is to show that the antichain cardinality statistic is $\qtequiv \const(q)$.

\begin{thm} \label{thm:aboapbq} 
For $P=\rect{a}{b}$, we have $\sum_{p\in P} \Tout{p} \qtequiv \frac{[a]_q[b]_q}{[a+b]_q}$.
\end{thm}

\begin{proof}
The proof is based on a careful analysis of a certain $\RR(q)$-linear combination of rooks. From \cref{thm:rect_sumone}, it is clear that
\begin{equation} \label{eqn:q_rook_sum}
\sum_{\substack{1 \leq i \leq a, \\ 1\leq j \leq b}} q^{(i-1)+(j-1)}R_{i,j} = [a]_q [b]_q.
\end{equation}
Indeed, \cref{thm:rect_sumone} says that the left-hand side of \eqref{eqn:q_rook_sum} is $\sum_{i,j}q^{(i-1)+(j-1)}$, which is the same as the product $(1+q+\cdots+q^{a-1})\cdot(1+q+\cdots+q^{b-1}) = [a]_q  [b]_q$. Now we will try to expand the left-hand side of \eqref{eqn:q_rook_sum} in terms of the toggleability statistics $\Tin{i,j}$ and $\Tout{i,j}$.

Fix some box $(i,j) \in [a]\times [b]$. Looking at the definition~\eqref{eqn:rect_rook} of the rooks, we see that a rook~$R_{i',j'}$ contributes a term of $+\Tin{i,j}$ if and only if $i' \geq i, j' \geq j$; contributes a term of $-\Tin{i,j}$ if and only if $i' \leq i-1, j' \leq j-1$; contributes a term of $+\Tout{i,j}$ if and only if~$i' \leq i, j' \leq j$; and contributes a term of $-\Tout{i,j}$ if and only if $i' \geq i+1, j' \geq j+1$. Hence, all the terms corresponding to $(i,j)$ in the left-hand side of \eqref{eqn:q_rook_sum} are
\[ \left( \sum_{\substack{i\leq i'\leq a, \\ j \leq j'\leq b}} q^{i'+j'-2} \ \ - \sum_{\substack{1\leq i' \leq i-1,\\1\leq j'\leq j-1}} q^{i'+j'-2}\right) \Tin{i,j} + \left( \sum_{\substack{1\leq i'\leq i, \\ 1\leq j'\leq j}} q^{i'+j'-2} \ \ - \sum_{\substack{i+1\leq i'\leq a,\\ j+1\leq j' \leq b}} q^{i'+j'-2} \right) \Tout{i,j}.\]
We can easily rewrite this expression in terms of $q$-numbers as
\begin{equation} \label{eqn:q_rook_sum_ij}
 ( q^{i+j-2}[a+1-i]_q[b+1-j]_q - [i-1]_q[j-1]_q) \Tin{i,j} + ( [i]_q[j]_q - q^{i+j}[a-i]_q[b-j]_q ) \Tout{i,j}.
\end{equation}
Some straightforward algebra shows that
\[ [i]_q[j]_q = q [i-1]_q[j-1]_q + [i+j-1]_q\]
and, similarly,
\[ -q^{i+j}[a-i]_q[b-j]_q = -q^{i+j-1}[a+1-i]_q[b+1-j]_q + q^{i+j-1}[a+b-(i+j-1)]_q.\]
Furthermore, $[i+j-1]_q+q^{i+j-1}[a+b-(i+j-1)]_q=[a+b]_q$. 
Applying these substitutions and using the definition of $\qT{i,j}$, we can show (after some work) that the expression in \eqref{eqn:q_rook_sum_ij} is also equal to
\[ ( q^{i+j-2}[a+1-i]_q[b+1-j]_q - [i-1]_q[j-1]_q) \qT{i,j} + [a+b]_q\Tout{i,j}.\]

From the previous paragraph, we conclude that
\begin{equation} \label{eqn:q_rook_sum_redux}
 \sum_{\substack{1 \leq i \leq a, \\1\leq j \leq b}} q^{(i-1)+(j-1)}R_{i,j} = \sum_{i,j} [a+b]_q\Tout{i,j} + \sum_{i,j}( q^{i+j-2}[a+1-i]_q[b+1-j]_q - [i-1]_q[j-1]_q) \qT{i,j}.
\end{equation}
Putting \eqref{eqn:q_rook_sum} and \eqref{eqn:q_rook_sum_redux} together, we see that
\[ [a]_q[b]_q \qtequiv [a+b]_q \cdot \sum_{i,j} \Tout{i,j}. \]
Multiplying both sides by $1/[a+b]_q \in \RR(q)$, we get $\sum_{p\in P} \Tout{p} \qtequiv \frac{[a]_q[b]_q}{[a+b]_q}$, as claimed.
\end{proof}

In fact, we can show that the fiber refinements of antichain cardinality are also~$\qtequiv \const(q)$.

\begin{thm} \label{thm:q_rectfiber}
If $B=\{i\}\times[b] \subseteq \rect{a}{b}$ is a positive fiber, then $\sum_{p \in B} \Tout{p} \qtequiv q^{a-i}\frac{[b]_q}{[a+b]_q}$. If $B=[a] \times \{j\}$ is a negative fiber, then $\sum_{p \in B} \Tout{p} \tequiv q^{b-j}\frac{[a]_q}{[a+b]_q}$.
\end{thm}

\begin{proof}
By symmetry, we may assume that $B=\{i\}\times[b] \subseteq \rect{a}{b}$ is a positive fiber. Set $\stat{f}_i\coloneqq \sum_{j} \Tout{i,j}$ for $1\leq i \leq a$. For any $1\leq i \leq a-1$, we have
\[0=\rook{i,1}-\rook{i+1,1} = \sum_{j}\Tout{i,j} - \sum_{j} \Tin{i+1,j},\]
where the first equality follows from \cref{thm:rect_sumone} and the second is a matter of straightforward algebra. (Alternatively, we can see that $\sum_{j}\Tout{i,j} = \sum_{j} \Tin{i+1,j}$ without the use of rooks: the equation is just asserting that we can toggle a box out of the $i$-th positive fiber if and only if we can toggle a box into the $(i+1)$-st positive fiber, which we also mentioned in the proof of \cref{thm:rectfiber}.) Adding~$\sum_{j} \qT{i+1,j}$ to the above equation yields
\[ \sum_{j} \qT{i+1,j} = \sum_{j}\Tout{i,j} - q\cdot \sum_{j}\Tout{i+1,j}; \]
in other words, $\stat{f}_i \qtequiv q\cdot \stat{f}_{i+1}$. Hence, $\stat{f}_1 \qtequiv q^{i-1} \cdot \stat{f}_i$ and $\stat{f}_i \qtequiv q^{1-i} \cdot \stat{f}_1$ for every $1\leq i \leq a$. 

Now, we know
from \cref{thm:aboapbq} that \[\stat{f}_1+\cdots+\stat{f}_a=\sum_{p\in P} \Tout{p} \qtequiv \frac{[a]_q[b]_q}{[a+b]_q},\]
which, using $\stat{f}_i \qtequiv q^{1-i} \cdot \stat{f}_1$, we can rewrite as
\[q^{0}\stat{f}_1+q^{-1}\stat{f}_1+\cdots+q^{-(a-1)}\stat{f}_1 \qtequiv \frac{[a]_q[b]_q}{[a+b]_q}.\]
Dividing both sides by $q^{-(a-1)}[a]_q$ gives $\stat{f}_1 \qtequiv  q^{a-1} \cdot \frac{[b]_q}{[a+b]_q}$. Finally, using $\stat{f}_1 \qtequiv q^{i-1} \cdot \stat{f}_i$ we obtain  $\stat{f}_i \qtequiv q^{a-i} \frac{[b]_q}{[a+b]_q}$, as desired.
\end{proof}

Observe how \cref{thm:aboapbq,thm:q_rectfiber} are direct $q$-analogues of \cref{thm:aboapb,thm:rectfiber}. Based on what we did in \cref{sec:main}, it is also reasonable to ask if $\stat{f} \qtequiv \const(q)$ for some $\stat{f}\in \Span\{\oii{p}\}$. However, we have not been able to find any such statistics $\stat{f}$, beyond those that happen to also have expressions as sums of the $\Tout{p}$ such as
\[\oii{a,1} = \sum_{j} \Tout{a,j}\quad\text{and}\quad \oii{1,b} = \sum_{i} \Tout{i,b}.\]
Note that these are certainly $\qtequiv\const(q)$ thanks to \cref{thm:q_rectfiber}.

\subsection{The shifted staircase}

In this subsection, we fix $P=\sstair{n}$, and we recall the rooks~$\rook{i,j}$ for the shifted staircase defined in~\eqref{eqn:sstair_rook}. We also recall $\iota\colon\J(\sstair{n})\to\J(\rect{n}{n})$.

Our goal once again is to show that the antichain cardinality statistic is $\qtequiv \const(q)$.

\begin{lemma} \label{lem:q_sstair_fiber}
For $1\leq i \leq n$, we have $\sum_{j\leq i}\Tout{j,i} + \sum_{i < j}\Tout{i,j} \qtequiv q^{n-i}\frac{[n]_q}{[2n]_q}$ for $P=\sstair{n}$.
\end{lemma}
\begin{proof}
Computing the statistic $\sum_{j\leq i}\Tout{j,i} + \sum_{i < j}\Tout{i,j}$ on an order ideal $I\in\J(P)$ is the same as computing the statistic $\sum_{j} \Tout{i,j}\colon\J(\rect{n}{n})\to\RR$ on $\iota(I)$. Hence, the lemma follows from \cref{thm:q_rectfiber}.
\end{proof}

\begin{lemma} \label{lem:q_sstair_diag}
We have $\sum_{i}\Tout{i,i}\qtequiv \frac{1}{1+q}$ for $P=\sstair{n}$.
\end{lemma}
\begin{proof}
We have
\[\sum_{i}(\Tout{i,i}+\Tin{i,i})=1.\]
Indeed, this is just asserting that for any order ideal $I\in \J(\sstair{n})$, there is exactly one box on the main diagonal that we can either toggle in or toggle out of $I$. Subtracting $\sum_{i} \qT{i,i}$ from each side of this equation gives
\[1-\sum_{i} \qT{i,i} = (1+q)\sum_{i}\Tout{i,i}.\]
Dividing by $(1+q)$ exactly yields $\sum_{i}\Tout{i,i}\qtequiv \frac{1}{1+q}$, as desired.
\end{proof}

\begin{thm} \label{thm:q_sstair_a}
For $P=\sstair{n}$, we have $\sum_{p\in P}\Tout{p} = \frac{\qbinom{n+1}{2}_q}{[2n]_q}$.
\end{thm}
\begin{proof}
As we have already mentioned before \cref{thm:sstair_a},
\[ 2\cdot \sum_{p\in P}\Tout{p} = \sum_{1\leq i \leq n}\left(\sum_{j\leq i}\Tout{j,i} + \sum_{i < j}\Tout{i,j}\right) + \sum_{i}\Tout{i,i}.\]
Hence, thanks to \cref{lem:q_sstair_fiber,lem:q_sstair_diag}, there is some $c(q)\in\RR(q)$ such that $\sum_{p\in P}\Tout{p} \qtequiv c(q)$. With a little bit of algebraic manipulation, one can show
\[\sum_{1\leq i \leq n}q^{n-i}\frac{[n]_q}{[2n]_q} + \frac{1}{1+q} =2\cdot  \frac{\qbinom{n+1}{2}_q}{[2n]_q}.\]
(It helps to multiply both sides of the equation by $[2n]_q(1+q)$.) However, as we will explain in the proof of \cref{thm:q_min_a} below, we can say in general that for any graded poset $P$, if we have $\sum_{p\in P}\Tout{p} \qtequiv c(q)$, then necessarily $c(q) = \frac{\sum_{p\in P}q^{\rk(P)-\rk(p)}}{[\rk(P)+2]_q}$. For the shifted staircase we have $\sum_{p\in P}q^{\rk(P)-\rk(p)}=\qbinom{n+1}{2}_q$ and $[\rk(P)+2]_q=[2n]_q$.
\end{proof}

Observe that \cref{thm:q_sstair_a} is a direct $q$-analogue of \cref{thm:sstair_a}. As with the rectangle, the only $\stat{f}$ that are sums of order ideal indicator functions and that we can show are $\qtequiv \const(q)$ are ones that happen also to have expressions as sums of antichain indicator functions; namely,
\[\oii{1,n} = \sum_{i} \Tout{i,n}\quad\text{and}\quad \sum_{1\leq i \leq n} \oii{i,i} - \sum_{1\leq i \leq n-1} \oii{i,i+1} = \sum_{i}\Tout{i,i},\]
which are $\qtequiv \const(q)$ thanks to \cref{lem:q_sstair_fiber,lem:q_sstair_diag}.

\subsection{Other minuscule posets} 

We now prove a $q$-analogue of \cref{thm:min_a}, asserting that if $P$ is any minuscule poset, then the antichain cardinality statistic is~$\qtequiv \const(q)$.

\begin{thm} \label{thm:q_min_a}
Let $P$ be a minuscule poset. Then $\sum_{p\in P}\Tout{p}\qtequiv \frac{\sum_{p\in P}q^{\rk(p)}}{[\rk(P)+2]_q}$.
\end{thm}

\begin{proof}
First, let us observe that if $P$ is any graded poset and $\sum_{p\in P}\Tout{p}\qtequiv c(q)$ for some $c(q)\in\RR$, then we must have $c(q)= \frac{\sum_{p\in P}q^{\rk(P)-\rk(p)}}{[\rk(P)+2]_q}$. Indeed, suppose that $\sum_{p\in P}\Tout{p} = c(q) + \sum_{p\in P} c_p(q) \qT{p}$, and set $I_i \coloneqq  \{p\in P\colon \rk(p)\leq i\}$ for $i=-1,0,1\ldots,\rk(P)$. It is straightforward to verify that we have $\sum_{i=-1}^{\rk(P)} q^{\rk(P)-i} \cdot \qT{p}(I_i) = 0$ for all $p \in P$, because the only nonzero terms will be when $i=\rk(p)-1$ and $i=\rk(p)$, and these terms will cancel. Hence,
\begin{align*}
\sum_{p\in P}q^{\rk(P)-\rk(p)} &= \sum_{i=-1}^{\rk(P)} q^{\rk(P)-i} \cdot \#\max(I_i) \\
&= \sum_{i=-1}^{\rk(P)} q^{\rk(P)-i} \cdot \sum_{p\in P}\Tout{p}(I_i) \\
&= \sum_{i=-1}^{\rk(P)} q^{\rk(P)-i} \cdot \left( c(q) + \sum_{p\in P} c_p(q) \qT{p}(I_i) \right) \\
&=  \sum_{i=-1}^{\rk(P)} q^{\rk(P)-i} \cdot c(q) \\
&= c(q) \cdot [\rk(P)+2]_q.
\end{align*} 
So dividing by $[\rk(P)+2]_q$, we see that $c(q)= \frac{\sum_{p\in P}q^{\rk(P)-\rk(p)}}{[\rk(P)+2]_q}$.

Now let $P$ be a minuscule poset. Since $P$ is graded, we know by the preceding paragraph that~$\frac{\sum_{p\in P}q^{\rk(P)-\rk(p)}}{[\rk(P)+2]_q}$ is the right element of $\RR(q)$ for $\sum_{p\in P}\Tout{p}$ to be $\qtequiv$ to. Since $P$ is self-dual, this is in fact the same as $\frac{\sum_{p\in P}q^{\rk(p)}}{[\rk(P)+2]_q}$. Thus, we only need to show that there is some $c(q) \in \RR(q)$ such that $\sum_{p\in P}\Tout{p}\qtequiv c(q)$. We have already shown this for the rectangle and the shifted staircase (\cref{thm:aboapbq} and \cref{thm:q_sstair_a}), so we need only address $\dtd{n}$, $\esixmin$, and $\esevmin$. 

First, consider $P=\dtd{n}$. Let $x_1<\cdots<x_{n-1}$ be the elements in the initial tail, $y_1,y_2$ be the two incomparable elements in the middle, and $z_{n-1}<\cdots<z_{1}$ the elements in the final tail. Then we claim that
\[(q^n+1)\sum_{p\in P} \Tout{p} = (q^{n-1}+1)-\sum_{i=1}^{n-1}(q^{n-1}+q^{i-1})\qT{x_i} - q^{n-1}\qT{y_1} - q^{n-1}\qT{y_2} -\sum_{i=1}^{n-1}(q^{n-1}-q^{n-i-1})\qT{z_i}.\]
As we mentioned in the proof of \cref{thm:min_a}, the order ideals in $\J(\dtd{n})$ are divided into a very small number of cases, and we can verify this equation by checking these cases.

Now consider $P=\esixmin$ or $\esevmin$. In this case, we checked by computer that the relevant system of equations has a solution (and the Sage code worksheet referenced in the acknowledgments includes this verification).
\end{proof}

\subsection{Root posets}

To conclude this section, let us consider the root posets $\arootp{n}$ and $\brootp{n}$. For these, there are apparently far fewer statistics $\stat{f}\colon \J(P)\to \RR$ with $\stat{f}\qtequiv \const(q)$, compared to $\rect{a}{b}$ and $\sstair{n}$.

First consider $\arootp{n}$, and recall the rooks $\rook{i}$ and reduced rooks $\rrook{i}$ for the Type~A root poset defined in~\eqref{eqn:a_rook} and~\eqref{eqn:a_red_rook}. For $\arootp{n}$, there is only one interesting function $\stat{f}\colon \J(P)\to \RR$ we can show is $\qtequiv \const(q)$, namely
\begin{equation} \label{eqn:a_q_stat_red}
\frac{1}{2}\left( \rrook{1}-\rrook{2}+\rrook{3}-\cdots\pm \rrook{n} \right) = \sum_{\substack{(i,j)\in\arootp{n}, \\ i+j\equiv n+1 \!\!\!\mod 2}} (-1)^{i-1}\, \Tout{i,j}.
\end{equation}
The pattern of antichain indicator functions on the right-hand side of~\eqref{eqn:a_q_stat_red} is depicted in \cref{fig:a_q_stat}. To verify that~\eqref{eqn:a_q_stat_red} holds is simply a matter of keeping track of how many times each box is attacked.

\begin{figure}
\begin{center}
\includegraphics[height=4.717cm]{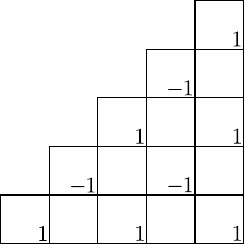}
\end{center}
\caption{The statistic $\frac{1}{2}\sum_{i=1}^{n}(-1)^{i-1}\, \rrook{i}$ on $\arootp{n}$. Here, $n=5$.} \label{fig:a_q_stat}
\end{figure}

\begin{thm}\label{thm:a_q_stat}
For $P=\arootp{n}$, we have
\[ \frac{1}{2}\left( \rrook{1}-\rrook{2}+\rrook{3}-\cdots\pm \rrook{n} \right) \qtequiv \begin{cases}0 &\textrm{if $n$ is even}; \\ \frac{1}{q+1} &\textrm{if $n$ is odd}.  \end{cases} \]
\end{thm}
\begin{proof}
We claim that
\begin{equation} \label{eqn:a_q_stat}
 \rook{1}-\rook{2}+\rook{3}-\cdots\pm \rook{n} =  \sum_{\substack{(i,j)\in\arootp{n}, \\ i+j\equiv n+1 \!\!\!\mod 2}} (-1)^{i-1}(\Tout{i,j} + \Tin{i,j}).
\end{equation}
The pattern of toggleability statistics on the right-hand side of~\eqref{eqn:a_q_stat} is the same as what is depicted in \cref{fig:a_q_stat}, except that everywhere we have a $\pm 1$ in the lower right corner of a box, we will also have a $\pm 1$ in the upper left corner of that box. 

To see why~\eqref{eqn:a_q_stat} holds, let $(i,j)\in \arootp{n}$. By inspecting the definition~\eqref{eqn:a_rook} of the rooks, we see that the rooks with a nonzero coefficient of $\Tout{i,j}$ are exactly $\rook{k}$ for $k=(n+1-j),\ldots,i$, and~$\rook{k}$ contributes a coefficient of $(-1)^{k-1}$ to the left-hand side of~\eqref{eqn:a_q_stat}. Hence, the overall coefficient of~$\Tout{i,j}$ in the left-hand side of~\eqref{eqn:a_q_stat} is
\[ \sum_{k=n+1-j}^{i} (-1)^{k-1} = \begin{cases} 0 &\textrm{ if $i+j \not\equiv n+1 \!\!\!\mod 2$}; \\ (-1)^{i-1} &\textrm{ if $i+j \equiv n+1 \!\!\!\mod 2$}. \end{cases}\]
Similarly, if $i+j > n +1$, then we see that the rooks with a nonzero coefficient of $\Tin{i,j}$ are exactly~$\rook{k}$ for $k=(n+1-j)+1,\ldots,i-1$, and $\rook{k}$ contributes a coefficient of $(-1)^{k}$ to the left-hand side of~\eqref{eqn:a_q_stat}. Hence, the overall coefficient of $\Tin{i,j}$ in the left-hand side of~\eqref{eqn:a_q_stat} is
\[ \sum_{k=(n+1-j)+1}^{i-1} (-1)^k = \begin{cases} 0 &\textrm{ if $i+j \not\equiv n+1 \!\!\!\mod 2$}; \\ (-1)^{i-1} &\textrm{ if $i+j \equiv n+1 \!\!\!\mod 2$}. \end{cases}\]
If $i+j=n+1$, then $R_i$ is the only rook that contributes a coefficient of $\Tin{i,j}$, and it contributes a coefficient of $(-1)^{i}$ to the left-hand side of~\eqref{eqn:a_q_stat}. Either way, the overall term corresponding to~$(i,j)$ is $(-1)^{i-1}(\Tin{i,j}+\Tout{i,j})$ if $i+j \equiv n+1 \!\!\!\mod 2$ and is~$0$ otherwise. So indeed~\eqref{eqn:a_q_stat} holds.

Now, \cref{thm:a_sumone} says that
\[ \rook{1}-\rook{2}+\rook{3}-\cdots\pm \rook{n} =\begin{cases} 0 &\textrm{if $n$ is even}; \\ 1 &\textrm{if $n$ is odd}. \end{cases} \]
Hence, we conclude from~\eqref{eqn:a_q_stat} that
\[ \sum_{\substack{(i,j)\in\arootp{n}, \\ i+j\equiv n+1 \mod 2}} (-1)^{i-1}(\Tout{i,j} + \Tin{i,j}) = \begin{cases} 0 &\textrm{if $n$ is even}; \\ 1 &\textrm{if $n$ is odd}. \end{cases}\]
Adding $\sum_{i+j\equiv n+1\!\!\!\mod 2} (-1)^{i} \, \qT{i,j}$ to this equation yields
\[  (1+q) \sum_{\substack{(i,j)\in\arootp{n}, \\ i+j\equiv n+1 \!\!\!\mod 2}} (-1)^{i-1} \, \Tout{i,j} =\sum_{i+j\equiv n+1\!\!\!\mod 2} (-1)^{i} \, \qT{i,j}+\begin{cases} 0 &\textrm{if $n$ is even}; \\ 1 &\textrm{if $n$ is odd}. \end{cases}\]
Finally, dividing by $(1+q)$ proves the theorem.
\end{proof}

Now consider $\brootp{n}$, and recall the various rooks and reduced rooks $\rook{i}$, $\rrook{i}$, $\varrook{i}$, $\rvarrook{i}$ defined on the Type~B root poset in~\eqref{eqn:b_rook}--\eqref{eqn:b_red_var_rook}. Also recall $\iota\colon \J(\brootp{n})\to\J(\arootp{2n-1})$. 

We get one statistic that is $\qtequiv \const(q)$ for $\brootp{n}$ via $\iota$.

\begin{thm} \label{thm:b_q_stat}
For $P=\brootp{n}$, we have $\sum_{i=1}^{n-1}(-1)^{i-1}\, \rvarrook{i} + \frac{(-1)^{n-1}}{2}\rvarrook{n}\qtequiv \frac{1}{q+1}$.
\end{thm}
\begin{proof}
This is directly what we get by translating the statistic in \cref{thm:a_q_stat} from $\arootp{2n-1}$ to $\brootp{n}$ via $\iota$.
\end{proof}

The statistic in \cref{thm:b_q_stat} is also equal to a checkerboard pattern of antichain indicator functions, similar to what is depicted in \cref{fig:a_q_stat}, although some of the coefficients will be $\pm 2$ instead of $\pm 1$.

There is one additional statistic we can show is $\qtequiv \const(q)$ for $\brootp{n}$.

\begin{prop} \label{prop:b_q_diag}
For $P=\brootp{n}$, we have $\sum_{i} \Tout{i,i} \qtequiv \frac{1}{q+1}$.
\end{prop}
\begin{proof}
As with the shifted staircase in the proof of~\cref{lem:q_sstair_diag}, it is easy to see for $\brootp{n}$ that
\[1 = \sum_{i}(\Tout{i,i} + \Tin{i,i}),\]
using rooks or otherwise. Then we can apply the same argument as in the proof of~\cref{lem:q_sstair_diag}.
\end{proof}

Finally, we note that for $\arootp{n}$, there are apparently no combinations of order ideal indicator functions that are $\qtequiv \const(q)$; while for $\brootp{n}$, we have
\[\sum_{i} \oii{i,i} - \sum_{i}\oii{i,i+1} = \sum_{i}\Tout{i,i},\]
as with the shifted staircase, so we get one such statistic from \cref{prop:b_q_diag}.

\section{Final remarks and future directions} \label{sec:final}

We end with some final thoughts and possible directions for future research.

\subsection{Comparison with other ways to exhibit homomesies}

Prior to the toggleability statistics technique, the main way that one would prove a rowmotion homomesy result is by finding a good \emph{model} for rowmotion. By ``model,'' we mean a bijection to some other combinatorial action, which is often some kind of ``rotation.'' For example, the \emph{Stanley-Thomas word} bijection~\cite[3.3.2]{propp2015homomesy} transports rowmotion of $\rect{a}{b}$ to rotation of binary words. And in~\cite[Main Theorem]{armstrong2013uniform}, the authors produce a bijection that transports rowmotion of $\arootp{n}$ to rotation of noncrossing matchings on~$[2n]$. See~\cite{striker2012promotion} for more examples of models of rowmotion.

We have seen that the toggleability statistics technique is broadly applicable, continues to work in modified contexts, and is relatively mindless to apply. However, it certainly does not render the approach of finding good models for rowmotion obsolete. This is because models can be used for much more than just homomesy: for instance, they can also tell us the order and orbit structure of the action. Some homomesy results may on their own imply some limited information about orbit sizes, because if an integer-valued statistic is $c$-mesic for some action, where $c=r/s$ in lowest terms, then every orbit must have size a multiple of $s$. But to obtain a really precise understanding of rowmotion on a given poset, one needs a good model.

We also note, however, that proving $\stat{f}\tequiv \const$ for some combinatorially significant function $\stat{f}\colon \J(P)\to\RR$ can have consequences quite separate from rowmotion and homomesy. Let us say briefly why this is. In~\cite{chan2017expected}, a probability distribution~$\mu$ on $\J(P)$ is called \dfn{toggle-symmetric} if $\mathbb{E}_{\mu}(\T{p})=0$ for all~$p\in P$.\footnote{For a probability distribution $\mu$ on $\J(P)$ and a statistic $\stat{f}\colon \J(P)\to \RR$, we use $\mathbb{E}_{\mu}(\stat{f})$ to denote the expectation of the random variable $\stat{f}(I)$ when $I\in\J(P)$ is distributed according to $\mu$.} For instance, \cref{lem:striker} implies that any distribution that is constant on rowmotion orbits is toggle-symmetric. But in~\cite[\S2]{chan2017expected}, it is shown that many other distributions on~$\J(P)$, especially distributions defined in terms of tableaux and related combinatorial objects, are also toggle-symmetric. If~$\stat{f}\tequiv c$ for some $c\in \RR$, then $\mathbb{E}_{\mu}(\stat{f})=c$ for any toggle-symmetric distribution~$\mu$ on~$\J(P)$. As a result, showing $\stat{f}\tequiv \const$ for certain $\stat{f}$ can lead to enumerative results about tableaux: see~\cite{reiner2018poset}. Moreover, showing $\stat{f}\qtequiv \const(q)$ can similarly yield results about the $q$-enumeration of tableaux: see the very recent paper~\cite{hopkins2021qenumeration}.
In that paper, a probability distribution~$\mu$ on $\J(P)$ is called \dfn{$q$-toggle-symmetric} if $\mathbb{E}_{\mu}(\T{p}^q)=0$ for all~$p\in P$.
Equivalently, a probability distribution on order ideals $I$ is $q$-toggle-symmetric if and only if for every poset element $p$, the conditional probability that $p$ is in $I$ given that $p$ is active in $I$ is $q/(1+q)$.
The paper gives numerous examples of $q$-toggle-symmetric probability distributions.

\subsection{Other posets where the technique works}

The posets we have focused on are by no means the only ones where it is possible to show that $\stat{f}\tequiv \const$ for combinatorially significant $\stat{f}$. For instance, consider the antichain cardinality statistic. In~\cite{chan2017expected}, it is shown that $\sum_{p\in P}\Tout p \tequiv \const$ for all so-called \emph{balanced} Young diagram shape posets~$P$. Here ``balancedness'' is a certain condition on the corners of the Young diagram. (It is basically the condition one needs for it to be possible to place rooks on the boxes of the diagram in such a way that every box is attacked the same number of times.) There are many balanced shapes beyond the rectangle. Similarly, in~\cite{hopkins2017cde}, it is shown that $\sum_{p\in P}\Tout p \tequiv \const$ for all \emph{shifted balanced} $P$. There are many shifted balanced shapes beyond the shifted staircase.

We have focused exclusively on the handful of posets that have really good rowmotion behavior, in the sense of finite order of piecewise-linear and birational rowmotion. Essentially the only known, interesting such examples are minuscule posets and root posets of coincidental type; see~\cite[\S13]{grinberg2015birational2}.

We should also note that in the $q$-ified world, there are many fewer examples. Apparently, the only Young diagram shapes with $\sum_{p\in P}\Tout p \qtequiv \const(q)$ are rectangles, and the only shifted shapes with $\sum_{p\in P}\Tout p \qtequiv \const(q)$ are shifted staircases. We have checked this by computer for shapes with at most $20$ boxes.

\subsection{Description of the toggleability spaces}

As we have explained several times now, our interest has been in statistics $\stat{f}\in\Span(\{1\}\cup\{\T{p}\colon p\in P\})$ that also belong either to $\Span\{\Tout{p}\colon p \in P\}$ or to $\Span\{\oii{p}\colon p \in P\}$. Let us use the shorter notations $\{f\tequiv \const\} \coloneqq \Span(\{1\}\cup\{\T{p}\colon p\in P\})$ and $V_A \coloneqq \Span\{\Tout{p}\colon p \in P\}$ and $V_I \coloneqq \Span\{\oii{p}\colon p \in P\}$. We informally refer to $\{\stat{f}\tequiv \const\} \cap V_A$ and $\{\stat{f}\tequiv \const\} \cap V_I$ as \dfn{toggleability spaces}.

In \cref{tab:vect_spaces}, we see conjectural descriptions of the toggleability spaces for the four main posets we have considered. We also see conjectural descriptions of the $q$-analogue spaces $\{\stat{f}\qtequiv \const(q)\} \cap V_A$ and $\{\stat{f}\qtequiv \const(q)\} \cap V_I$, where $\{\stat{f}\qtequiv \const(q)\}\coloneqq \{\stat{f}\colon \J(P)\to \RR\colon \stat{f}\qtequiv c(q) \textrm{ for some } c(q)\in\RR(q)\}$.

\begin{table}
\begin{adjustwidth}{-0.99in}{-0.99in}
\centering
\renewcommand{\arraystretch}{1.5}
\begin{tabular}{c | c | c | c | c}
Poset $P$ & $\{\stat{f}\tequiv \const\} \cap V_A$ & $\{\stat{f}\tequiv \const\} \cap V_I$ & $\{\stat{f}\qtequiv \const(q)\}  \cap V_A$ &  $\{\stat{f}\qtequiv \const(q)\}  \cap V_I$ \\ \hline
$\rect{a}{b}$ & \parbox{1.35in}{ \vspace{.5\baselineskip} {\bf dimension:} $a+b-1$ \\ {\bf basis:} \\$\{\sum_{j} \Tout{i,j}\colon i\in [a]\}$\\$\cup \{\sum_{i} \Tout{i,j}\colon j \in [b]\}$\\ minus any element } & \parbox{1.5in}{ {\bf dimension}: $a+b-1$ \\ {\bf basis}: \\ $\{\sum_{j-i=k} \oii{i,j}\colon$ \\$k=-a+1, \ldots, b-1\}$ } & \parbox{1.35in}{  \vspace{.5\baselineskip} {\bf dimension}: $a+b-1$ \\ {\bf basis}: \\$\{\sum_{j} \Tout{i,j}\colon i\in [a]\}$ \\ $\cup \{\sum_{i} \Tout{i,j}\colon j \in [b]\}$ \\ minus any element } &  \parbox{1.25in}{ {\bf dimension}: $2$ \\ {\bf basis}: \\$\{\oii{a,1} = \sum_{j} \Tout{a,j},$\\ $ \oii{1,b} = \sum_{i} \Tout{i,b}\}$ } \\ \hline
$\sstair{n}$ & \parbox{1.35in}{ \vspace{.5\baselineskip} {\bf dimension}: $2n-1$ \\ {\bf basis}: \\ $\{\rrook{i,i}\colon i\in[n]\} \cup$ \\ $\{\rrook{i,i+1}\colon i \in [n-1]\}$ \\ with $\rrook{i,j}$ as in~\eqref{eqn:sstair_red_rook}} & \parbox{1.55in}{  {\bf dimension}: $2n-1$ \\ {\bf basis}: \\ $\{\sum_{j-i=k} \oii{i,j}\colon$ \\$k=0,\ldots,n-1\} \cup$ \\$\{\textrm{$n-1$ more functions}\}$}& \parbox{1.5in}{ {\bf dimension}: $n+1$ \\ {\bf basis}: \\ $\{\sum_{j\leq i}\Tout{j,i} + \sum_{i < j}\Tout{i,j}\colon$ \\ $ i \in [n]\} \cup \{\sum_{i} \Tout{i,i}\}$ } &  \parbox{1.25in}{ {\bf dimension}: 2 \\ {\bf basis}: \\ $\{\oii{1,n} = \sum_{i} \Tout{i,n},$ \\$ \sum_{i} \oii{i,i} - \sum_{i} \oii{i,i+1} $ \\ $= \sum_{i} \Tout{i,i}\}$} \\ \hline
$\arootp{n}$ &   \parbox{1.35in}{ \vspace{.5\baselineskip} {\bf dimension}: $n$ \\ {\bf basis}: \\ $\{\rrook{i}\colon i\in[n]\}$ with $\rrook{i}$ as in~\eqref{eqn:a_red_rook}} & \parbox{1.5in}{\vspace{.5\baselineskip} {\bf dimension}: $n$ \\ {\bf basis}: \\ $\{2\cdot\sum_{j-i=k} \oii{i,j}$ \\$  - \sum_{j-i=k-1} \oii{i,j}$ \\ $-\sum_{j-i=k+1} \oii{i,j}\colon$\\$ k=(n-1), (n-3),$\\$ \ldots,-(n-3),-(n-1)\}$} &\parbox{1.5in}{{\bf dimension}: 1 \\ {\bf basis}: \\ $\{\frac{1}{2}(\rrook{1}-\rrook{2}+\cdots \pm \rrook{n})\}$ with $\rrook{i}$ as in~\eqref{eqn:a_red_rook} } &\parbox{1.25in}{{\bf dimension}: $0$} \\ \hline
$\brootp{n}$ & \parbox{1.35in}{ {\bf dimension}: $2n-1$ \\ {\bf basis}: \\ $\{\rrook{i}\colon i=1,\ldots,n\} \cup$ \\ $\{\rvarrook{i}\colon i = 2,\ldots,n\}$ \\ with $\rrook{i}$ and $\rvarrook{i}$ as in~\eqref{eqn:b_red_rook} and~\eqref{eqn:b_red_var_rook}} & \parbox{1.55in}{ \vspace{.5\baselineskip} {\bf dimension}: $2n-1$ \\ {\bf basis}: \\  $\{2\cdot\sum_{j-i=k} \oii{i,j}$ \\$  - \sum_{j-i=k-1} \oii{i,j}$ \\ $-\sum_{j-i=k+1} \oii{i,j}\colon$\\$ k=2,4,\ldots,2n-2\} \cup$ \\ $\{\sum_{i} \oii{i,i}-\sum_{i} \oii{i,i+1}\}\cup$ \\$\{\textrm{$n-1$ more functions}\}$} &\parbox{1.5in}{{\bf dimension}: 2 \\ {\bf basis}: \\ $\{\sum_{i} \Tout{i,i}\}\cup$\\ $\{\rvarrook{1}-\cdots \mp \rvarrook{n-1} \pm \frac{1}{2}\rvarrook{n}\}$ with $\rvarrook{i}$ as in~\eqref{eqn:b_red_var_rook} } & \parbox{1.25in}{{\bf dimension}: $1$ \\ {\bf basis}: \\ $\{\sum_{i} \oii{i,i}-\sum_{i} \oii{i,i+1}$ \\ $=\sum_{i} \Tout{i,i}\}$}  \\ \hline
\end{tabular}
\medskip
\end{adjustwidth}
\caption{Conjectural descriptions of the toggleability spaces for our posets of interest.} \label{tab:vect_spaces}
\end{table}

We emphasize that the descriptions of these various toggleability spaces are conjectural (found via computer experimentation for specific posets belonging to the infinite family in question). For an individual poset $P$, numerical linear algebra tells us when a given collection of functions spans the whole toggleability space; for an infinite collection of posets $P$, this method is unavailable.

Looking at \cref{tab:vect_spaces}, one notices that we have apparently already found all the functions in the toggleability spaces for most of these spaces. But there are some spaces where we are ``missing'' functions: for example, with $P=\sstair{n}$, there are $n-1$ additional, linearly independent statistics in $\{\stat{f}\tequiv \const\}  \cap V_A$ we do not yet know.\footnote{However, for \emph{one} more such statistic $\stat{f}$, see \cref{subsec:more_min} below.} We also remark that there are of course other nice bases for some of these spaces than the ones listed in~\cref{tab:vect_spaces}. For example, suppose $P=\rect{a}{b}$, and let $B\subseteq P$ be any subset of size $a+b-1$ which does not contain four boxes of the form $(i_1,j_1), (i_1,j_2), (i_2,j_1), (i_2,j_2)$; then another basis of $\{\stat{f} \tequiv \const\}  \cap V_I$ is $\{\rrook{i,j}\colon (i,j)\in B\}$, where $\rrook{i,j}$ is as in~\eqref{eqn:rect_red_rook}.

There are other natural ``vector space''-based approaches to homomesy. For instance, one could directly consider the space of rowmotion homomesies that belong to the span of some set of basic observable functions like $\{\Tout{p}\colon p \in P\}$ or $\{\oii{p}\colon p \in P\}$. Experimentation shows, however, that such spaces of homomesies can be somewhat unwieldy. The toggleability spaces appear nicer: for instance, they often possess bases whose elements all have nonnegative coefficients (when expressed in the observables), and sometimes even $\{0,1\}$-coefficients.

Another vector space approach to homomesy is proposed in~\cite{proppspectrum}. There, Propp considers the space spanned by the functions $\oii{p} \circ \rowm^i$ for $p \in P$ and $i\geq 0$. Propp's space is also useful for proving rowmotion homomesy results because showing that a statistic $\stat{f}$ can be written in the form $\stat{g}-\stat{g}\circ \rowm$ implies it is homomesic under rowmotion by a telescoping argument. Propp's space neither contains nor is contained in our toggleability spaces.\footnote{Leaving aside the fact that Propp extends his original vector space to make it closed under the action of rowmotion, Propp considers functions of the form $\stat{g}-\stat{g}\circ \rowm$ where $\stat{g}$ is spanned by the statistics $\oii{i,j}$. It is not hard to see that $\T{p}$ is not of this form. However, if $\stat{g} = -\Tout{p}$, then $\T{p} = \stat{g}-\stat{g}\circ \rowm$. 
So the present paper does show the importance of coboundaries -- just not the ones Propp looked at.} 
A defect of the framework of~\cite{proppspectrum} with respect to the present paper (from a practical if not a theoretical point of view)
is that, in the framework of~\cite{proppspectrum}, a statistic $\stat{f}$ that can be written in the form $\stat{g}-\stat{g}\circ \rowm$ can be so written in infinitely many ways. Thus, it is often unclear how to make these choices of $\stat{g}$ systematically for an infinite family of posets. On the other hand, \cref{thm:linearindependence} shows there are no choices when writing $\stat{f}$ as a combination of signed toggleability statistics.

One final word about~\cref{tab:vect_spaces}: based on this table, one might be tempted to wonder if we always have $\dim(\{\stat{f}\tequiv \const\} \cap V_A)=\dim(\{\stat{f}\tequiv \const\} \cap V_I)$ for all posets. But this is not true, and indeed there is in general no relation between these dimensions: we have $\dim(\{\stat{f}\tequiv \const\} \cap V_A) > \dim(\{\stat{f}\tequiv \const\} \cap V_I)$ for some posets, and $\dim(\{\stat{f}\tequiv \const\} \cap V_A) < \dim(\{\stat{f}\tequiv \const\} \cap V_I)$ for others.

\subsection{More with minuscule posets} \label{subsec:more_min}

Recall that a minuscule poset $P$ is determined by the choice of a Dynkin diagram $\Gamma$ and a minuscule node $i \in \Gamma$. In fact, the minuscule poset $P$ also comes with a surjection $\gamma\colon P\to \Gamma$, which we refer to as a \dfn{coloring} of the elements of $P$ by the nodes of the Dynkin diagram. See, e.g.,~\cite[\S3]{okada2020birational} for a description of all of the minuscule posets with their colorings.

Rush and Wang~\cite[Theorem 1.2]{rush2015homomesy} showed that for any minuscule poset $P$ and any $j \in \Gamma$, the function $I \mapsto \#(I\cap \gamma^{-1}(j))$ is homomesic under rowmotion. Observe how these statistics refine order ideal cardinality. 

We believe these color refinements of order ideal cardinality are also~$\tequiv \const$. For example, when $P=\rect{a}{b}$, which is a Type~$A_{a+b-1}$ minuscule poset, the color preimages $\gamma^{-1}(j)$ are exactly the files, and we have seen that the file refinements of order ideal cardinality are $\tequiv \const$. On the other hand, the shifted staircase $\sstair{n}$ can be realized as either a Type~$B_n$ minuscule poset or a Type~$D_{n+1}$ minuscule poset. If we view it as a Type~$B_n$ minuscule poset, the color preimages are exactly files, and again we have already seen that these are $\tequiv \const$. However, if we view it as a Type~$D_{n+1}$ minuscule poset, the main diagonal file breaks into two color preimages. In this way, for the shifted staircase, we get two statistics $\oii{1,1}+\oii{3,3}+\oii{5,5}+\cdots$ and $\oii{2,2}+\oii{4,4}+\oii{6,6}+\cdots$ that are both separately homomesic under rowmotion and should both be $\tequiv \const$.

It should be possible to show that these color refinements of order ideal cardinality are $\tequiv \const$ in a manner similar to all of our other arguments here. But we will not do that for the sake of concision.

Speaking of minuscule posets, let us also remark that in a series of papers~\cite{rush2013orbits,rush2015homomesy,rush2016cde}, Rush (with coauthors) has developed a systematic theory of toggling for minuscule posets, which has allowed him to prove results \emph{uniformly}, i.e., without recourse to the classification. It would be interesting to give uniform proofs of some of the new results concerning minuscule posets we obtained in this paper (e.g., that $\sum_{p\in P}\oii{p}\tequiv \const$ and that $\sum_{p\in P}\Tout{p}\qtequiv \const(q)$).

\subsection{Extension of the technique with antichain toggleability statistics} \label{subsec:antichain_extension}

In \cref{subsec:nonexamples}, we saw several non-examples, where the toggleability statistics technique failed to prove a known or conjectured rowmotion homomesy result. Let us explain how a certain extension of the technique might address many of these non-examples.

For $A\in \A(P)$, define the \dfn{antichain toggleability statistics} $\Tin{A}, \Tout{A}, \T{A}\colon \J(P)\to \RR$ by
\begin{align*}
\Tin{A}(I) &\coloneqq \begin{cases} 1 &\textrm{if $A\subseteq \min(P\setminus I)$}, \\
0 &\textrm{otherwise}; \end{cases}\\
\Tout{A}(I) &\coloneqq \begin{cases} 1 &\textrm{if $A\subseteq \max(I)$}, \\
0 &\textrm{otherwise}; \end{cases}\\
\T{A}(I) &\coloneqq \Tin{A}(I) - \Tout{A}(I),
\end{align*}
for all $I\in \J(P)$. The $\Tin{A}$ and $\Tout{A}$ measure whether we can toggle the entire antichain $A$ into or out of an order ideal. Note that, equivalently, we have
\[\Tin{A} = \prod_{p\in A} \Tin{p} \qquad \text{and}\qquad \Tout{A} = \prod_{p\in A} \Tout{p}\]
as pointwise products.

It is again easy to see that $A\subseteq \min(P\setminus I)$ if and only if $A\subseteq \max(\rowm(I))$ for all $A \in \A(P)$ and $I \in \J(P)$, and hence that Striker's observation persists to these antichain toggleability statistics:

\begin{lemma} \label{lem:anti_striker}
For every $A\in \A(P)$, $\T{A}$ is $0$-mesic under rowmotion.
\end{lemma}

In fact, when working with antichain toggleability statistics, we actually have a converse to Striker's observation:

\begin{thm} \label{thm:anti_span}
The space of functions $\J(P)\to \RR$ that are $0$-mesic under rowmotion is equal to $\Span\{\T{A}\colon A \in \A(P)\}$.
\end{thm}

\begin{proof}
Let $H_0$ be the space of functions that are $0$-mesic under rowmotion. We know by \cref{lem:anti_striker} that $\Span\{\T{A}\colon A \in \A(P)\}\subseteq H_0$, so it suffices to prove that $\dim\Span\{\T{A}\colon A \in \A(P)\}$ is greater than or equal to $\dim H_0$. Let $O_1,\ldots, O_m$ be the orbits of rowmotion on $\J(P)$. By definition, $H_0$ is the space of functions $\stat{f}\colon \J(P)\to\RR$ satisfying the linear equations $\sum_{x\in O_j}\stat{f}(x)=0$ for $1\leq j\leq m$. Hence, $\dim H_0=\#\J(P)-m=\#\A(P)-m$. Let $\A(P)^{\RR}$ be the set of real-valued functions on $\A(P)$, and let $Q\colon\A(P)^{\RR}\to \Span\{\T{A}\colon A \in \A(P)\}$ be the linear map defined by $Q(\pi)=\sum_{A\in\A(P)}\pi(A)\T{A}$. We have $\dim\Span\{\T{A}\colon A \in \A(P)\} = \dim\A(P)^{\RR}-\dim\ker Q=\#\A(P)-\dim\ker Q$. We will show that $\dim\ker Q\leq m$, which will complete the proof. 

Suppose $\pi\in\ker Q$. Recall the definition of antichain rowmotion $\pan\colon \A(P)\to\A(P)$ from the proof of \cref{thm:linearindependence}. Let $A\in\A(P)$, and let $I\in\J(P)$ be the order ideal generated by $A$ (so $\max(I)=A$). Then $0=Q(\pi)(I)=\sum_{A'\in\A(P)}\pi(A')\T{A'}(I)$. Because $A'\subseteq \min(P\setminus I)$ if and only if~$A'\subseteq \max(\rowm(I))$, this equation is equivalent to $\sum_{A'\subseteq A} \pi(A')=\sum_{A'\subseteq \pan(A)} \pi(A')$. 

Consider the linear transformation $L\colon\A(P)^{\RR}\to \A(P)^{\RR}$ given by $L(\pi)(A)=\sum_{A'\subseteq A}\pi(A)$. The map $L$ is easily seen to be invertible: for example, it is represented by a lower-triangular matrix if we order the natural basis of $\A(P)^{\RR}$ in a way that respects inclusion of antichains. Moreover, the previous paragraph tells us that $\pi\in\ker Q$ if and only if $L(\pi)$ is invariant under antichain rowmotion. This shows that the dimension of $\ker Q$ is at most the dimension of the space of functions $\A(P)\to\RR$ that are invariant under antichain rowmotion. But this latter space has dimension equal to the number of orbits of antichain rowmotion, which is $m$. 
\end{proof}

So if we have a function $\stat{f}\colon\J(P)\to\RR$ that we think is homomesic under rowmotion, then one way we could prove it is by writing $\stat{f}=c+\sum_{A\in \A(P)} c_A\T{A}$ for $c, c_A \in \RR$. Moreover, \cref{thm:anti_span} says this approach will ``always work.'' However, in practice, it is much harder to write a function as a linear combination of the $\T{A}$ than as a linear combination of the $\T{p}$ because there are many more coefficients $c_A$ to deal with. Indeed, these coefficients are not even uniquely determined: there are many linear dependencies among the $\T{A}$, which is not the case with the $\T{p}$, as we know from \cref{thm:linearindependence}.

Thus, even though this extension (using antichain toggleability statistics) of our technique for proving rowmotion homomesy ``always works,'' it is not straightforward to apply. Furthermore, we do not know the proper way to think about the antichain toggleability statistics at the PL and birational levels, so the approach sketched in this subsection does not say anything about PL or birational homomesies.

\subsection{More about \texorpdfstring{$q$}{q}-rowmotion}

We think $q$-rowmotion deserves to be studied further. For example, can anything precise be said about the order or orbit structure of $q$-rowmotion? (As mentioned, weak statements about orbit sizes can sometimes be extracted from homomesy results, like the ones we proved for $q$-rowmotion in this paper.)  Can we make sense of $q$-rowmotion with $q=-1$?

Also, in \cref{sec:pl_birational,sec:q}, we have seen two different kinds of generalizations of rowmotion: piecewise-linear and birational rowmotion; and $q$-rowmotion. Is it possible to join these two different generalizations? Let us be a little more specific. For fixed $m\geq 1$, order-preserving maps (or ``labelings'') $P\to \{0,1,\ldots,m\}$ are in bijection with the rational points in the order polytope $\mathcal{O}(P)$ with denominator dividing $m$, and for this reason carry an action of piecewise-linear rowmotion. Could we consider such labelings with several different flavors of $0$'s, $1$'s, ..., $m$'s and define a $q$-analogue of rowmotion on these flavorful labelings? Does this $q$-PL rowmotion have any nice properties?

\bibliography{bibliography}{}
\bibliographystyle{abbrv}

\end{document}